\documentclass[a4paper,11pt]{amsart}

\usepackage{graphicx}
\usepackage{amsmath}
\usepackage{amsmath,amscd}
\usepackage{amssymb}
\usepackage{enumerate}
\usepackage[active]{srcltx}
\usepackage{hyperref}

\newtheorem{thm}{Theorem}
\newtheorem{lem}[thm]{Lemma}
\newtheorem{prop}[thm]{Proposition}
\newtheorem{cor}[thm]{Corollary}

\theoremstyle{definition}
\newtheorem{defn}{Definition}[section]

\theoremstyle{remark}
\newtheorem{remark}{Remark} 

\theoremstyle{plain}

\def\CC{{\mathbb C}}

\def\NN{{\mathbb N}}
\def\QQ{{\mathbb Q}}

\def\RR{{\mathbb R}}

\def\ZZ{{\mathbb Z}}

\def\vecu{{\text{\boldmath$u$}}}
\def\vecv{{\text{\boldmath$v$}}}

\def\vecw{{\text{\boldmath$w$}}}
\def\vecx{{\text{\boldmath$x$}}}

\def\scrA{{\mathcal A}}
\def\scrB{{\mathcal B}}
\def\scrC{{\mathcal C}}
\def\scrD{{\mathcal D}}

\def\scrF{{\mathcal F}}
\def\scrH{{\mathcal H}}

\def\scrI{{\mathcal I}}

\def\scrK{{\mathcal K}}

\def\scrP{{\mathcal P}}

\def\scrS{{\mathcal S}}

\def\scrU{{\mathcal U}}

\def\scrW{{\mathcal W}}

\def\scrY{{\mathcal Y}}
\def\scrZ{{\mathcal Z}}

\def\fB{{\mathfrak B}}

\def\fG{{\mathfrak G}}

\def\Re{\operatorname{Re}}

\def\dim{\operatorname{dim}}

\def\SL{\operatorname{SL}}

\def\SO{\operatorname{SO}}

\def\tr{\operatorname{tr}}

\def\supp{\operatorname{supp}}

\def\Proj{\operatorname{Proj}}

\def\Onder#1#2#3#4#5{#1 \setbox0=\hbox{$#1$}\setbox1=\hbox{$#2$}
       \dimen0=.5\wd0 \dimen1=\dimen0 \dimen2=\dp0 \dimen3=\dimen2
       \advance\dimen0 by .5\wd1 \advance\dimen0 by -#4
       \advance\dimen1 by -.5\wd1 \advance\dimen1 by -#4
       \advance\dimen2 by -#3 \advance\dimen2 by \ht1
       \advance\dimen2 by 0.3ex \advance\dimen3 by #5
        \kern-\dimen0\raisebox{-\dimen2}[0ex][\dimen3]{\box1}
       \kern\dimen1}

\newcommand{\GaG}{\Gamma\backslash G}

\newcommand{\sfrac}[2]{{\textstyle \frac {#1}{#2}}}

\newcommand{\fg}{\mathfrak{g}}
\newcommand{\fp}{\mathfrak{p}}
\newcommand{\fu}{\mathfrak{u}}
\newcommand{\fa}{\mathfrak{a}}
\newcommand{\fh}{\mathfrak{h}}
\newcommand{\fm}{\mathfrak{m}}
\newcommand{\fn}{\mathfrak{n}}
\newcommand{\fk}{\mathfrak{k}}

\newcommand{\Ug}{\mathcal{U}(\mathfrak{g}_{\mathbb{C}})}
\newcommand{\Uh}{\mathcal{U}(\mathfrak{h}_{\mathbb{C}})}
\newcommand{\Ugm}{\mathcal{U}^m(\mathfrak{g}_{\mathbb{C}})}
\newcommand{\Zg}{\mathcal{Z}(\mathfrak{g}_{\mathbb{C}})}

\newcommand{\ad}{\mathrm{ad}}
\newcommand{\Ad}{\mathrm{Ad}}
\newcommand{\Stab}{\mathrm{Stab}}
\newcommand{\ft}{\mathfrak{t}}

\newcommand{\fq}{\mathfrak{q}}

\newcommand{\rank}{\mathrm{rank} }
\newcommand{\disG}[2]{\mathsf{dist}_G(#1,#2) }
\newcommand{\disGa}[2]{\mathsf{dist}_{\Gamma\backslash G}(#1,#2) }
\newcommand{\bk}{\mathbf{k}}
\newcommand{\bj}{\mathbf{j}}
\newcommand{\bi}{\mathbf{i}}
\newcommand{\bl}{\mathbf{l}}
\newcommand{\bm}{\mathbf{m}}
\newcommand{\bI}{\mathbf{I}}

\begin{document}
\title[The Rate of Equidistribution for Translates of Horospheres]{On the Rate of Equidistribution of Expanding Translates of Horospheres in $\Gamma\backslash G$}

\author{Samuel C. Edwards}

\address{Department of Mathematics, Box 480, Uppsala University, SE-75106 Uppsala, Sweden}
\email{samuel.edwards@math.uu.se}


\date{\today}

\subjclass[2000]{}


\begin{abstract}
Let $G$ be a semisimple Lie group and $\Gamma$ a lattice in $G$. We generalize a method of Burger to prove precise effective equidistribution results for translates of pieces of horospheres in the homogeneous space $\Gamma\backslash G$.
\end{abstract}

\maketitle

\section{Introduction}\label{Intro}
\subsection{Background}
Let $G$ be a connected Lie group and $\Gamma$ a lattice in $G$. Several outstanding problems, particularly in the field of number theory, have been resolved thanks to breakthroughs in the study of orbits related to various subgroups of $G$ on the homogeneous space $\GaG$. Amongst the most well-understood types of orbits on $\GaG$ are those that involve so-called horospherical subgroups of $G$. Equidistribution properties of the orbits of such subgroups play a role in the proofs of several more difficult results in the field, as well as in a number of applications in areas such as number theory, mathematical physics, and geometry; for a few examples and further references, cf.\ Kleinbock and Margulis \cite{KleinMarg}, Marklof \cite{Marklof1, Marklof2, Marklof3}, Marklof and Str\"ombergsson \cite{MarklofStrom}, and Mohammadi and Oh \cite{MohammadiOh}.

Recall that for a one-parameter subgroup $\lbrace g_t \rbrace_{t\in\RR}$ (henceforth denoted $g_{\RR}$) of $G$, the expanding horospherical subgroup $U^+$ with respect to $g_{\RR}$ is defined as
\begin{equation*}
U^+=\lbrace h\in G\,:\, \lim_{ t\rightarrow+\infty} g_{-t} h g_{t} =e \rbrace,
\end{equation*}
and the contracting horospherical subgroup $U^-$ with respect to $g_{\RR}$ as
\begin{equation*}
U^-=\lbrace h\in G\,:\, \lim_{t\rightarrow-\infty} g_{-t} h g_{t} =e \rbrace.
\end{equation*}
This article is a continuation of work started in \cite{Edwards};  attempting to obtain precise rates of effective equidistribution for expanding translates of pieces of horospherical orbits.

Under certain conditions, translates of horospheres equidistribute in $\GaG$. Assume that $g_{\RR}$ is $\RR$-diagonalizable, i.e.\ each operator $\Ad_{g_t}$ (on the Lie algebra $\fg$ of $G$) is diagonalizable over $\RR$. Let $\mu$ denote the unique $G$-invariant Borel probability measure on $\GaG$. Then for any bounded continuous function $f$ on $\GaG$ and any probability measure $\lambda$ on $U^+$ which is absolutely continuous with respect to a Haar measure $du$ on $U^+$, we have (cf.,\ e.g.,\ \cite[Proposition 2.2.1]{KleinMarg})
\begin{equation}\label{EQUIDIST}
\lim_{t\rightarrow-\infty} \int_{U^+} f(xug_t)\,d\lambda(u)=\int_{\GaG}f\,d\mu\qquad \forall x\in \GaG.
\end{equation}
One of the most common ways of proving equidistribution statements similar to \eqref{EQUIDIST} is by making use of the ubiquitous ``Margulis' thickening technique", the behind-lying idea of which originates in the thesis of Margulis, cf.\ \cite{Margulis}. By imposing certain restrictions on $g_{\RR}$, $f$, and $\lambda$, \eqref{EQUIDIST} can be made effective: from \cite[Proposition 2.4.8]{KleinMarg}, if the action of $g_{\RR}$ is exponentially mixing, we have that for $f\in C^{\infty}_c(\GaG)$ and $\chi\in C_c^{\infty}(U^+)$,
\begin{equation*}
\left| \int_{U^+}\chi(u) f(xug_t)\,du-\int_{U^+} \chi\,du\int_{\GaG}f\,d\mu \right|\ll_{G,f,\chi,x} e^{\delta t} \qquad \forall t\leq 0,
\end{equation*}
for some explicit $\delta>0$. In this article, we consider the problem of giving a precise bound for this difference, with particular focus on obtaining as great a rate of exponential decay as possible. Our results follow in a line of work by Hejhal \cite{Hejhal}, Str\"ombergsson \cite{Strom1}, and S\"odergren \cite{Sod} studying this type of question on  hyperbolic manifolds, i.e.\ establishing corresponding equidistribution statements for spaces $\Gamma\backslash \SO(n,1)/\SO(n)$. The results mentioned above make extensive use of the spectral theory of automorphic forms on hyperbolic space, and consider only the equidistribution of translates of pieces of \emph{closed} horospheres. The connection between the spectral theory of automorphic forms, in particular Eisenstein series, and the equidistribution of closed horospheres goes back to Selberg (unpublished work), Zagier \cite{Zagier}, and Sarnak \cite{Sarnak}, who studied the rate of equidistribution for translates of entire closed horocycles in the case $G=\SL(2,\RR)$.

In \cite{Strom}, Str\"ombergsson, while generalizing results of Burger \cite{Burger90} on precise equidistribution statements for horospherical orbits on $\GaG$ in the case $G=\SL(2,\RR)$, observed that the method used there can be used to strengthen the results of \cite{Strom1} regarding the equidistribution of translates of pieces of horospheres (cf.\ \cite[Remark 3.4]{Strom}). We note that Flaminio and Forni \cite{FlamForn} also proved precise equidistribution results in the case $G=\SL(2,\RR)$, cf.\ \cite[Theorems 1.7 and 5.14]{FlamForn}. The methods of \cite{FlamForn} are somewhat different to those of \cite{Burger90,Strom}.

The main tool of \cite{Burger90,Strom} is a representation-theoretic method first developed by Burger in \cite{Burger90}. At the heart of this method is an identity, \cite[Lemma 1]{Burger90}, associating the action of $U^+$ with that of $g_{\RR}$ in an arbitrary irreducible unitary representation. In \cite{Edwards}, we used a similar identity in the case $G=\SL(2,\CC)$ to prove precise effective equidistribution results similar to those of \cite{Hejhal, Sod, Strom1}, though for translates of pieces of \emph{all} horospherical orbits; i.e.\ not just the closed ones. It is this method that we now develop for more general groups $G$.  

\subsection{Main Results} From now on we let $G$ be a connected semisimple Lie group with finite center, and $\Gamma$ a lattice in $G$ satisfying the assumptions of Langlands (cf.\ \cite[Chapter 2]{Langlands}, \cite[Chapter 2]{OsborneWarner}). The one-parameter subgroup $g_{\RR}$ is assumed to be $\RR$-diagonalizable throughout. We also fix an element $Y\in\fg$ such that
\begin{equation*}
g_t=\exp(tY)\qquad\forall t\in \RR.
\end{equation*}
These assumptions ensure that there exists a parabolic subgroup $P$ with Langlands decomposition $P=NAM$ (cf.\ \cite[Chapter VII.4]{Knapp2}) such that $g_{\RR}\subset A$ and $U^+=N$; we will review the structure theory of $G$ and its parabolic subgroups in detail in Section \ref{Struc}.

As previously mentioned, our method of proof is heavily representation-theoretic. The previous uses of this method \cite{Burger90, Edwards, Strom} are all restricted to the cases $G=\SL(2,\RR)$ and $\SL(2,\CC)$, and make use of the classification of the unitary dual of $G$, as well as the decomposition of the right-regular representation of $G$ on $L^2(\GaG)$, denoted $\rho$, into irreducible unitary representations. In these cases, this information allows one to relate the rate of equidistribution with the spectrum of the Laplace operator on the hyperbolic orbifold $\GaG/K$, where $K=\SO(2)$ or $K=\mathrm{SU}(2)$. For more general groups $G$, a complete classification of the unitary dual is not available to us. Furthermore, the decomposition of $(\rho,L^2(\GaG))$ is somewhat more complicated, cf.\ \cite{Langlands}. For these reasons, we instead compare the equidistribution of the relevant translates with the decay of the matrix coefficients of $g_{\RR}$. The main result of this article is that translates of pieces of horospheres equidistribute with the same exponential rate as the matrix coefficients of $g_{\RR}$ decay. As far as we are aware, all previous results on the effective equidistribution of these types of translates either give a worse  rate, cf.,\ eg.,\ \cite[Proposition 2.4.8]{KleinMarg}, or are restricted to groups $G$ of real rank one (and focus mainly on pieces of closed horospheres), cf.\ \cite{Burger90,Edwards,FlamForn,Hejhal,Sarnak,Sod,Strom1,Strom,Zagier}.

In order to state our main result, we fix a maximal compact subgroup $K$ of $G$, and make the following definition:
\begin{defn}\label{RATEDEF}
A number $\eta\in \RR_{>0}$ is said to be a \emph{rate of decay} for the matrix coefficients of  $g_{\RR}$ in a unitary representation $(\pi,\scrH)$ of $G$ if there exist $C,\,q\geq 0$ such that for all $K$-finite vectors $\vecu,\vecv\in\scrH$ and $t\in\RR$,
\begin{equation}\label{DECAYDEF}
|\langle \pi(g_t)\vecu,\vecv\rangle|\leq C \|\vecu\|_{\scrH}\|\vecv\|_{\scrH} \left(\dim(\pi(K)\vecu)\dim(\pi(K)\vecv)\right)^{1/2}(1+|t|^q)e^{-\eta|t|}.
\end{equation}
\end{defn}
The decay of matrix coefficients is one of the most of studied aspects of the unitary representations of $G$; cf.,\ eg.,\ \cite{Casselman, Cowling, Howe, Oh}. We let $L^2(\GaG)_0$ denote the orthogonal complement of the one-dimensional subspace of $L^2(\GaG)$ consisting of the constant functions. Rates of decay of matrix coefficients for $(\rho,L^2(\GaG)_0)$ have deep connections to many important problems in the study of lattices in Lie groups.

We must also introduce Sobolev norms and spaces; these will be important for quantifying the regularity of the function $f$ and the measure $\lambda$ in \eqref{EQUIDIST}. For an open subset $B$ of $U^+$, and a choice of basis for the Lie algebra of $U^+$, we define a Sobolev norm $\|\cdot\|_{W^{k,p}(B)}$ on $C_c^{\infty}(B)$ by letting $\|\chi\|_{W^{k,p}(B)}$ denote the sums of the $L^p$ norms of all the Lie derivatives of $\chi$ corresponding to monomials in our chosen basis with order not greater than $k$. The closure of $C_c^{\infty}(B)$ with respect to the norm $\|\cdot\|_{W^{k,p}(B)}$ is denoted $W^{k,p}_0(B)$. 

In an analogous way, we choose a basis for the Lie algebra of $G$, and use this to define Sobolev norms $\|\cdot\|_{\scrS^m(\GaG)}$ (this is done in greater detail in Section \ref{Reptheorysec}): for a smooth function $f$ on $\GaG$ such that $f$ and all its derivatives are in $L^2(\GaG)$, let $\|f\|_{\scrS^m(\GaG)}^2$ denote the sum of the squares of the $L^2$ norms of $f$ and all its Lie derivatives corresponding to monomials in our chosen basis with order not greater than $m$. The closure of the space of functions $f\in C^{\infty}(\GaG)\cap L^2(\GaG)$ such that $\|f\|_{\scrS^m(\GaG)}<\infty$ with respect to $\|\cdot\|_{\scrS^m(\GaG)}$ is denoted $\scrS^m(\GaG)$. 

Related to Sobolev norms is the invariant height function $\scrY_{\Gamma}$. This function, in conjunction with Sobolev norms, is used to give pointwise bounds for functions in $\scrS^m(\GaG)$. If $\Gamma$ is not cocompact, $\scrY_{\Gamma}(x)$ provides a measure of how far into a cusp the point $x\in\GaG$ is. A stringent definition of $\scrY_{\Gamma}$ is given in Section \ref{SOBINEQSEC}.

Our main result can now be stated:

\begin{thm}\label{maintheorem1}
Let $(\rho,V)$ be a subrepresentation of $(\rho,L^2(\GaG)_0)$, and $\eta$ a rate of decay for the matrix coefficients of $g_{\RR}$ in $(\rho,V)$. For any open, relatively compact subset $B$ of $U^+$ there exist $C=C(G,\Gamma,g_{\RR},B,\eta)>0$, $W>0$, $m_1,m_2\in\NN$ such that 
\begin{equation}\label{mainthmbd}
\left| \int_{U^+} \chi(u)f(xug_{t}) \,du\right|\leq C \|\chi\|_{W^{m_2,\infty}(B)}\|f\|_{\scrS^{m_1}(\GaG)}\scrY_{\Gamma}(x)^2(1+|t|^W)e^{\eta t}
\end{equation}
for all $t \leq 0$, $f\in V\cap\scrS^{m_1}(\GaG)$, $\chi\in W^{m_2,\infty}_0(B)$, and $x\in\GaG$.
\end{thm}
\begin{remark}
While it is possible to keep track of, and explicitly state, the dependency of the implied constants on $\eta$ and $g_{\RR}$, it is somewhat tedious to do so. We simply note that the constants are inversely proportional to the quotient of $\eta$ and the ``distance" of $Y$ to the walls of the Weyl chamber in which it lies. It is also possible to make the dependency on $B$ completely explicit; it is related to the ``covering number" of $B$ with respect to some fixed neighbourhood of the identity in $G$.
\end{remark}
\begin{remark}
If we restrict to the case of translates of entire closed horospheres, i.e.\ when $\Gamma\cap U^+$ is a lattice in $U^+$, $x=\Gamma e$, and $\chi$ is suitably chosen so that the integral in the left-hand side of \eqref{mainthmbd} may be written as $\int_{\Gamma\cap U^+\backslash U^+} f(\Gamma u g_t)\,du$, simplifications occur in the proof of Theorem \ref{maintheorem1} that produce a result similar in nature to \cite[Theorem 1]{Sarnak}. Also, if $G$ is algebraic and $\Gamma$ is arithmetic, the proof of Theorem \ref{maintheorem1} may once again be adjusted so as to permit $\chi\in W^{m_2,2}_0(B)$, with a corresponding change of norm in the right-hand side of \eqref{mainthmbd}. Both these points are discussed further in Remarks \ref{ClosedHoros} and \ref{L2remark} below.
\end{remark}

As previously stated, the main focus of this article has been the rate of exponential decay in Theorem \ref{maintheorem1}. For this reason, we have assumed that $f$ and $\chi$ are as smooth as needed to obtain this rate. Using Proposition \ref{BurgerProp2} below, one sees that our proof may be slightly adjusted so as to also give an exponential rate of decay for corresponding averages when the test functions $\chi$ are restricted to $W_0^{k_0,\infty}(B)$ for any $k_0\geq 1$. The rate obtained for such $\chi$ may, however, be much smaller than the rate of decay of the matrix coefficients of $g_{\RR}$. Furthermore, it is unclear whether this rate is greater than that given by using approximations of $\chi\in W^{k_0,\infty}_0(B)$ by elements of $W^{m_2,\infty}_0(B)$ together with Theorem \ref{maintheorem1} ($m_2$ being as in Theorem \ref{maintheorem1}). We also note that the results of \cite{Burger90, Edwards, Strom} give variations of Theorem \ref{maintheorem1} (for $G=\SL(2,\RR)$ or $\SL(2,\CC)$) in the case that $\chi$ is an indicator function of a subset $B\subset U^+$. It seems possible that our method of proof may also be used to directly (i.e.\ without having to use approximations by Sobolev functions) give effective equidistribution statements for such $\chi$ for general groups $G$ as well, though not without complications; compare Corollary \ref{scrYAVG} below with \cite[Proposition 6]{Edwards}. The rate one obtains in this manner, however, seems to be bounded by $\min_{\alpha\in \Sigma^+(\fg,\fa)} \alpha(Y)$, where $\Sigma^+(\fg,\fa)$ is the set of positive restricted roots corresponding to the Langlands decomposition $P=U^+AM$. For the cases $G=\SL(2,\RR)$ or $ \SL(2,\CC)$, this value is greater than or equal to the rate of decay for the matrix coefficients of $g_{\RR}$ in any non-tempered unitary representation of $G$ (this is due to the fact that in these cases $G$ has real rank one and the dimension of the restricted root spaces is less than or equal to two). This is not the case though for more general groups $G$, and at present we do not see how one might prove a result with a rate of exponential decay similar to that of Theorem \ref{maintheorem1} for indicator functions.

Shifting perspective slightly, we now instead fix a horospherical subgroup and consider the equidistribution of translates of a piece of a horospherical orbit by elements from an entire positive (closed) Weyl chamber. In order to state our second theorem, we fix a parabolic subgroup $P$ of $G$ with Langlands decomposition $P=NAM$. Recall that the subgroup $N$ is the unipotent radical of $P$, and $A$ is simply connected and abelian. We define the following subset of $A$:
\begin{equation*}
A^+=\lbrace a\in A\,:\, \lim_{l\rightarrow \infty} a^{-l} n a^{l}=e\qquad \forall n \in N\rbrace,
\end{equation*}
and denote the topological closure of $A^+$ by $\overline{A^+}$. The Lie algebra of $A$ is denoted by $\fa$. Since $A$ is simply connected, the exponential map $\exp:\fa\rightarrow A$ is a diffeomorphism (with inverse $\log:A\rightarrow \fa$), and we let $\fa^+=\log(A^+)$, hence $\overline{\fa^+}=\log(\overline{A^+})$. Note that $N$ is the expanding horospherical subgroup with respect to any one-parameter subgroup $\exp(\RR H)$, where $H\in\fa^+$. Moreover, every horospherical subgroup of $G$ is the unipotent radical of some parabolic subgroup. 

\begin{defn}\label{UNIRATEDEF}
A norm $\|\cdot\|$ on $\fa$ is said to \emph{control the decay} of the matrix coefficients of $\overline{A^+}$ in a unitary representation $(\pi,\scrH)$ of $G$ if there exist $C,\,q\geq 0$ such that for all $K$-finite vectors $\vecu,\vecv\in\scrH$ and all $H\in\overline{\fa^+}$,
\begin{equation}\label{UNIFORMMATRIXCOEFFBD}
|\langle \pi(\exp(H))\vecu,\vecv\rangle|\leq C \|\vecu\|\|\vecv\| \left(\dim(\pi(K)\vecu)\dim(\pi(K)\vecv)\right)^{1/2}(1+\|H\|^q)e^{-\|H\|}.
\end{equation}
\end{defn}

The following theorem shows that the translates of pieces of $N$ equidistribute with the same speed as the matrix coefficients decay, essentially uniformly over all $\overline{A^+}$:
\begin{thm}\label{maintheorem2}
Let $(\rho,V)$ be a subrepresentation of $(\rho,L^2(\GaG)_0)$, and $\|\cdot\|$ a norm on $\fa$ that controls the decay of matrix coefficients of $\overline{A^+}$ in $(\rho,V)$. There then exist constants $c$ (depending only on $\fa$ and $\|\cdot\|$) and $W$ (depending only on $G$ and $A$) such that for any relatively compact subset $B\subset N$ and $0<\epsilon<\min\lbrace 1,(2c)^{-1}\rbrace$, there exists $C=C(G,\Gamma,A,B,\|\cdot\|,\epsilon)$ such that if $m_1=W(\lfloor\frac{2}{\epsilon}\rfloor+3)+1+\lceil\frac{\dim G}{2}\rceil $ and $m_2=\lfloor \frac{2}{\epsilon}\rfloor$, we have
\begin{equation}\label{thm2bdd}
\left| \int_{N} \chi(n)f\big(xn\exp(-H)\big) \,dn\right|\!\leq \!C \|\chi\|_{W^{m_2,\infty}(B)}\|f\|_{\scrS^{m_1}(\GaG)}\scrY_{\Gamma}(x)^2(1+\|H\|^{W+1})e^{-(1-c\epsilon)\|H\|}
\end{equation}
for all $H\in \overline{\fa^+}$, $f\in V\cap\scrS^{m_{1}}(\GaG)$, $\chi\in W^{m_{2},\infty}_0(B)$, and $x\in\GaG$.
\end{thm}
\begin{remark}
As was the case for Theorem \ref{maintheorem1}, in order to give a (relatively) simple statement, we have refrained from stating (and keeping track of) the dependency of the constant $C$ on the various parameters. From the proofs, one sees that it is possible to make many of these dependencies completely explicit. It can also be seen that the constant $W$ is bounded over all choices of $A$ by the order of the Weyl group of $G$. Furthermore, given a set $\scrK\subset\overline{\fa^+}$ whose intersection with $\fa^+$ and the relative interior of any face of $\overline{\fa^+}$ is compact, there exists $\epsilon_{\scrK}>0$ such that if $\epsilon\leq\epsilon_{\scrK}$, \eqref{thm2bdd} holds with $e^{-\|H\|}$ in place of $e^{-(1-c\epsilon)\|H\|}$ for all $H\in\RR\scrK$. Finally, it can be shown that for a given norm $\|\cdot\|$ and compact subset $\scrC\subset \RR^+$, the dependency of the constants $C$ and $c$ with respect to the choice of norm is uniformly bounded over the set $\lbrace \eta\|\cdot\|\,:\, \eta\in\scrC\rbrace$.
\end{remark}

\begin{remark}
Recently, the equidistribution of translates of pieces of $N$ by elements of $G$ that are not necessarily in $\overline{A^+}$ has been studied by a number of authors, cf.\ \cite{DabbsKellyLi, KleinMarg3, KleinShiWeiss, KleinWeiss, MohammadiGolsefidy, ShahWeiss, Shi}. Another line of recent study has been the equidistribution of translates of horospheres in the case where $G$ is of rank one and $\Gamma$ a geometrically finite group, but not necessarily a lattice, cf.,\ eg.,\ \cite{Kim,KontOh,LeeOh,MohammadiOh}. These equidistribution results have applications to a number of problems in number theory. It would be interesting to see if the method developed here can be adapted to these settings.
\end{remark}

\subsection{Outline of Article}
Since our line of proof is quite similar to that of \cite{Edwards}, the general structure of this article is more or less the same. After reviewing the necessary background material in Section \ref{Prelims}, in Section \ref{BURGERSEC} we turn our attention to proving the main technical result required for the proof of Theorem \ref{maintheorem1}: a generalization of ``Burger's formula" \cite[Lemma 1]{Burger90} for general semisimple Lie groups. The first order of business is to establish a Lie algebra identity, Lemma \ref{LIEIDENT}, which will enable us to express an ordinary differential equation (with respect to $t$) for averaging operators, in irreducible unitary representations $(\pi,\scrH)$, of the form $\vecv\mapsto\int_{U^+} \chi(u)\pi(ug_t)\vecv\,du$ (where $\vecv\in\scrH$). We then use this differential equation to give an integral representation of the averaging operators (see Proposition \ref{BurgerProp1}); it is this integral representation that corresponds to \cite[Lemma 1]{Burger90} and \cite[Proposition 4]{Edwards}. Unlike in \cite{Burger90, Edwards}, this integral representation is not sufficient for our purposes, and we must instead make use of the differentiability of $\chi$ and an iterative application of Proposition \ref{BurgerProp1} to obtain an integral identity which suits our needs. This is done in Proposition \ref{BurgerProp2}.

We then proceed to prove the required properties of the invariant height function $\scrY_{\Gamma}$. The main tool here is the reduction theory of lattices in semisimple Lie groups; using this, we prove that $\scrY_{\Gamma}\in L^1(\GaG)$ (Proposition \ref{Ybound} and Corollary \ref{scrYL1cor}). This fact, together with a ``thickening" argument allows us to give a bound $\int_{B} \scrY_{\Gamma}(xug_t)\,du\ll_{\Gamma,B} \scrY_{\Gamma}(x)^2$ for all $x\in \GaG$, see Corollary \ref{scrYAVG} below. Matters here are considerably less complicated than when proving the corresponding results needed when considering ``sharp cut-off functions" $\chi$ in the case $G=\SL(2,\CC)$ (cf.\ \cite[Proposition 6]{Edwards}).

Section \ref{thm1proof} is devoted to the proof of Theorem \ref{maintheorem1}. Having established the necessary prerequisites, the proof is fairly straightforward: the representation $(\rho,V)$ is decomposed into irreducible unitary representations $(\pi,\scrH)$ and the identity \eqref{MAINBURGERIDENT} of Proposition \ref{BurgerProp2} is used in each of these. The bounds provided in Proposition \ref{BurgerProp2} then suffice to exchange the direct integral decomposition with the integrals occurring in \eqref{MAINBURGERIDENT}. These bounds are then used once again, together with Corollary \ref{scrYAVG}, to give the bound stated in Theorem \ref{maintheorem1}.

Finally, in Section \ref{UNIFORMSEC}, we prove Theorem \ref{maintheorem2}. Firstly, for each $\epsilon>0$ we construct certain compact subsets of $\overline{\fa^+}$ on which we will prove that Theorem \ref{maintheorem1} holds uniformly. We then review the results from Sections \ref{BURGERSEC} and \ref{SOBSEC}, making relatively minor adjustments in order to obtain statements which are uniform in a new parameter $H\in\overline{\fa^+}$. The proof of Theorem \ref{maintheorem2} is based on the observation that every $H\in\overline{\fa^+}$ such that $\|H\|>0$ may be decomposed as $H=H_{\epsilon}+J_{\epsilon}$, with $\frac{H_{\epsilon}}{\|H_{\epsilon}\| }$ lying in one of the compact subsets constructed at the beginning of the section. Applying Theorem \ref{maintheorem1} with respect to $Y=\frac{H_{\epsilon}}{\|H_{\epsilon}\| }$ and showing that the contribution from $J_{\epsilon}$ is negligible concludes the proof.
\subsection{Acknowledgements}
The research leading to these results was funded by Swedish Research Council Grant 621-2011-3629. I am grateful to my advisor Andreas Str\"ombergsson for suggesting this problem, as well as for many inspiring discussions and helpful suggestions during this work. I would also like to thank Volodymyr Mazorchuk for an interesting discussion regarding Lemma \ref{LIEIDENT}.
\section{Preliminaries}\label{Prelims}
\subsection{Structure Theory}\label{Struc}
We start by reviewing the structure theory of (real) semisimple Lie groups and their parabolic subgroups; our source for this is \cite[Chapters 6 and 7]{Knapp2}, and we summarize some of the results there that we will require.

Recall that we assume throughout that $G$ is a connected Lie group with finite center whose Lie algebra $\fg$ is semisimple. We have also fixed a maximal compact subgroup $K$ of $G$, the Lie algebra of which is denoted $\fk$. Let $\theta$ be a Cartan involution of $\fg$ such that $\fk$ is the $-1$-eigenspace for $\theta$. Denoting the $+1$-eigenspace of $\theta$ by $\fp$, we have the corresponding Cartan decomposition $\fg=\fp\oplus\fk$; recall that the direct sum is orthogonal with respect to the Killing form $B=B(\cdot,\cdot)$ of $\fg$. We now choose and fix once and for all a maximal abelian subalgebra $\fa_0$ of $\fp$. Note that any other choice of maximal abelian subalgebra of $\fp$ is of the form $\Ad_k\fa_0$ for some $k\in K$. Let $\Sigma(\fg,\fa_0)$ denote the restricted roots of $\fa_0$; this is the set of all $\lambda\in\fa_0^*\setminus\lbrace 0\rbrace$ such that the corresponding root space $\fg_{\lambda}=\lbrace X\in\fg\,:\,\ad_HX=\lambda(H)\quad\forall H\in\fa_0\rbrace$ is non-zero. A ``notion of positivity on $\fa^*$" (cf.\ \cite[Chapter II.5]{Knapp2}) is fixed, and the set of positive restricted roots for $\fa_0$ is denoted $\Sigma^+(\fg,\fa_0)$. Let $\Sigma_0^+(\fg,\fa_0)\subset \Sigma^+(\fg,\fa_0)$ denote the set of simple (positive) restricted roots; recall that this is the unique subset such that every element of $\Sigma^+(\fg,\fa_0)$ may be written as a non-negative integral linear combination of elements $\Sigma_0^+(\fg,\fa_0)$. A useful fact is that the elements of $\Sigma_0^+(\fg,\fa_0)$ constitute a basis of $\fa_0^*$. Letting $\fn_0=\bigoplus_{\lambda\in\Sigma^+(\fg,\fa_0)}\fg_{\lambda}$ gives rise to an Iwasawa decomposition of $\fg$:
\begin{equation*}
\fg=\fn_0\oplus \fa_0\oplus \fk.
\end{equation*}
We let $A_0$ be the analytic subgroup of $G$ corresponding to $\fa_0$ and $N_0$ the analytic subgroup corresponding to $\fn_0$. Both $A_0$ and $N_0$ are simply connected, $A_0$ is abelian, and $N_0$ is $\Ad$-unipotent. The corresponding Iwasawa decomposition of $G$ is $G=N_0A_0K$. Recall that the map $N_0\times A_0 \times K\rightarrow G$, $(n,a,k)\mapsto nak$ is a diffeomorphism onto.

We now turn our attention to the parabolic subgroups of $G$. Here we follow \cite[Chapter V.5]{Knapp1}. The subgroup $P_0=N_G(N_0)$ is fixed as the standard minimal parabolic subgroup. Recall that then any closed subgroup of $G$ containing $P_0$ is called a standard parabolic subgroup, and a parabolic subgroup is any subgroup of $G$ which is conjugate to a standard parabolic subgroup. Letting $\fm_0=Z_{\fk}(\fa_0)$ and $M_0=Z_K(\fa_0)$, we obtain a Langlands decomposition of $P_0$:  $P_0=N_0A_0M_0$. If $\fq$ is the Lie algebra of a parabolic subgroup $P$, then there is a Langlands decomposition of $\fq$:
\begin{equation*}
\fq=\fn\oplus\widetilde{\fa}\oplus\widetilde{\fm},
\end{equation*}
where $\fn$, $\widetilde{\fa}$, and $\widetilde{\fm}$ are the subalgebras of $\fg$ which are uniquely defined by the following properties:
\begin{enumerate}[i)]
\item $\fn$, $\widetilde{\fa}$, and $\widetilde{\fm}$ are orthogonal with respect to the inner product on $\fg$ defined by $\langle X_1,X_2\rangle:=-B(X_1,\theta X_2)$ .
\item $\widetilde{\fa}\oplus \widetilde{\fm}=\fq\cap \theta \fq$.
\item $\widetilde{\fa}=\fp\cap Z_{\widetilde{\fa}\oplus\widetilde{\fm}}$.
\end{enumerate} 
The Langlands decomposition described above may be used to give other Langlands decompositions of $\fq$: let $\fa=\Ad_k\widetilde{\fa}$ for some $k\in N_K(P)$ and $\fm=\lbrace J\in Z_{\fg}(\fa)\,:\, \langle H,J\rangle=0\quad \forall H\in \fa\rbrace$. We then have another Langlands decomposition of $\fq$:
\begin{equation*}
\fq=\fn\oplus\fa\oplus\fm.
\end{equation*}
Recall that $\fa$ is abelian, $\fn$ is nilpotent, and $\fa\oplus\fm$ normalizes $\fn$. Moreover, if we let $\Sigma(\fg,\fa)$ denote the restricted roots of $\fa$, i.e.\ the elements $\lambda\in\fa^*\setminus\lbrace 0 \rbrace$ such that $\fg_{\lambda}=\lbrace X\in\fg\,:\, \ad_H X=\lambda(H)=X\quad \forall H\in \fa\rbrace \neq \lbrace 0 \rbrace$, there exists a notion of positivity on $\fa^*$ such that
\begin{equation*}
\fn=\bigoplus_{\lambda\in\Sigma^+(\fg,\fa)} \fg_{\lambda},
\end{equation*}
$\Sigma^+(\fg,\fa)$ being the set of positive elements in $\Sigma(\fg,\fa)$. As was the case for $\fa_0$, there then exists a unique subset $\Sigma_0^+(\fg,\fa)\subset\Sigma^+(\fg,\fa)$ such that $\Sigma_0^+(\fg,\fa)$ forms a basis for $\fa^*$, and the coordinates of every element of $\Sigma^+(\fg,\fa)$ with respect to this basis are non-negative integers. An important subset of $\fa$ is the positive Weyl chamber $\fa^+$:
\begin{equation*}
\fa^+=\lbrace H\in\fa\,:\, \lambda(H)>0\quad\forall \lambda\in \Sigma^+(\fg,\fa)\rbrace=\lbrace H\in\fa\,:\, \lambda(H)>0\quad\forall \lambda\in \Sigma_0^+(\fg,\fa)\rbrace.
\end{equation*}
The (topological) closure of $\fa^+$ is denoted $\overline{\fa^+}$.
Now defining
\begin{equation*}
\fn^-=\theta \fn= \bigoplus_{\lambda\in\Sigma^+(\fg,\fa)} \fg_{-\lambda},
\end{equation*}
we see that
\begin{equation*}
\fg=\fn\oplus\fa\oplus \fm\oplus\fn^-.
\end{equation*}
Another, more general, type of Langlands decomposition will be needed when discussing the reduction theory of lattices in $G$; as this will only be needed in Section \ref{ReducSec}, we postpone the introduction of this until then.

A Langlands decomposition of $\fq$ gives rise to a corresponding Langlands decomposition of $P$: let $N$, $A$ and $M'$ be the analytic subgroups of $G$ with Lie algebras $\fn$, $\fa$, and $\fm$, respectively. Letting $M=M'Z_K(\fa)$, we have
\begin{equation*}
P=NAM,
\end{equation*}
and the map $N\times A\times M\rightarrow P$, $(n,a,m)\mapsto nam$ is a diffeomorphism onto. Similarly to the groups $N_0$, $A_0$, $N$ and $A$ are simply connected, $A$ is abelian, and $N$ is $\Ad$-unipotent. The exponentials of the positive Weyl chamber $\fa^+$ and its closure $\overline{\fa^+}$ are denoted $A^+$ and $\overline{A^+}$, respectively. We make two final definitions that will prove useful when integrating over $G$: let $N^-$ be the analytic subgroup with Lie algebra $\fn^-$, and 
\begin{equation*}
\rho_{\fa}= \frac{1}{2}\!\!\!\!\sum_{\;\;\;\;\lambda\in\Sigma^+(\fg,\fa)}\!\! \!\!\dim(\fg_{\lambda}) \lambda \in\fa^*.
\end{equation*}
Since $A$ is simply connected, we may also view $\rho_{\fa}$ as a function on $A$; by abusing notation slightly, we write $\rho_{\fa}(a)=\rho_{\fa}(\log_{\fa}a)$ for $a\in A$.

We conclude this discussion by recalling the explicit construction of standard parabolic subgroups (cf.\ \cite[Chapter VII.7]{Knapp2}). Assume now that $P=NAM$ is a standard parabolic subgroup. There then exists a subset $\scrF\subset \Sigma_0^+(\fg,\fa_0)$ such that
\begin{align*}
\fa&=\fa_{\scrF}=\lbrace H\in \fa_0\,:\, \alpha(H)=0\quad\forall \alpha\in\scrF\rbrace\\
\fn&=\fn_{\scrF}=\bigoplus_{\substack{\lambda\in\Sigma^+(\fg,\fa_0)\\\lambda|_{\fa_{\scrF}\neq 0}}} \fg_{\lambda}
\\ \fm&=\fm_{\scrF}=\lbrace J\in Z_{\fg}(\fa_{\scrF})\,:\, B(H,J)=0\quad\forall H\in\fa_{\scrF}\rbrace.
\end{align*} 
Moreover, the subset $\scrF$ uniquely determines $P$; there are therefore exactly $2^{\dim \fa_0}$ standard parabolic subgroups, and we may write $P=P_{\scrF}=N_{\scrF}A_{\scrF}M_{\scrF}$. Given subsets $\scrF_1\subset \scrF_2\subset \Sigma_{0}^+(\fg,\fa_0)$, we have $P_{\scrF_1}\subset P_{\scrF_2}$, $M_{\scrF_1}\subset M_{\scrF_2}$, $A_{\scrF_2}\subset A_{\scrF_1}$, $N_{\scrF_2}\subset N_{\scrF_1}$, and $\rho_{\fa_{\scrF_2}}=\rho_{\fa_{\scrF_1}}|_{\fa_{\scrF_2}}$. Finally, observe that $P_0=P_{\emptyset}$ and $G=P_{\Sigma_0^+(\fg,\fa_0)}$.
\subsection{Measures and Integration}\label{INTMESSEC} We will denote the Haar measure on $G$ by $d\mu$, or $dg$, and the corresponding measure on $\GaG$ by $\mu_{\GaG}$ or $dx$ (recall that the Haar measure on $G$ has been normalized so that $(\GaG,\mu_{\GaG})$ is a probability space). We will use similar notation for unimodular subgroups $H$ of $G$, i.e.\ $\mu_H$ or $dh$ will denote a Haar measure on $H$. For subgroups $H$ which are not unimodular, $d_lh$ and $d_rh$ will be used to denote left and right Haar measures, respectively. The Haar measure on $G$ can be decomposed with respect to various subgroups, cf.\ \cite[Chapter VIII.3]{Knapp2}; the most important decomposition for us is the following (cf.\ \cite[Proposition 8.45]{Knapp2}): let $P=NAM$ be a parabolic subgroup. Then $NAMN^-$ is an open dense subset of $G$, and the Haar measures on $N$, $A$, $M$, $N^-$ can be normalized so that $dg=e^{-2\rho_{\fa}(a)}\,d\overline{n}\,dm\,da\,dn$, i.e.\ for $f\in C_c(G)$,
\begin{equation}\label{measuredecomp}
\int_G f(g)\,dg= \int_N \int_A \int_M\int_{N^-}f(nam\overline{n})\,e^{-2\rho_{\fa}(a)}\,d\overline{n}\,dm\,da\,dn.
\end{equation}

\subsection{Representation Theory}\label{Reptheorysec}
As in \cite{Edwards}, the key ingredient in the proofs of our main results is the theory of unitary representations. Recall that a unitary representation of $G$ is a pair $(\pi,\scrH)$, where $\scrH$ is a separable Hilbert space, and $\pi$ is a homomorphism of topological groups from $G$ to the group of unitary operators on $\scrH$ (equipped with the strong operator topology). We say that a unitary representation $(\pi,\scrH)$ is irreducible if $\scrH$ has no non-trivial closed subspaces that are invariant under $\pi(G)$. The irreducible unitary representations may be viewed as the ``building blocks'' for all other unitary representations in the following way: given a unitary representation $(\pi,\scrH)$ of $G$, there exist a locally compact Hausdorff space $\mathsf{Z}$ and a positive Radon measure $\upsilon$ on $\mathsf{Z}$ such that 
\begin{equation}\label{INTEGRALDECOMP1}
(\pi,\mathcal{H})\cong \left(\int_{\mathsf{Z}}^{\oplus} \mathcal{\pi}_{\zeta}\,d\upsilon(\zeta), \int_{\mathsf{Z}}^{\oplus} \mathcal{H}_{\zeta}\,d\upsilon(\zeta) \right),
\end{equation}
where $(\pi_{\zeta}, \mathcal{H}_{\zeta})$ is an irreducible unitary representation of $G$ for $\upsilon$-a.e.\ $\zeta$ (cf.,\ e.g.., \cite[Corollary 14.9.5]{Wallach}). This is called the direct integral decomposition of $(\pi,\scrH)$. In Section \ref{BURGERSEC} we will prove certain identities for operators in irreducible unitary representations of $G$. The direct integral decomposition of $(\rho,V)$ will then be used to prove that similar identities hold in $(\rho,V)$.

The complexification of $\fg$ is denoted $\fg_{\CC}$; similar notation is used for subalgebras of $\fg$. We let $\Ug$ denote the universal enveloping algebra of $\fg_{\CC}$, and the terms in the canonical filtration of $\Ug$ are denoted $\Ugm$, where $m\in\NN$. Recall that a smooth vector of $(\pi,\scrH)$ is a vector $\vecv\in\scrH$ such that the map $G\rightarrow\scrH$, $g\mapsto \pi(g)\vecv$ is a $C^{\infty}$-function. The subspace of smooth vectors is denoted $\scrH^{\infty}$; by a well-known result of G\r{a}rding, $\scrH^{\infty}$ is dense in $\scrH$. We then have a (Lie algebra) representation $d\pi$ of $\fg$ on $\scrH^{\infty}$, called the derived representation. For $X\in \mathfrak{g}$,   $d\pi(X)$ is defined through
\begin{equation}\label{diff1}
d\pi(X)\vecv:=\left.\frac{d}{dt}\right|_{t=0}\!\!\!\! \pi\big(\exp(tX)\big)\vecv,\qquad \vecv\in\scrH^{\infty}.
\end{equation} 
This extends (by linearity and composition) to a representation $d\pi$ of $\Ug$ in the canonical way. Of particular importance to us is the action of the center of $\Ug$, which we denote by $\Zg$. By Schur's lemma, if $(\pi,\scrH)$ is irreducible, then for each $Z\in\Zg$ there exists a scalar $\lambda_Z\in \CC$ such that $d\pi(Z)\vecv=\lambda_Z\vecv$ for all $v\in\scrH^{\infty}$.

While it is common to consider the $K$-finite vectors of a representation $(\pi,\scrH)$ and use the dimensions of the spaces of $K$-translates of such vectors to help quantify certain aspects of the action of $G$ on $\scrH$ (cf.,\ e.g.,\ \cite[Chapters VII \& VIII]{Knapp1} as well as Definitions \ref{RATEDEF} and \ref{UNIRATEDEF}), for our purposes it is more convenient to use Sobolev norms for $\pi$ on $\scrH^{\infty}$. We choose a basis $\lbrace X_j\rbrace_{j=1,\ldots, \dim\fg}$ of $\fg$, and define the Sobolev norm of order $m$, $\|\cdot\|_{\scrS^m(\scrH)}$, on $\scrH^{\infty}$ by
\begin{equation}\label{sobnorm}
\|\vecv\|_{\scrS^m(\mathcal{H})}^2:=\sum_{U} \|d\pi(U)\vecv\|^2\qquad\vecv\in\scrH^{\infty},
\end{equation}
the sum running over all monomials $U$ in the elements $X_j$ of degree not greater than $m$. Note that this sum includes the term ``1'' of order zero. While different bases of $\fg$ give rise to different Sobolev norms, these norms will be equivalent. The closure of $\scrH^{\infty}$ with respect to the norm $\|\cdot\|_{\scrS^m(\scrH)}$ is denoted $\scrS^m(\scrH)$. For notational convenience, in the case $(\pi,\scrH)=(\rho,L^2(\GaG))$ we denote $\scrS^m(L^2(\GaG))$ by $\scrS^m(\GaG)$. Note that for our subrepresentation $(\rho,V)$ we have $\|f\|_{\scrS^m(V)}=\|f\|_{\scrS^m(\GaG)}$ for all $f\in V^{\infty}$, and hence $\scrS^m(V)=V\cap \scrS^m(\GaG)$.

For any $m\geq 0$, there exists a continuous function $C: G\rightarrow \RR_{>0}$ (depending only on $G$, $m$ and the choice of basis $\lbrace X_j\rbrace_{j=1,\ldots, \dim\fg}$) such that for any unitary representation $(\pi,\scrH)$,
 \begin{equation}\label{transgbd}
 \|\pi(g)\vecv\|_{\scrS^m(\scrH)}\leq C(g)\|\vecv\|_{\scrS^m(\scrH)}\qquad\forall g\in G,\,\vecv\in \scrH^{\infty}.
 \end{equation}
  
The direct integral decomposition of a unitary representation $(\pi,\mathcal{H})$ into irreducible representations given in \eqref{INTEGRALDECOMP1} may be used to decompose norms on $\scrH$ and $\scrH^{\infty}$: for $\vecv\in\mathcal{H}$,
 \begin{equation*}
 \|\vecv\|_{\mathcal{H}}^2=\int_{\mathsf{Z}} \|\vecv_{\zeta}\|_{\mathcal{H}_{\zeta}}^2\,d\upsilon (\zeta),
 \end{equation*}
and
\begin{equation*}
\pi(g)\vecv=\int_{\mathsf{Z}} \pi_{\zeta}(g)\vecv_{\zeta}\,d\upsilon (\zeta).
\end{equation*}
Moreover, 
\begin{equation}\label{SOBNORMDECOMP}
 \|\vecv\|_{\scrS^m(\mathcal{H})}^2=\int_{\mathsf{Z}} \|\vecv_{\zeta}\|_{\scrS^m(\mathcal{H}_{\zeta})}^2\,d\upsilon (\zeta)\qquad\forall \vecv\in\scrH^{\infty},\:m\in\NN.
\end{equation}
Direct integral decompositions allow us to give explicit constructions of \emph{intertwining operators}; recall that for a representation $(\pi,\scrH)$, these are the operators that commute with $\pi(g)$ for all $g\in G$. Assume that 
a representation $(\pi,\scrH)$ has a direct integral decomposition as in \eqref{INTEGRALDECOMP1}, and let $f:\mathsf{Z}\rightarrow\CC$ be a bounded, measurable function. We then define an operator $T_f$ on $\scrH$ through
\begin{equation}\label{inter1}
T_f\vecv:=\int_{\mathsf{Z}}f(\zeta)\vecv_{\zeta}\,d\upsilon (\zeta)\qquad\forall \vecv\in\scrH.
\end{equation}
This then gives
\begin{equation*}
T_f\pi(g)\vecv=\int_{\mathsf{Z}}f(\zeta)\pi_{\zeta}(g)\vecv_{\zeta}\,d\upsilon (\zeta)=\int_{\mathsf{Z}}\pi_{\zeta}(g)f(\zeta)\vecv_{\zeta}\,d\upsilon (\zeta)=\pi(g)T_f\vecv,
\end{equation*}
as well as $\|T_f\vecv\|\leq \left( \sup_{\zeta\in\mathsf{Z}} |f(\zeta)|\right)\|\vecv\|$ for all $\vecv\in\scrH$. We will have use of intertwining operators $T_f$ for measurable functions $f$ which are unbounded on $\mathsf{Z}$. Operators of this type will, in general, be unbounded on $\scrH$. The operators of this kind that will be of use to us will, however, be bounded operators on $\scrH^{\infty}$ with respect to suitably chosen Sobolev norms.
\subsection{Decay of Matrix Coefficients} We review some material regarding the quantitative decay of matrix coefficients of unitary representations. 
Two results regarding rates of decay of matrix coefficients will be needed; let $(\pi,\scrH)$ be an arbitrary unitary representation of $G$. We then have the following:

\begin{lem}\label{DECOMPRATE} Suppose that $(\pi,\scrH)\cong\left(\int_{\mathsf{Z}}^{\oplus} \mathcal{\pi}_{\zeta}\,d\upsilon(\zeta), \int_{\mathsf{Z}}^{\oplus} \mathcal{H}_{\zeta}\,d\upsilon(\zeta) \right)$, where each $(\pi_{\zeta},\scrH_{\zeta})$ is an irreducible unitary representation of $G$.
\begin{enumerate}[i)]
\item If $\eta$ is a rate of decay for the matrix coefficients of $g_{\RR}$ in $(\pi,\scrH)$, then $\eta$ is a rate of decay for the matrix coefficients of $g_{\RR}$ in $(\pi_{\zeta},\scrH_{\zeta})$ for $\upsilon$-a.e. $\zeta$.
\item Let $P=NAM$ be a parabolic subgroup of $G$. If $\|\cdot\|$ controls the rate of decay for the matrix coefficients of $\overline{A^+}$ in $(\pi,\scrH)$, then $\|\cdot\|$  controls the rate of decay for the matrix coefficients of $\overline{A^+}$ in $(\pi_{\zeta},\scrH_{\zeta})$ for $\upsilon$-a.e. $\zeta$.
\end{enumerate}
\end{lem} 
\begin{proof}
By e.g.\ \cite[Theorem 1.7]{Fell}, for $\upsilon$-a.e. $\zeta$, $(\pi_{\zeta},\scrH_{\zeta})$ is weakly contained in $(\pi,\scrH)$. The unitary dual of $K$ is denoted by $\widehat{K}$;  we then have the orthogonal decomposition $\scrH=\bigoplus_{\tau\in\widehat{K}}\scrH_{\tau}$, where $(\pi|_K,\scrH_{\tau})$ is isomorphic to a number (possibly countably infinite) of copies of $\tau$. Let $\Psi$ denote the function on $\widehat{K}$ defined by $\Psi(\tau)=\dim\tau$. By \cite[p. 206 (1)]{Knapp1}, $\dim(\pi(K)\vecv)\leq \Psi(\tau)^2$ for all $\vecv\in\scrH_{\tau}$. In order to prove \textit{i)}, we let $C$ and $q$ be as in \eqref{DECAYDEF}, and define
\begin{equation*}
\Phi(g)=\begin{cases} C(1+|t|^q)e^{-\eta|t|}\qquad &\mathrm{if}\;g\in Kg_{t}K\\1\qquad&\mathrm{otherwise}.\end{cases}
\end{equation*}
In the case of \textit{ii)}, without loss of generality  we assume that $\fa=\fa_{\scrF}$, for some $\scrF\subset\Sigma_0^+(\fg,\fa_0)$, and hence $\fa\subset \fa_0$. We now let $C$ and $q$ be as in \eqref{UNIFORMMATRIXCOEFFBD}, and define
\begin{equation*}
\Phi(g)= C\min_{\substack{w\in N_K(\fa_0)/Z_K(\fa_0)\\\Ad_wH\in\overline{\fa^+}}}(1+\|\Ad_w H\|^q)e^{-\|\Ad_w H\|}
\end{equation*}
if $g\in K\exp(\pm H)K$ for some $H\in\overline{\fa^+}$, and $\Phi(g)=1$ otherwise. In either case, by \cite[Theorem 7.39]{Knapp2} and the fact that $N_K(\fa_0)/Z_K(\fa_0)$ is a finite group, $\Phi$ is well-defined.  Furthermore, $\Phi$ is bi $K$-invariant and $\Phi(g)=\Phi(g^{-1})$ for all $g\in G$. Using the language of \cite[Chapter 6]{Howe}, \eqref{DECAYDEF} or \eqref{UNIFORMMATRIXCOEFFBD} thus implies that $(\pi,\scrH)$ is \emph{$(\Phi,\Psi)$-bounded} (see \cite[(6.8)]{Howe}). By \cite[Lemma 6.2(b)]{Howe} and \cite[p. 287 j)]{Howe} and we then have that for $\upsilon$-a.e. $\zeta$, $(\pi_{\zeta},\scrH_{\zeta})$ is $(\Phi,\sqrt{2}\Psi)$-bounded. Since in case \textit{i)} $\Phi(g_t)\leq C(1+|t|^q)e^{-\eta|t|}$, and in case \textit{ii)} $\Phi(\exp(H))\leq C(1+\| H\|^q)e^{-\| H\|}$ for all $H\in\overline{\fa^+}$, we are done.
\end{proof}

\begin{lem}\label{SOBNORMS}
\begin{enumerate}[i)]
\item If $\eta$ is a rate of decay for the matrix coefficients of $g_{\RR}$ in $(\pi,\scrH)$, then there exist $m\in \NN$, and $C,q\geq 0$ such that for all  $t\in\RR$, $\vecu,\vecv\in\scrH^{\infty}$,
\begin{equation}\label{SOBDECAY}
|\langle \pi(g_t)\vecu,\vecv\rangle|\leq C \|\vecu\|_{\scrS^m(\scrH)}\|\vecv\|_{\scrS^m(\scrH)} (1+|t|^q)e^{-\eta|t|}.
\end{equation}
\item Let $P=NAM$ be a parabolic subgroup of $G$. If $\|\cdot\|$ controls the decay of matrix coefficients of $\overline{A^+}$ in $(\pi,\scrH)$, then there exist $m\in \NN$, and $C,q\geq0$ such that for all  $H\in\overline{\fa^+}$ and $\vecu,\vecv\in\scrH^{\infty}$,
\begin{equation}\label{SOBNORMDECAY}
|\langle \pi(\exp(-H))\vecu,\vecv\rangle|\leq C \|\vecu\|_{\scrS^m(\scrH)}\|\vecv\|_{\scrS^m(\scrH)} (1+\|H\|^q)e^{-\|H\|}.
\end{equation}
\end{enumerate}
\end{lem}
\begin{proof}
This is fairly standard, c.f., e.g., the proofs of \cite[Theorem 3.1]{Katok} and \cite[Theorem 6]{Maucourant}: let $X_1,\ldots,X_d$ be an orthonormal basis for $\fk$ with respect to an invariant inner product, and define $\scrC_K=\sum_{i=1}^d X_i^2$, $\Omega=1-\scrC_K\in \scrZ(\fk_{\CC})$. As in the proof of Lemma \ref{DECOMPRATE}, we decompose $\scrH$ as $\scrH=\bigoplus_{\tau\in\widehat{K}}\scrH_{\tau}$. By Schur's lemma, $d\pi(\Omega)$ acts as a scalar $c(\tau)$ on each $\scrH_{\tau}^{\infty}=\scrH_{\tau}\cap \scrH^{\infty}$. The skew-symmetry of $d\pi(X)$ for all $X\in\fg$ gives $c(\tau)\in\RR_{\geq1}$. We decompose $\vecu$ and $\vecv$ as $\vecu=\sum_{\tau}\vecu_{\tau}$ and $\vecv=\sum_{\tau}\vecv_{\tau}$ with $\vecu_{\tau},\,\vecv_{\tau}\in\scrH_{\tau}^{\infty}$. Mimicking the proof of \cite[4.4.2.3]{Warner}, for $m\in\NN$ large enough, we have  
\begin{equation}\label{taubd}
\sum_{\tau\in\widehat{K}} c(\tau)^{-2m}(\dim\tau)^{2}<\infty,
\end{equation}
where $\dim\tau$ is the dimension of any vector space in a realization of $\tau$. For any $m\in\NN$, we have
\begin{equation*}
\|\vecu_{\tau}\|= c(\tau)^{-m}\|d\pi(\Omega^m)\vecu_{\tau}\|,\qquad \|\vecv_{\tau}\|= c(\tau)^{-m}\|d\pi(\Omega^m)\vecv_{\tau}\|.
\end{equation*}
We now let $g$ be either $g_t$ or $\exp(-H)$ for some $H\in\overline{\fa^+}$. Let $D=(1+|t|^q)e^{-\eta|t|}$ if $g=g_t$ and $D=(1+\|H\|^q)e^{-\|H\|}$ if $g=\exp(-H)$. Using either \eqref{DECAYDEF} or \eqref{UNIFORMMATRIXCOEFFBD}, we have
\begin{align*}
|\langle\pi(g)\vecu,\vecv\rangle|\!\!\leq\!\! \sum_{\tau\in\widehat{K}}\sum_{\sigma\in\widehat{K}}\!|\langle \pi(g)\vecu_{\tau},\vecv_{\sigma}\rangle|\!\leq CD\!\sum_{\tau\in\widehat{K}}\sum_{\sigma\in\widehat{K}} \|\vecu_{\tau}\|\|\vecv_{\sigma}\| \left(\dim(\pi(K)\vecu_{\tau})\dim(\pi(K)\vecv_{\sigma})\right)^{1/2}.
\end{align*}
As before, by \cite[p. 206 (1)]{Knapp1}, $\dim(\pi(K)\vecu_{\tau})\leq (\dim\tau)^2$ and $\dim(\pi(K)\vecv_{\sigma})\leq (\dim\sigma)^2$. Using this, together with the Cauchy-Schwarz inequality and \eqref{taubd}, gives
\begin{align*}
\leq &CD\sum_{\tau\in\widehat{K}}\sum_{\sigma\in\widehat{K}} \|\vecu_{\tau}\|\|\vecv_{\sigma}\| (\dim\tau)(\dim\sigma)\\&=CD\left(\sum_{\tau\in\widehat{K}} c(\tau)^{-m}\|d\pi(\Omega^m)\vecu_{\tau}\|(\dim\tau)\right)\left(\sum_{\sigma\in\widehat{K}} c(\sigma)^{-m}\|d\pi(\Omega^m)\vecv_{\sigma}\|(\dim\sigma) \right)\\&\leq CD \left( \sum_{\tau\in\widehat{K}} \|d\pi(\Omega^m)\vecu_{\tau}\|^2 \right)^{1/2}\left( \sum_{\sigma\in\widehat{K}} \|d\pi(\Omega^m)\vecv_{\sigma}\|^2 \right)^{1/2}\left(\sum_{\varrho\in\widehat{K}} c(\varrho)^{-2m}\dim(\varrho)^2  \right) \\&= CD\|d\pi(\Omega^m)\vecu\|\|d\pi(\Omega^m)\vecv\|\left(\sum_{\tau\in\widehat{K}} c(\tau)^{-2m}(\dim\tau)^{2}\right)\\&\quad\ll D \|\vecu\|_{\scrS^{2m}(\scrH)}\|\vecv\|_{\scrS^{2m}(\scrH)},
\end{align*}
as desired.
\end{proof}

\section{Integral Formulas}\label{BURGERSEC}
The aim of this section is to establish formulae similar to  \cite[Lemma 1]{Burger90} and \cite[Proposition 4]{Edwards} for our case of a general group $G$. 
\subsection{Harish-Chandra Isomorphisms} As a starting point, we have the following lemma:
\begin{lem}\label{LIEIDENT}
Let $P=NAM$ be a parabolic subgroup of $G$, with corresponding Lie algebra $\fn\oplus\fa\oplus\fm$. Given $H\in\fa$, we can find elements $Z_0,\, Z_1, \ldots, Z_{W-1} $ of $\Zg$ such that
\begin{equation*}
H^W+H^{W-1}Z_{W-1}+\ldots+ H Z_1+Z_0\in\fn\,\Ug.
\end{equation*}
\end{lem}
\begin{remark}
Although the number $W$ depends on the choice of $P$ and $H$, it is uniformly bounded over all such choices, as seen in the proof.
\end{remark}
\begin{proof} Much of the proof follows by mimicking the proof of \cite[Proposition 8.22]{Knapp1}. For the sake of completeness, we write out some of the details: without loss of generality, we may assume that $P=NAM=N_{\scrF}A_{\scrF}M_{\scrF}\supset P_0$ for some subset $\scrF\subset\Sigma_0^+(\fg,\fa_0)$. By the construction of standard parabolic subgroups, we then have $\fa_0\oplus\fm_0\subset \fa\oplus\fm $. We choose a maximal abelian subalgebra $\ft_0$ of $\fm_0$; by \cite[Proposition 6.47]{Knapp2} $\fh_{\CC}=(\fa_0\oplus\ft_0)_{\CC}$ is a Cartan subalgebra of $\fg_{\CC}$. Let $\ft=\ft_0\oplus\fa^{\perp}$, where $\fa^{\perp}$ is the orthogonal complement of $\fa$ in $\fa_0$ with respect to the Killing form. Note that $\ft$ is a maximal abelian subalgebra of $\fm$, as seen from the construction of $\fm$ in \cite[(7.77a)]{Knapp2}. Since $\fh_{\CC}$ is a Cartan subalgebra of $\fg_{\CC}$, we have a root space decomposition $\fg_{\CC}=\fh_{\CC}\oplus\bigoplus_{\lambda\in\Delta(\fg_{\CC},\fh_{\CC})} \fg_{\lambda}$. We order the roots so that $\Sigma^+(\fg,\fa)$ consists of the elements of $\Delta^+(\fg_{\CC},\fh_{\CC})$ that are non-zero when restricted to $\fa$. This ensures that $\fn_{\CC}\subset \fg_{\CC}^+$, where $\fg^+_{\CC}=\bigoplus_{\lambda\in\Delta^+(\fg_{\CC},\fh_{\CC})} \fg_{\lambda}$. The subalgebra $\fg_{\CC}^-$ is defined in a similar fashion, and we get a triangular decomposition of $\fg_{\CC}$: $\fg_{\CC}=\fg_{\CC}^+\oplus \fh_{\CC} \oplus \fg_{\CC}^-$. Defining the action of the Weyl group $\scrW=\scrW(\fg_{\CC},\fh_{\CC})$ on $\fh_{\CC}$ by $\lambda(wX)=(w^{-1}\lambda)(X)$ for all $\lambda\in\fh_{\CC}^*$, $X\in\fh_{\CC}$, we consider the polynomial (with coefficients in $\Uh$) given by the product
\begin{equation}\label{HPOLY}
p_H(x)=\prod_{w\in\scrW/\mathrm{Stab}_{\scrW}(H)} \big(x-(wH+\delta_{\fh_{\CC}}(H)1)\big)=x^W+J_{W-1}x^{W-1}+\ldots+ J_1 x + J_0,
\end{equation} 
where $\delta_{\fh_{\CC}}(H)=\frac{1}{2}\sum_{\lambda\in\Delta^+(\fg_{\CC},\fh_{\CC})}\lambda(H)=\rho_{\fa}(H)$. We denote the set of $\scrW$-invariant elements of $\Uh$ by $\Uh^{\scrW}$; since $\Uh\cong\mathrm{Sym(\fh_{\CC})}$, we may view $\Uh$ as the space of polynomial functions on $\fh_{\CC}^*$, and the action of $\scrW$ is then defined by $(wJ)(\lambda)=J(w^{-1}\lambda)$, where $J\in\Uh$, $w\in\scrW$, and $\lambda\in\fh_{\CC}^*$. From their definition, we see that the coefficients $J_{W-1},\ldots, J_0$ of $p(x)$ are in $\Uh^{\scrW}$, and $p_H\big(H+\delta_{\fh_{\CC}}(H)\big)=0$. Let $\sigma$ denote the ``$\delta_{\fh_{\CC}}$-shift'' on $\Uh$, i.e.\ the unique $\CC$-algebra automorphism given by by $\sigma(X)=X+\delta_{\fh_{\CC}}(X)1$ for all $X\in\fh_{\CC}$, and then extending to all of $\Uh$ in the canonical way. The triangular decomposition of $\fg_{\CC}$ and the Poincar\'e-Birkhoff-Witt theorem give $\Ug=\Uh\oplus (\fg^+_{\CC}\Ug+\Ug\fg^-_{\CC})$. Let $\Proj_{\fh_{\CC}}:\Ug\rightarrow\Uh$ be the projection onto the first summand in this decomposition, and $\gamma=\sigma\circ\Proj_{\fh_{\CC}}$. Then $\gamma$ is the \emph{Harish-Chandra isomorphism}, an isomorphism from $\Zg$ to $\Uh^{\scrW}$. (Indeed, this is seen by following \cite[Chapter V.5]{Knapp2}, but using $\Delta^-$ as the set of positive roots; then ``$\gamma_n'$, ``$\delta$'', and ``$\gamma$'' in \cite[Chapter V.5]{Knapp2} correspond to $\mathrm{Proj}_{\fh_{\CC}}$, $-\delta_{\fh_{\CC}}$, and  $\gamma|_{\Zg}$ in our set-up, and the desired statement is given by \cite[Theorem 5.44]{Knapp2}.) We now let $Z_j=\gamma|_{\Zg}^{-1}(J_j)$, $j=0,1,\ldots, W-1$. Recall that $p_H(\sigma(H))=0$, i.e.\
\begin{equation*}\sigma(H)^W+\gamma(Z_{W-1})\sigma(H)^{W-1}+\ldots+\gamma(Z_{1})\sigma(H)+\gamma(Z_0)=0.\end{equation*} Applying $\sigma^{-1}$ to this identity gives
\begin{align*}
H^W+ \Proj_{\fh_{\CC}}(Z_{W-1})H^{W-1}+\ldots+ \Proj_{\fh_{\CC}}(Z_1)H+\Proj_{\fh_{\CC}}(Z_0)=0.
\end{align*}
Since $Z_j-\Proj_{\fh_{\CC}}(Z_j) \in \fg_{\CC}^+\Ug$, we now have
\begin{equation*}
P(H):=H^W+H^{W-1}Z_{W-1}+\ldots+ H Z_1+Z_0\in\fg_{\CC}^+\Ug.
\end{equation*}
It remains to prove that this polynomial is in fact in $\fn\,\Ug$ (which is a subspace of $\fg_{\CC}^+\Ug$). Note that $\Proj_{\fh_{\CC}}(P(H))=0$. We decompose $\fg$ as $\fg=\fn\oplus\fa\oplus\fm\oplus\fn^-$, and, as noted in \cite[Proof of Proposition 8.22]{Knapp1}, the same type of argument as in \cite[Proof of Proposition 5.34 (b)]{Knapp2} gives
\begin{equation*}
\Zg\subset\fn_{\CC}\Ug\oplus \scrU\big((\fa\oplus\fm)_{\CC}\big),
\end{equation*}
hence
\begin{equation}\label{CENTERCONTAIN}
\Zg\scrU(\fa_{\CC})\subset\fn_{\CC}\Ug\oplus \scrU\big((\fa\oplus\fm)_{\CC}\big).
\end{equation}
Note that $\fa\oplus\fm$ is a reductive Lie algebra with Cartan subalgebra $\fa\oplus\ft$. We thus have a triangular decomposition $(\fa\oplus\fm)_{\CC}= \fm_{\CC}^-\oplus(\fa\oplus\ft)_{\CC}\oplus \fm_{\CC}^+$, where the order on $\Delta(\fm_{\CC},\ft_{\CC})$ that defines $\fm_{\CC}^{\pm}$ is chosen so that $\Delta^+(\fm_{\CC},\ft_{\CC})=\Delta^+(\fg_{\CC},\fh_{\CC})|_{\ft_{\CC}}\setminus\lbrace 0\rbrace$. This choice of positivity gives $\fg_{\CC}^+=\fn_{\CC}\oplus \fm_{\CC}^+$. Letting $\Proj_{(\fa+\fm)_{\CC}}:\Zg\scrU(\fa_{\CC})\rightarrow\scrU\big((\fa\oplus\fm)_{\CC}\big)$ be the projection onto the second summand of \eqref{CENTERCONTAIN}, one checks that for all $Z\in\Zg$, $\Proj_{(\fa\oplus\fm)_{\CC}}(Z)\in\scrZ\big((\fa\oplus\fm)_{\CC}\big)$, hence $\Proj_{(\fa\oplus\fm)_{\CC}}\big(\Zg\scrU(\fa_{\CC})\big)\subset\scrZ\big((\fa\oplus\fm)_{\CC}\big)$. Since $\fa\oplus\ft$ is a Cartan subalgebra of the reductive Lie algebra, $\fa\oplus\fm$, by the same arguments as for $\fg_{\CC}$, $\scrZ\big((\fa\oplus\fm)_{\CC}\big)\subset \scrU\big((\fa\oplus\ft)_{\CC}\big)\oplus\fm_{\CC}^+\scrU\big((\fa\oplus\ft)_{\CC}\big)$. Letting $\Proj_{(\fa\oplus\ft)_{\CC}}$ denote the projection from $\scrZ\big((\fa\oplus\fm)_{\CC}\big)$ onto the first summand, we get
\begin{align*}
\Proj_{\fh_{\CC}}|_{\Zg\scrU(\fa_{\CC})}=\Proj_{(\fa\oplus\ft)_{\CC}}\circ \Proj_{(\fa\oplus\fm)_{\CC}}
\end{align*}
(cf.\ \cite[p. 225 (8.34)]{Knapp1}). Since composition with an appropriate half-sum of positive roots turns $\Proj_{(\fa\oplus\ft)_{\CC}}$ into a Harish-Chandra isomorphism for $\scrZ\big((\fa\oplus\fm)_{\CC}\big)$, $\Proj_{(\fa\oplus\ft)_{\CC}}$ is invertible. In particular, since  $\Proj_{(\fa\oplus\ft)_{\CC}}\big(\Proj_{(\fa\oplus\fm)_{\CC}}(P(H))\big)=\Proj_{\fh_{\CC}}(P(H))=0$, we conclude that $\Proj_{(\fa\oplus\fm)_{\CC}}(P(H))=0$, i.e.\ $P(H)\in\fn_{\CC}\,\Ug=\fn\,\Ug$.
\end{proof}

\subsection{Differential Equations}\label{DIFFEQSSEC} We shall now use Lemma \ref{LIEIDENT} to express the ``averaging operator'' in terms of differential operators. Let $(\pi,\scrH)$ be an irreducible unitary representation of $G$, and assume that $\vecv\in\scrH^{\infty}$. For $\chi\in C^{\infty}_c(U^+)$ define $\scrI_{\vecv}^{\chi}:G\rightarrow\scrH$ by
\begin{equation*}
\scrI_{\vecv}^{\chi}(g)=\int_{U^+} \chi(u)\pi(ug)\vecv\,du.
\end{equation*}
As we shall see later, the restriction that $\chi\in C_c^{\infty}(U^+)$ may be replaced by the condition that $\chi\in C_c^{m}(U^+)$ for some fixed number $m$, but for convenience it is assumed throughout the remainder of Section \ref{BURGERSEC} that $\chi$ is smooth. Note that $\scrI_{\vecv}^{\chi}$ is linear in both $\vecv$ and $\chi$, and for $X\in\fg$:
\begin{align}\label{XDIFFOPI}
(X\scrI_{\vecv}^{\chi})(g)=\left.\frac{d}{ds}\right|_{s=0}\scrI_{\vecv}^{\chi}(g\exp(sX)\big)&=\left.\frac{d}{ds}\right|_{s=0}\int_{U^+} \chi(u)\pi(ug\exp(sX))\vecv\,du\\\notag=&\int_{U^+} \chi(u)\pi(ug)\left.\frac{d}{ds}\right|_{s=0}\pi(\exp(sX))\vecv\,du=\scrI_{d\pi(X)\vecv}^{\chi}(g),
\end{align}
the exchange of integration and differentiation being permitted since $\chi$ has compact support. We use linearity and composition to extend this to all of $\Ug$, i.e.\
\begin{equation}\label{AVGDIFF}
\left(U\scrI_{\vecv}^{\chi}\right)(g)=\scrI_{d\pi(U)\vecv}^{\chi}(g)\qquad \forall U\in \Ug.
\end{equation} 
Similarly, $\fu^+$ acts on $C_c^{\infty}(U^+)$ by 
\begin{equation*}
(X\psi)(u):=\left.\frac{d}{ds}\right|_{s=0} \psi\left(u\exp(sX)\right)\qquad\forall X\in\fu^+,\,u\in U^+,\,\psi\in C_c^{\infty}(U^+).
\end{equation*}
By linearity and composition, this defines $X\psi$ for all $X\in \scrU(\fu^+_{\CC})$. The next lemma relates the action of $\scrU(\fu^+_{\CC})$ on $\scrI_{\vecv}^{\chi}(g)$ and $\chi$ for certain special $g\in G$:
\begin{lem}\label{CHIDIFF}
Let $X\in\fu^+$ and $g\in G$ be such that $\Ad_gX\in\fu^+$. Then for all $\vecv\in\scrH^{\infty}$ and $\chi\in C_c^{\infty}(U^+)$,
\begin{equation*}
\left(X\scrI_{\vecv}^{\chi}\right)(g)=-\scrI_{\vecv}^{\Ad_gX\chi}(g).
\end{equation*}
\end{lem}
\begin{proof}
\begin{align*}
\left(X\scrI_{\vecv}^{\chi}\right)(g)=&\left.\frac{d}{ds}\right|_{s=0}\int_{U^+} \chi(u)\pi\big(ug\exp(sX)\big)\vecv\,du\\&=
\left.\frac{d}{ds}\right|_{s=0}\int_{U^+} \chi(u)\pi\big(u\exp(s\Ad_gX)g\big)\vecv\,du
\\&=\int_{U^+} \left.\frac{d}{ds}\right|_{s=0}\chi\big(u\exp(-s\Ad_gX)\big)\pi(ug)\vecv\,du=\scrI_{\vecv}^{-\Ad_gX\chi}(g).
\end{align*}
\end{proof}

We now let $P=NAM$ be the \emph{unique} parabolic subgroup such that $Y\in\fa^+$ and $U^+=N$. This allows us to fix a root basis $X_1,\ldots, X_d$ of $\fu^+=\fn$; for each $j=1,\ldots, d$, there is some $\alpha_j>0$ such that $\Ad_{g_t}X_j=e^{\alpha_jt}X_j$ for all $t\in\RR$. Since $\log(g_1)=Y\in \fa$, we apply Lemma \ref{LIEIDENT} to find $W\in\NN_{>0}$,  $Z_0,\ldots, Z_{W-1}\in\Zg$, and $U_1,\ldots, U_d \in\Ug$ such that
\begin{equation}\label{Yident}
Y^W+Y^{W-1}Z_{W-1}+\ldots+ Y Z_1+Z_0=\sum_{j=1}^d X_jU_j
\end{equation}
(as is clear from the proof of Lemma \ref{LIEIDENT}, if $k_0\in K$ is such that $\Ad_{k_0}\fa=\fa_{\scrF}$, for some $\scrF\subset\Sigma_0^+(\fg,\fa_0)$, then  $W=|\scrW/\mathrm{Stab}_{\scrW}(\Ad_{k_0}Y)|$, $\scrW$ being the Weyl group defined with respect to $(\fa_0\oplus\ft_0)_{\CC}$). Combining \eqref{AVGDIFF}, \eqref{Yident} and Lemma \ref{CHIDIFF} gives
\begin{equation}\label{DIFFS1}
(Y^W\scrI_{\vecv}^{\chi})(g_t)+\sum_{i=0}^{W-1}(Y^i\scrI_{d\pi(Z_i)\vecv}^{\chi})(g_t)=-\sum_{j=1}^d e^{\alpha_jt}\scrI_{d\pi(U_j)\vecv}^{X_j\chi}(g_t).
\end{equation}
By Schur's lemma, each $d\pi(Z_i)$ acts as a scalar $a_i\in\CC$. This fact and calculations analogous to those in \eqref{XDIFFOPI} give
\begin{equation*}
(Y^{i}\scrI_{d\pi(Z_i)\vecv}^{\chi})(g_t)=a_i\frac{d^{i}}{dt^{i}}\scrI_{\vecv}^{\chi}(g_t),\qquad a_i\in\CC.
\end{equation*}
We let $\lambda_1,\ldots, \lambda_W$ be the roots of the polynomial $z^W+a_{W-1}z^{W-1}+\ldots+ a_1z+a_0$. Abusing notation slightly by letting $\scrI_{\vecv}^{\chi}(t)=\scrI_{\vecv}^{\chi}(g_t)$, \eqref{DIFFS1} may be rewritten as
\begin{equation}\label{MAINODE}
\left( \prod_{i=1}^W (\sfrac{d}{dt}-\lambda_i)\right)\scrI_{\vecv}^{\chi}(t)=-\sum_{j=1}^d e^{\alpha_jt}\scrI_{d\pi(U_j)\vecv}^{X_j\chi}(t).
\end{equation}
Now, each $\scrI_{d\pi(U_j)\vecv}^{X_j\chi}(t)$ will solve the same type of equation, i.e.\
\begin{equation*}
\left( \prod_{i=1}^W (\sfrac{d}{dt}-\lambda_i)\right)\scrI_{d\pi(U_j)\vecv}^{X_j\chi}(t)=-\sum_{k=1}^d e^{\alpha_kt}\scrI_{d\pi(U_kU_j)\vecv}^{X_kX_j\chi}(t).
\end{equation*}
This property will be important in obtaining a better rate than $\min_j \alpha_j$ in Theorem \ref{maintheorem1}.
\subsection{Burger's Formula}\label{BURGERFORMULASEC}
Here we present ``integral formulas'' for the solution of \eqref{MAINODE}. From now on assume that $(\pi,\scrH)$ is an irreducible unitary representation with $\eta$ as a rate of decay for the matrix coefficients of $g_{\RR}$ in $(\pi,\scrH)$; in particular (cf.\ Lemma \ref{SOBNORMS}), we assume that there exists $m\in\NN$ such that
 \begin{equation}\label{EPSILONBD}
 |\langle \pi(g_t)\vecu,\vecv\rangle|\ll_{\epsilon} \|\vecu\|_{\scrS^m(\scrH)}\|\vecv\|_{\scrS^m(\scrH)} e^{(\eta-\epsilon)t}
 \end{equation}
for all $\epsilon>0$, $t\leq 0$, and $\vecu,\vecv\in\scrH^{\infty}$. 
Throughout this section we assume that we have the differential equation
\begin{equation*}
\left(\prod_{i=1}^W (\sfrac{d}{dt}-\lambda_i)\right)\scrI_{\vecv}^{\chi}(t)=\psi(t).
\end{equation*}
We let $\underline{\lambda}$ be the multi-set of order $W$ consisting of the complex numbers $\lambda_i$, $i=1,\ldots,W$. Since we will require several different orderings of the $\lambda_i$s, we consider the following construction: let $\mathrm{Sym}_W$ denote the group of permutations of the set $\lbrace 1,\,2,\,\ldots,\,W\rbrace$. We may identify $\CC^W/\mathrm{Sym}_W$ with the family of multi-sets of complex numbers having order $W$. The element $\underline{\lambda}$ may thus be viewed as an element of $\CC^W/\mathrm{Sym}_W$. 

\begin{lem}\label{gint lem} 
Assume that $\|\psi(t)\|\leq C_{\psi,\gamma}e^{\gamma t}$ for some positive number $\gamma$ and all $t\leq 0$, and choose a number $\beta_{\gamma}$ so that \textit{i)} $0<\beta_{\gamma}< \gamma$ \textit{ii)} $\beta_{\gamma}\leq \eta$. Label and order the $\lambda_i$s in the following way:
\begin{equation*}
\underline{\lambda}=(\lambda_1,\lambda_2,\ldots,\lambda_W)=(\lambda_1^-,\lambda_2^-,\ldots, \lambda_{m_1}^-,\lambda^+_1,\lambda^+_2,\ldots, \lambda_{m_2}^+),
\end{equation*}
where $\beta_{\gamma}> \Re(\lambda_1^-)\geq \Re(\lambda_2^-)\geq\ldots\geq \Re(\lambda_{m_1}^-)$, and $\Re(\lambda_1^+)\geq \Re(\lambda_2^+)\geq\ldots \geq\Re(\lambda_{m_2}^+)\geq \beta_{\gamma}$. If $m_1>0$, then for all $t\in\RR$ we have
\begin{align}\label{gdefsec14}
\left( \prod_{i=1}^{m_2} (\sfrac{d}{dt}-\lambda_i^+)\right)\scrI_{\vecv}^{\chi}(t)= e^{\lambda_{m_1}^-t} \int_{-\infty}^t & e^{(\lambda_{m_1-1}^- -\lambda_{m_1}^- )t_{m_1}}  \int_{-\infty}^{t_{m_1}}  e^{(\lambda_{m_1-2}^- -\lambda_{m_1-1}^- )t_{m_1-1}} \\\notag&\ldots \int_{-\infty}^{t_3} e^{(\lambda_1^--\lambda^-_2)t_2}\int_{-\infty}^{t_2} e^{-\lambda_1^-t_1} \psi(t_1)\,dt_1\,dt_2\,\ldots dt_{m_1-1}\,dt_{m_1}
\end{align}
(where the product in the left-hand side is viewed as empty if $m_2=0$).
\end{lem}
\begin{proof}
The proof follows that of \cite[Lemma 3]{Edwards}: we set $\lambda_0^-=0$, and for $i=0,\,1,\ldots,\,m_1$, let $\psi_i(t)$ be such that
\begin{equation*}
e^{\lambda^-_i t}\psi_i(t)=\left( \prod_{j={i+1}}^{m_1}( \sfrac{d}{dt}-\lambda^-_j)\right)\left( \prod_{j=1}^{m_2}( \sfrac{d}{dt}-\lambda^+_j)\right)\scrI_{\vecv}^{\chi}(t).
\end{equation*}
In particular,
\begin{equation}\label{psim1-}
e^{\lambda^-_{m_1} t}\psi_{m_1}(t)=\left( \prod_{j=1}^{m_2}( \sfrac{d}{dt}-\lambda^+_j)\right)\scrI_{\vecv}^{\chi}(t).
\end{equation}
Note that for $i\geq1$,
\begin{equation*}
\psi_i'(t)=e^{(\lambda_{i-1}^--\lambda_i^-)t}\psi_{i-1}(t).
\end{equation*}
By the fundamental theorem of calculus,
\begin{equation*}
\psi_i(t)-\psi_i(r)=\int_r^t e^{(\lambda^-_{i-1}-\lambda^-_i)s}\psi_{i-1}(s)\,ds.
\end{equation*}
Now let $i\in\lbrace 1,\ldots,m_1\rbrace$, and assume we have the bound
\begin{equation}\label{psiibndas}
 \|\psi_{i-1}(s)\|\leq D_{i-1}e^{(\gamma-\Re(\lambda_{i-1}^-)) s}\qquad\mathrm{for\,\,all}\,\,s\leq 0.
\end{equation} 
Then for all $r\leq t\leq 0$ we have
\begin{align*}
\|\psi_i(t)-\psi_i(r)\|&\leq \int_r^t e^{\big(\Re(\lambda^-_{i-1})-\Re(\lambda^-_{i})\big)s} D_{i-1}e^{(\gamma-\Re(\lambda_{i-1}^-))s}\,ds\\&\quad\ll e^{(\gamma-\Re(\lambda^-_i))t}-e^{(\gamma-\Re(\lambda^-_i))r}.
\end{align*}
It follows that $\psi_i(r)$ converges to some vector $\vecv_{i,\infty}$ as $r\rightarrow-\infty$, and
\begin{equation*}
\psi_i(t)=\vecv_{i,\infty}+\int_{-\infty}^{t}e^{(\lambda_{i-1}^--\lambda^-_i)s}\psi_{i-1}(s)\,ds.
\end{equation*}
We now wish to prove that $\vecv_{i,\infty}=0$. Let $\vecw_i=\left( \prod_{j={i+1}}^W( d\pi(Y)-\lambda_j)\right)\vecv$, and note that from the definition of $\psi_i(r)$, 
\begin{equation*}
 \psi_i(r)=e^{-\lambda^-_ir}\scrI^{\chi}_{\vecw_i}(r),
\end{equation*}
hence
\begin{equation*}
\vecv_{i,\infty}=\lim_{r\rightarrow-\infty} e^{-\lambda^-_ir}\scrI^{\chi}_{\vecw_i}(r).
\end{equation*}
We now use the rate of decay of matrix coefficients to prove that $\langle \vecv_{i,\infty},\vecu\rangle=0$ for all $\vecu\in\scrH^{\infty}$. Since $\scrH^{\infty}$ is dense in $\scrH$, we may then conclude that $\vecv_{i,\infty}=0$. We have
\begin{align*}
|\langle\vecv_{i,\infty},\vecu\rangle|&= \lim_{r\rightarrow-\infty} |\langle  e^{-\lambda^-_ir}\scrI^{\chi}_{\vecw_i}(r),\vecu\rangle| \leq \liminf_{r\rightarrow-\infty}e^{-\Re(\lambda^-_i)r} \int_{U^+} |\chi(u)||\langle\pi(ug_r)\vecw_i,\vecu\rangle| \,du\\&\quad=\liminf_{r\rightarrow-\infty}e^{-\Re(\lambda^-_i)r}\int_{U^+} |\chi(u)||\langle\pi(g_r)\vecw_i,\pi(u^{-1})\vecu\rangle |\,du\\&\qquad \ll_{\epsilon} \liminf_{r\rightarrow-\infty}  \|\chi\|_{L^{\infty}(N)}\|\vecw_i\|_{\scrS^m(\scrH)}e^{(\eta-\epsilon-\Re(\lambda^-_i))r} \int_{\supp(\chi)}  \!\!\|\pi(u^{-1})\vecu\|_{\scrS^m(\scrH)}\,du
\end{align*}
where \eqref{EPSILONBD} was used in the last inequality. Choosing $\epsilon$ so that $\eta-\epsilon-\Re(\lambda_i^-) > 0$, and noting that by \eqref{transgbd}, the integral over $\supp(\chi)$ is finite, we conclude that $\langle \vecv_{i,\infty},\vecu\rangle=0$. We have thus proved that under the assumption \eqref{psiibndas}, we have
\begin{equation}\label{gdeflemeq}
\psi_i(t) =\int_{-\infty}^{t}e^{(\lambda_{i-1}^--\lambda^-_i)s}\psi_{i-1}(s)\,ds,
\end{equation}
and 
\begin{equation}\label{PSIIBD}
\| \psi_i(t)\|\leq  \int_{-\infty}^t e^{\big(\Re(\lambda^-_{i-1})-\Re(\lambda^-_i)\big)s}  D_{i-1}e^{(\gamma-\Re(\lambda_{i-1}^-)) s}\,ds=\frac{  D_{i-1}e^{(\gamma-\Re(\lambda_{i}^-)) s}}{\gamma-\Re(\lambda_i^-)}.
\end{equation}
Since $\|\psi_0(t)\|=\|\psi(t)\|\leq C_{\psi,\gamma}e^{\gamma t}$, induction establishes that \eqref{gdeflemeq} and \eqref{psiibndas} hold for all $i$, $1\leq i \leq m_1$. Using this fact together with $\psi_0(t)=\psi(t)$ and \eqref{psim1-}, we conclude that \eqref{gdefsec14} holds.
\end{proof}
Our first version of ``Burger's formula'' (cf.\ \cite[Lemma 1]{Burger90} and \cite[Proposition 4]{Edwards}) can now be stated. Let $|\underline{\lambda}|_{\infty}=\max_{1\leq i\leq W}|\lambda_i|$. We then have the following:
\begin{prop}\label{BurgerProp1}
There exist (completely explicit) functions
\begin{equation*}F:\,(\CC^W/\mathrm{Sym}_W)\times \RR_{>0}\times\RR_{\leq 0}\times \RR_{\leq 0}\rightarrow\CC
\end{equation*}
and 
\begin{equation*}
F_0,\, F_1,\ldots,\,F_W:\,(\CC^W/\mathrm{Sym}_W)\times \RR_{>0}\times\RR_{\leq 0}\rightarrow\CC,
\end{equation*}
such that the following hold:
\begin{enumerate}[i)] 
\item Whenever $\scrI_{\vecv}^{\chi}(t)$, $\psi(t)$ and $\beta_{\gamma}$ satisfy the assumptions of Lemma \ref{gint lem}, we have for all $t\leq 0$:
\begin{equation*}
\scrI_{\vecv}^{\chi}(t)=\int_{-\infty}^0 F(\underline{\lambda},\beta_{\gamma},t,s) \psi(s)\,ds +\sum_{i=0}^{W} F_{i}(\underline{\lambda},\beta_{\gamma},t)\scrI_{d\pi(Y^{i})\vecv}^{\chi}(0).
\end{equation*} 
\item
\begin{equation*}
|F(\underline{\lambda},\beta_{\gamma},t,s)|\ll |t|^{W-1-m_0}|s|^{m_0} e^{\beta_{\gamma}(t-s)}.
\end{equation*}
\item
\begin{equation*}
|F_{i}(\underline{\lambda},\beta_{\gamma}, t)| \ll (1+|\underline{\lambda}|_{\infty}^{W})(1+|t|^W)e^{\beta_{\gamma} t}.
\end{equation*}
\end{enumerate}
where $m_0=\max\lbrace 0, m_1-1\rbrace$ ($m_1$ being the same as in Lemma \ref{gint lem}). The implied constants depend only on $W$.
\end{prop}
\begin{proof}
We use the same notation as in Lemma \ref{gint lem}, as well as letting $\Psi(t)$ denote the left-hand side of \eqref{gdefsec14}. We also let $\lambda_0^+=0$, and define $\Psi_i(t)$ successively for $i=m_2,$ $m_2-1,\ldots$, $0$ through
\begin{align}\label{gidiffs}
\scrI_{\vecv}^{\chi}(t)&=e^{\lambda_{m_2}^+t} \Psi_{m_2}(t),\\\notag \sfrac{d}{dt} \Psi_{i+1}(t)&= e^{(\lambda^+_i-\lambda^+_{i+1})t}\Psi_i(t).
\end{align}
From these definitions, we obtain
\begin{equation*}
e^{\lambda_i^+t} \Psi_i(t)=\left( \prod_{j=i+1}^{m_2} (\sfrac{d}{dt}-\lambda_j^+)\right)\scrI_{\vecv}^{\chi}(t),
\end{equation*}
with (as usual) the product being viewed as empty if $i=m_2$. In particular, $\Psi_0(t)=\psi(t)$ if $m_1=0$, and $\Psi_0(t)=\Psi(t)$ otherwise. Making the definition
\begin{equation*}
\vecv_i=\left( \prod_{j=i+1}^{m_2} (d\pi(Y)-\lambda_j^+)\right)\vecv
\end{equation*}
(so $\vecv_{m_2}=\vecv$) gives
\begin{equation*}
\Psi_i(t)=\scrI_{\vecv_i}^{\chi}(t).
\end{equation*}
There are thus polynomials $\scrP_{i,j}(\underline{\lambda})$ (in fact, elementary symmetric polynomials in the $\lambda_i^+$s) such that
\begin{equation*}
\Psi_i(0)=\sum_{j=0}^{m_2}\scrP_{i,j}(\underline{\lambda})\scrI_{d\pi(Y^j)\vecv}^{\chi}(0).
\end{equation*}
Assuming $m_1>0$, Lemma \ref{gint lem} gives
\begin{align*}
\Psi_0(t)= \int_{-\infty} ^t \ldots \int_{-\infty}^{r_3} \int_{-\infty}^{r_2} \psi(r_1) \exp\left(\lambda_{m_1}^-t-\lambda_1^-r_1 + \sum_{i=2}^{m_1} (\lambda_{i-1}^- -\lambda_{i}^-)r_i  \right)\,dr_1\,dr_2\,\ldots dr_{m_1}.
\end{align*}
Reversing the order of the integration gives
\begin{align*}
\Psi_0(t)=\int_{-\infty}^0 \psi(s)\Phi_0(t,s)\,ds,
\end{align*}
where
\begin{equation*}
\Phi_0(t,s)= \begin{cases}0\qquad &\mathrm{if}\,s>t\\ \int_{s}^{t} \int_{r_2}^{t} \ldots \int_{r_{m_1-1}}^t\!\! \exp\!\left(\lambda_{m_1}^-t-\lambda_1^-s + \sum_{i=2}^{m_1} (\lambda_{i-1}^- -\lambda_{i}^-)r_i  \right)dr_{m_1}\,\ldots dr_2\qquad &\mathrm{if}\, s\leq t.\end{cases}
\end{equation*}
Repeatedly using $\Re(\lambda_{i-1}^-)\geq \Re(\lambda_{i}^-)$, and finally $\beta_{\gamma}>\Re(\lambda_1^-)$, gives
\begin{equation}\label{Phi0bd}
|\Phi_0(t,s)|\leq \begin{cases}0\qquad &\mathrm{if}\,s>t\\(t-s)^{m_1-1}\exp\big(\beta_{\gamma}(t-s)\big)\qquad &\mathrm{if}\, s\leq t.\end{cases}
\end{equation}
We now write
\begin{equation*}
\Psi_i(t)=\int_{-\infty}^0\psi(s)\Phi_i(t,s)\,ds+\sum_{j=0}^{m_2}\Pi_{i,j}(t,\underline{\lambda})\scrI_{d\pi(Y^j)\vecv}^{\chi}(0),
\end{equation*}
where $\Pi_{0,j}(t,\underline{\lambda})\equiv 0$, and the other $\Pi_{i,j}$ and $\Phi_i$ are defined iteratively using \eqref{gidiffs} (if $m_2>0$) and the fundamental theorem of calculus, i.e.\ for $i=0,$ $1,\ldots, m_2-1$,
\begin{align*}
\Psi_{i+1}(t)&=\Psi_{i+1}(0)-\int_t^0 e^{(\lambda_i^+-\lambda_{i+1}^+)r}\Psi_i(r)\,dr\\=&\left(\sum_{j=0}^{m_2}\scrP_{i+1,j}(\underline{\lambda})\scrI_{d\pi(Y^j)}^{\chi}(0)\right)\\&\qquad\qquad\qquad-\int_t^0e^{(\lambda_i^+-\lambda_{i+1}^+)r}\left( \int_{-\infty}^0\psi(s)\Phi_i(r,s)\,ds+\sum_{j=0}^{m_2}\Pi_{i,j}(r,\underline{\lambda})\scrI_{d\pi(Y^j)\vecv}^{\chi}(0)\right)\,dr.
\end{align*}
This yields
\begin{equation*}
\Phi_{i+1}(t,s)=-\int_t^0 e^{(\lambda_i^+-\lambda_{i+1}^+)r}\Phi_i(r,s)\,dr,
\end{equation*}
and
\begin{equation}\label{PIDEF}
\Pi_{i+1,j}(t,\underline{\lambda})=\scrP_{i+1,j}(\underline{\lambda})-\int_t^0 e^{(\lambda_i^+-\lambda_{i+1}^+)r} \Pi_{i,j}(r,\underline{\lambda})\,dr.
\end{equation}
Note that all the $\scrP_{i,j}$ have degrees less than or equal to $m_2$; there is thus an absolute constant $D$ such that for all $i,j$, $|\scrP_{i,j}(\underline{\lambda})|\leq D(1+|\underline{\lambda}|_{\infty}^{m_2})$. By induction, we have
\begin{equation}\label{Pibd}
|\Pi_{i,j}(t,\underline{\lambda})|\ll (1+|\underline{\lambda}|_{\infty}^{m_2})(1+|t|^{i}).
\end{equation}
In a similar fashion, 
\begin{equation*}
|\Phi_i(t,s)|\leq |t|^{i}|s|^{m_1-1}e^{\beta_{\gamma} (t-s)}e^{-\Re(\lambda_i^+) t},
\end{equation*}
where the bound for $i=0$ holds by \eqref{Phi0bd} (recall that $\lambda_0^+=0$), and for $i=1$, $2,\ldots,\,m_2$ by induction. Since $\scrI_{\vecv}^{\chi}(t)=e^{\lambda_{m_2}^+t}\Psi_{m_2}(t)$, these bounds establish \textit{i)}, \textit{ii)}, and \textit{iii)} when $m_1>0$. 

It remains to consider the case $m_1=0$. We then have $\Psi_0(t)=\psi(t)$,
hence
\begin{align*}
\Psi_1(t)=&\Psi_1(0)-\int_t^0 e^{-\lambda_1^+s}\psi(s)\,ds\\&=\left(\sum_{j=0}^{m_2}\scrP_{1,j}(\underline{\lambda})\scrI_{d\pi(Y^j)\vecv}^{\chi}(0)\right)+\int_{-\infty}^0 \phi_{1}(t,s)\psi(s)\,ds,
\end{align*}
where 
\begin{equation*}
\phi_1(t,s)=\begin{cases} 0\qquad&\mathrm{if}\, s<t\\ -e^{-\lambda_1^+s}\qquad&\mathrm{if}\,s\geq t. \end{cases}
\end{equation*}
We now proceed as in the case $m_1>0$: making repeated use of the fundamental theorem of calculus, we write, for $i=1$, $2,\ldots$, $m_2$,
\begin{equation*}
\Psi_i(t)=\int_{-\infty}^0\psi(s)\phi_i(t,s)\,ds+\sum_{j=0}^{m_2}\Pi_{i,j}(t,\underline{\lambda})\scrI_{d\pi(Y^j)\vecv}^{\chi}(0),
\end{equation*}
where $\Pi_{1,j}(t,\underline{\lambda})=\scrP_{1,j}(\underline{\lambda})$, and $\Pi_{i,j}(t,\underline{\lambda})$ and $\phi_i(t,s)$ for $i=2$, $3,\ldots, m_2$ are given recursively by \eqref{PIDEF} and
\begin{equation*}
\phi_{i+1}(t,s)=-\int_t^0e^{(\lambda_{i}^+-\lambda_{i+1}^+)r}\phi_i(r,s)\,dr,
\end{equation*} 
respectively. Induction now gives
\begin{equation*}
|\phi_{i}(t,s)|\leq \begin{cases} 0\qquad&\mathrm{if}\, s<t\\ (s-t)^{i-1}e^{-\Re(\lambda_i^+)s}\qquad&\mathrm{if}\,s\geq t. \end{cases}
\end{equation*}
Observing that $\scrI_{\vecv}^{\chi}(t)=e^{\lambda_{m_2}^+t}\Psi_{m_2}(t)$, as well as the fact that the bound \eqref{Pibd} still holds, concludes the proof. 
\end{proof}

We now use Proposition \ref{BurgerProp1} to demonstrate how the bound on $\psi$ can be ``lifted'' to $\scrI_{\vecv}^{\chi}$:

\begin{cor}\label{fbd}
Suppose that $\scrI_{\vecv}^{\chi}(t)$, $\psi(t)$ and $\beta_{\gamma}$ satisfy the assumptions of Lemma \ref{gint lem}. There then exists $n_0\in\NN$ such that for all $t\leq 0$, $\epsilon>0$,
\begin{equation*}
\|\scrI_{\vecv}^{\chi}(t)\|\ll\left( \|\chi\|_{L^1(U^+)}\|\vecv\|_{\scrS^{n_0}(\scrH)}+C_{\psi,\gamma}\right) e^{(\beta_{\gamma}-\epsilon) t},
\end{equation*}
where the implied constant depends on $G$, $g_{\RR}$, $\epsilon$, and $\gamma-\beta_{\gamma}$.
\end{cor}
\begin{proof}
By Proposition \ref{BurgerProp1},
\begin{equation*}
\|\scrI_{\vecv}^{\chi}(t)\| \leq \int_{-\infty}^0 |F(\underline{\lambda},\beta_{\gamma},t,s)|\| \psi(s)\|\,ds +\sum_{i=0}^{W} |F_{i}(\underline{\lambda},\beta_{\gamma},t)|\|\scrI_{d\pi(Y^{i})\vecv}^{\chi}(0)\|.
\end{equation*}
Using $\|\psi(s)\|\leq C_{\psi,\gamma} e^{\gamma s}$, the bound on $F_{i}(\underline{\lambda},\beta_{\gamma},t)$ from Proposition \ref{BurgerProp1} \textit{iii)}, and
\begin{equation*}
\|\scrI_{d\pi(Y^{i})\vecv}^{\chi}(0)\|\leq \int_{U^+} |\chi(u)| \| \pi(u)d\pi(Y^{i})\vecv\|\,du\ll \|\chi\|_{L^1(U^+)}\|\vecv\|_{\scrS^{i}(\scrH)}
\end{equation*}
gives
\begin{equation*}
\|\scrI_{\vecv}^{\chi}(t)\|\ll (1+|\underline{\lambda}|_{\infty}^{W})(1+|t|^W)e^{\beta_{\gamma} t}\|\chi\|_{L^1(U^+)}\|\vecv\|_{\scrS^W(\scrH)}+C_{\psi,\gamma}\int_{-\infty}^0 |F(\underline{\lambda},\beta_{\gamma},t,s)|e^{\gamma s}\,ds.
\end{equation*}
Now, by Proposition \ref{BurgerProp1} \textit{ii)}, 
\begin{align}\label{F(t,s)Int}
\int_{-\infty}^0 |F(\underline{\lambda},\beta_{\gamma},t,s)|e^{\gamma s}\,ds &\ll |t|^{W-1-m_0}\int_{-\infty}^0 |s|^{m_0}e^{\beta_{\gamma}(t-s)}e^{\gamma s} \,ds \\\notag&\quad \ll (1+|t|^{W}) e^{\beta_{\gamma} t}.
\end{align}
We now have
\begin{equation*}
\|\scrI_{\vecv}^{\chi}(t)\|\ll \left( \|\chi\|_{L^1(U^+)}(1+|\underline{\lambda}|_{\infty}^{W})\|\vecv\|_{\scrS^W(\scrH)}+C_{\psi,\gamma}\right) e^{(\beta_{\gamma}-\epsilon) t}.
\end{equation*}
The definition of the $\lambda_i$s is now used to bound $|\underline{\lambda}|_{\infty}$: recall that the $\lambda_i$s are the solutions to the polynomial equation
\begin{equation*}
z^W+d\pi(Z_{W-1})z^{W-1}+\ldots+d\pi(Z_1)z+d\pi(Z_0)=0,
\end{equation*}
so by Cauchy's bound, $|\underline{\lambda}|_{\infty}\leq 1+ \max_i |d\pi(Z_i)|$. This gives 
\begin{align}\label{lambdaSob}
|\underline{\lambda}|_{\infty}^{W}\|\vecv\|_{\scrS^W(\scrH)}\leq& \|\vecv\|_{\scrS^W(\scrH)}+ \max_i |d\pi(Z_i^W)|\|\vecv\|_{\scrS^W(\scrH)}\\\notag &=\|\vecv\|_{\scrS^W(\scrH)}+\max_i \|d\pi(Z_i^W)\vecv\|_{\scrS^W(\scrH)}\ll \|\vecv\|_{\scrS^{n_0}(\scrH)}
\end{align}
for some $n_0\in \NN$.
\end{proof}

We will now make use of the smoothness of $\chi$ to repeatedly apply Proposition  \ref{BurgerProp1} and Corollary \ref{fbd} to obtain the final version of ``Burger's formula'' that will be used in the proof of Theorem \ref{maintheorem1}. In preparation of the proof, we first introduce some more notation.

Recall that we have fixed a basis $X_1,\,\ldots, X_d$ of $\fu^+$, and elements $U_1,\,\ldots, U_d$ in $\Ug$. We introduce the following multi-index notation: let $\mathbf{j}=(j_1,j_2,\ldots, j_{l})\in  \bigcup_{i=1}^{\infty} \lbrace 1,\ldots, d\rbrace^i$, and define $|\mathbf{j}|=l$ (we call $|\mathbf{j}|$ the \emph{order} of $\mathbf{j}$). For a multi-index $\mathbf{j}$, we define $X_{\mathbf{j}}=X_{j_1}X_{j_2}\ldots X_{j_l}$, and $U_{\mathbf{j}} =U_{j_1}U_{j_2}\ldots U_{j_l}$. We also let $\bI_k$ and $\bI_{\leq k}$ denote the sets of multi-indices of order $k$, and order less than or equal to $k$, respectively. We also assume that $\bI_{\leq k}$ includes $\bI_0$, the set of multi-indices of order zero (i.e.\ if $\bk\in\bI_0$, then $d\pi(U_{\bk})\vecv=\vecv$ and $X_{\bk}\chi=\chi$). Note that $\bI_0$ is in fact a singleton set, consisting only of the ``empty string'', which we denote $\mathbf{0}$. We also let $\bI=\bigcup_{k=0}^{\infty}\bI_k$ be the set of all multi-indices.
 
This multi-index notation can be used to explicitly express the Sobolev norm $\|\cdot\|_{W^{k,p}(U^+)}$; for $\chi\in C_c^{\infty}(U^+)$,
\begin{equation*}
\|\chi\|_{W^{k,p}(U^+)}=\sum_{\mathbf{\bk}\in\bI_{\leq k}} \| X_{\mathbf{k}}\chi \|_{L^p(U^+)}.
\end{equation*}
We let $n_0$ be such that for all $j$, $U_j\in\scrU^{n_0}(\mathfrak{g}_{\CC})$ (without loss of generality, we may take $n_0$ to be the same integer as in Corollary \ref{fbd}). From \eqref{MAINODE}, we have that for any multi-index $\mathbf{k}$, 
\begin{equation}\label{multiindexdiff}
\left( \prod_{i=1}^W (\sfrac{d}{dt}-\lambda_i)\right)\scrI_{d\pi(U_{\bk})\vecv}^{X_{\bk}\chi}(t)=-\sum_{j=1}^d e^{\alpha_jt}\scrI_{d\pi(U_jU_{\bk})\vecv}^{X_jX_{\bk}\chi}(t),
\end{equation}
and, moreover, the following bound holds:
\begin{equation}\label{RHDBND}
\left\|-\sum_{j=1}^d e^{\alpha_jt}\scrI_{d\pi(U_jU_{\bk})\vecv}^{X_jX_{\bk}\chi}(t)\right\| \ll \|\chi\|_{W^{|\bk|+1,1}(U^+)}\|\vecv\|_{\scrS^{|\bk|n_0+1}(\scrH)} e^{\left(\min_j\alpha_j\right) t}.
\end{equation}
We may now state the main result of this section:

\begin{prop}\label{BurgerProp2}
Let $\alpha\in (0,\min_j \alpha_j]$, and define
\begin{equation*}
k_0=\begin{cases}
1\qquad &\mathrm{if\,\,}\alpha \geq 2\eta \\
2\qquad &\mathrm{if\,\,}\eta<\alpha < 2\eta \\
\lfloor \sfrac{2\eta}{\alpha}\rfloor\qquad &\mathrm{if\,\,}\alpha\leq\eta.  
\end{cases}
\end{equation*}
There then exist, for given $(\pi,\scrH)$ and $g_{\RR}$, functions $\scrC_{\bk}:\RR_{\leq0}\times \RR_{\leq 0}\rightarrow\CC$ ($\,\bk\in\bI_{k_0}$) and $\scrD_{\bj,i}:\RR_{\leq0}\rightarrow\CC$ ($\,\bj\in \bI_{\leq k_0-1}$ and $i\in\lbrace 0,\ldots, W\rbrace $) such that for all $\vecv\in\scrH^{\infty}$, $\chi\in C_c^{\infty}(U^+)$, and $t\leq0$:
\begin{equation}\label{MAINBURGERIDENT}
\scrI_{\vecv}^{\chi}(t)=\sum_{\bk\in\bI_{k_0}}\int_{-\infty}^0 \scrC_{\bk}(t,s)\scrI_{d\pi(U_{\bk})\vecv}^{X_{\bk}\chi}(s)\,ds+\sum_{\bj\in\bI_{\leq k_0-1}}\sum_{i=0}^{W} \scrD_{\bj,i}(t)\scrI_{d\pi(Y^{i}U_{\bj})\vecv}^{X_{\bj}\chi}(0).
\end{equation}
Furthermore, the following hold:\vspace{2pt}
\begin{enumerate}
\item $|\scrC_{\bk}(t,s)|\ll(1+|t|^{W})e^{\eta t+\frac{\alpha}{5}s}$, where the implied constant depends only on $G$, $\eta$, and $\alpha$.
\item There exists $n_0\in\NN$, depending only on $G$, such that for any $m\in\NN$,
\begin{equation*}
|\scrD_{\bj,i}(t)|\ll(1+|t|^{W}) e^{\eta t}\inf_{\vecw\in\scrH^{\infty}\setminus\lbrace \mathbf{0} \rbrace} \frac{\|\vecw\|_{\scrS^{n_0+m}(\scrH)}}{\|\vecw\|_{\scrS^{m}(\scrH)}},
\end{equation*}
where the implied constant depends only on $G$, $g_{\RR}$, $\eta$, $\alpha$, and $m$.
\end{enumerate}
\end{prop}
\begin{proof}
By \eqref{MAINODE}, and the fact that $\|-\sum_{j=1}^d e^{\alpha_jt}\scrI_{d\pi(U_j)\vecv}^{X_j\chi}(t)\|\leq C e^{\alpha t}$ for some $C>0$ and all $t\leq 0$, if $\alpha\geq 2\eta$, we may apply Proposition \ref{BurgerProp1} with $\psi(t)=\sum_{j=1}^d -e^{\alpha_jt}\scrI_{d\pi(U_j)\vecv}^{X_j\chi}(t)$, $\gamma=\alpha$, and $\beta_{\gamma}=\eta$. The bounds in Proposition \ref{BurgerProp1} \textit{ii)} and \textit{iii)}, and \eqref{lambdaSob} then suffice to prove the proposition; in particular, \eqref{lambdaSob} provides the bound $(1+ |\underline{\lambda}|_{\infty}^W)\|\vecw\|_{\scrS^m(\scrH)} \ll \|\vecw\|_{\scrS^{m+n_0}(\scrH)}$ (for all $m\in\NN$ and $\vecw\in\scrH^{\infty}$), which, together with  Proposition \ref{BurgerProp1} \textit{iii)}, proves \textit{(2)}.

If $\eta < \alpha < 2\eta$, we will apply Proposition \ref{BurgerProp1} twice: by \eqref{multiindexdiff} and \eqref{RHDBND}, Proposition \ref{BurgerProp1} may be applied to $\scrI_{d\pi(U_j)\vecv}^{X_j\chi}(t)$ (for each $j=1,\ldots,d$) with $\psi(t)=-\sum_{k=1}^d e^{\alpha_k t}\scrI_{d\pi(U_kU_{j})\vecv}^{X_kX_{j}\chi}(t)$, $\gamma=\alpha$, and $\beta_{\gamma}=\eta-\frac{\alpha}{4}$. This gives
\begin{equation*}
\scrI_{d\pi(U_j)\vecv}^{X_j\chi}(t)\!=\!\int_{-\infty}^0 \!\!F(\underline{\lambda},\eta-\sfrac{\alpha}{4},t,s) \!\left(-\sum_{k=1}^d e^{\alpha_ks}\scrI_{d\pi(U_k U_j)\vecv}^{X_k X_j\chi}(s)\!\right)\!ds +\sum_{i=0}^{W} F_i(\underline{\lambda},\eta-\sfrac{\alpha}{4},t)\scrI_{d\pi(Y^{i}U_j)\vecv}^{X_j\chi}(0).
\end{equation*}
By letting $\epsilon=\frac{\alpha}{4}$ in Corollary \ref{fbd}, $\|\scrI_{d\pi(U_j)\vecv}^{X_j\chi}(t)\|\ll e^{(\eta-\frac{\alpha}{2})t}$ for all $t\leq 0$  and $j=1,\ldots,d$. This now gives the bound $ \|-\sum_{j=1}^d e^{\alpha_jt}\scrI_{d\pi(U_j)\vecv}^{X_j\chi}(t)\|\ll e^{(\eta+\frac{\alpha}{2}) t}$, which enables us to apply Proposition \ref{BurgerProp1} to $\scrI_{\vecv}^{\chi}(t)$ with $\psi(t)=-\sum_{j=1}^d e^{\alpha_jt}\scrI_{d\pi(U_j)\vecv}^{X_j\chi}(t)$, $\gamma=\eta+\frac{\alpha}{2}$, and $\beta_{\gamma}=\eta $, hence
\begin{equation*}
\scrI_{\vecv}^{\chi}(t)=\int_{-\infty}^0 F(\underline{\lambda},\eta,t,s) \left(-\sum_{j=1}^d e^{\alpha_js}\scrI_{d\pi(U_j)\vecv}^{X_j\chi}(s)\right)\,ds +\sum_{i=0}^{W} F_{i}(\underline{\lambda},\eta,t)\scrI_{d\pi(Y^{i})\vecv}^{\chi}(0).
\end{equation*}
Our previous expression for $\scrI_{d\pi(U_j)\vecv}^{X_j\chi}$ is now substituted into this, giving
\begin{align*}
\scrI_{\vecv}^{\chi}(t)= \sum_{j=1}^d&\sum_{k=1}^d \int_{-\infty}^0 \left( \int_{-\infty}^0 F(\underline{\lambda},\eta,t,r) F(\underline{\lambda},\eta-\sfrac{\alpha}{4},r,s)e^{\alpha_j r+\alpha_k s}\,dr\right) \scrI_{d\pi(U_kU_j)\vecv}^{X_kX_j\chi}(s)\,ds\\-&\sum_{j=1}^d \sum_{i=0}^{W} \left(\int_{-\infty}^0 F(\underline{\lambda},\eta,t,s) F_{i}(\underline{\lambda},\eta-\sfrac{\alpha}{4},s)e^{\alpha_j s}\,ds\right)\scrI_{d\pi(Y^{i} U_j)\vecv}^{X_j\chi}(0)\\&\qquad +\sum_{i=0}^{W} F_{i}(\underline{\lambda},\eta,t)\scrI_{d\pi(Y^{i})\vecv}^{\chi}(0).
\end{align*}
As previously (the case $\alpha\geq 2\eta$), the bounds from Proposition \ref{BurgerProp1} and \eqref{lambdaSob} are now used to bound the terms $\int_{-\infty}^0 F(\underline{\lambda},\eta,t,r)  F(\underline{\lambda},\eta-\frac{\alpha}{4},r,s)\,e^{\alpha_j r+\alpha_k s}\,dr$, $\int_{-\infty}^0 F(\underline{\lambda},\eta,t,s) F_{i}(\underline{\lambda},\eta-\frac{\alpha}{4},s)e^{\alpha_j s}\,ds$, and $F_{i}(\underline{\lambda},\eta,t)$ as in the statement of the proposition. 

In the case $\alpha\leq \eta$, a similar argument is used, though now involving higher order multi-indices and backwards induction on the order of these; recall that $k_0=\lfloor \frac{2\eta}{\alpha}\rfloor$, and thus $k_0\leq \frac{2\eta}{\alpha}<k_0+1$. We now define
\begin{equation*}
\epsilon=\frac{\alpha (k_0+1)-2\eta}{4(k_0-1)}.
\end{equation*}
Note that $\epsilon\in(0,\frac{\alpha}{4}]$, and $\eta-\alpha < (k_0-1)(\frac{\alpha}{2}-\epsilon)$. For $0\leq l\leq k_0-1$, define $\delta_{l}= l(\frac{\alpha}{2}-\epsilon)$. By \eqref{multiindexdiff}, we have that for any multi-index $\mathbf{k}'$ of order $k_0-1$, 
\begin{equation*}
\left( \prod_{i=1}^W (\sfrac{d}{dt}-\lambda_i)\right)\scrI_{d\pi(U_{\bk'})\vecv}^{X_{\bk'}\chi}(t)=-\sum_{j=1}^d e^{\alpha_jt}\scrI_{d\pi(U_jU_{\bk'})\vecv}^{X_jX_{\bk'}\chi}(t),
\end{equation*}
and by \eqref{RHDBND}, we have
\begin{equation*}
\left\|-\sum_{j=1}^d e^{\alpha_jt}\scrI_{d\pi(U_jU_{\bk'})\vecv}^{X_jX_{\bk'}\chi}(t)\right\| \ll_d \|\chi\|_{W^{k_0,1}(U^+)}\|\vecv\|_{\scrS^{(k_0-1)n_0+1}(\scrH)} e^{\alpha t}.
\end{equation*}
Proposition \ref{BurgerProp1} may thus be applied to $\scrI_{d\pi(U_{\bk'})\vecv}^{X_{\bk'}\chi}(t)$ with $\psi(t)=-\sum_{j=1}^d e^{\alpha_jt}\scrI_{d\pi(U_jU_{\bk'})\vecv}^{X_jX_{\bk'}\chi}(t)$, $\gamma=\alpha$, and $\beta_{\gamma}=\frac{\alpha}{2}$, giving
\begin{align*}
\scrI_{d\pi(U_{\bk'})\vecv}^{X_{\bk'}\chi}(t)=\sum_{j=1}^d \int_{-\infty}^0 -F(\underline{\lambda},\sfrac{\alpha}{2},t,s) e^{\alpha_js}\scrI_{d\pi(U_jU_{\bk'})\vecv}^{X_jX_{\bk'}\chi}(s)\,ds +\sum_{i=0}^{W} F_{i}(\underline{\lambda},\sfrac{\alpha}{2},t)\scrI_{d\pi(Y^{i}U_{\bk'})\vecv}^{X_{\bk'}\chi}(0).
\end{align*}
By Proposition \ref{BurgerProp1} \textit{ii)}, $|F(\underline{\lambda},\sfrac{\alpha}{2},t,s)e^{\alpha_js}|\ll e^{\delta_{1}t+(\frac{\alpha}{2}-\epsilon)s}$, and by Proposition \ref{BurgerProp1} \textit{iii)} and \eqref{lambdaSob}, $|F_{i}(\underline{\lambda},\frac{\alpha}{2},t)|\ll e^{\delta_{1} t}\inf_{\vecw\in\scrH^{\infty}\setminus\lbrace \mathbf{0}\rbrace}\frac{\|\vecw\|_{\scrS^{n_0+m}(\scrH)}}{\|\vecw\|_{\scrS^m(\scrH)}}$ (for any $m\in\NN$). We now use induction: assume that for any multi-index $\mathbf{i}$ of order $k_0-l$ (for some $l\in\lbrace 1,\ldots,k_0-1\rbrace$), there exist sets of $\CC$-valued functions $\lbrace\scrA_{\bi,\bj}\rbrace$ and $\lbrace\scrB_{\mathbf{i},\mathbf{l},i}\rbrace$ such that
\begin{align}\label{INDUCTIONHYP}
\scrI_{d\pi(U_{\mathbf{i}})\vecv}^{X_{\mathbf{i}}\chi}(t)=\sum_{\bj\in \bI_{l}} \int_{-\infty}^0 \scrA_{\bi,\bj}(t,s)\scrI_{d\pi(U_{\mathbf{j}}U_{\mathbf{i}})\vecv}^{X_{\mathbf{j}}X_{\mathbf{i}}\chi}(s)\,ds+\sum_{\bl\in\bI_{\leq l-1}}\sum_{i=0}^W \scrB_{\mathbf{i},\mathbf{l},i}(t)\scrI_{d\pi(Y^{i}U_{\bl}U_{\bi})\vecv}^{X_{\bl}X_{\bi}\chi}(0),
\end{align}
where 

\textit{i)} $|\scrA_{\bi,\bj}(t,s)|\ll_{G,\eta,\alpha}  e^{\delta_{l}t+(\frac{\alpha}{2}-\epsilon)s}$,

 \textit{ii)} $|\scrB_{\mathbf{i},\mathbf{l},i}(t)|\ll_{G,g_{\RR},\eta,\alpha}e^{\delta_{l}t}\inf_{\vecw\in\scrH^{\infty}\setminus\lbrace \mathbf{0}\rbrace} \frac{\|\vecw\|_{\scrS^m(\scrH)}}{\|\vecw\|_{\scrS^{n_0+m}(\scrH)}}\qquad \forall m\in\NN$. 

Under these assumptions, using the bound $\|\scrI_{d\pi(U_{\bj}U_{\bi})\vecv}^{X_{\bj}X_{\bi}\chi}(s)\|\ll 1$ in \eqref{INDUCTIONHYP} immediately gives 
\begin{equation}\label{Ibibd}
\| \scrI_{d\pi(U_{\mathbf{i}})\vecv}^{X_{\mathbf{i}}\chi}(t)\|\ll e^{\delta_{l} t}.
\end{equation}
We now select an arbitrary multi-index $\bk$ of order $k_0-(l+1)$. Since $\scrI_{d\pi(U_{\mathbf{k}})\vecv}^{X_{\mathbf{k}}\chi}(t)$ satisfies the differential equation \eqref{multiindexdiff}, the right-hand side of which consists of terms of the form $e^{\alpha_jt}\scrI_{d\pi(U_jU_{\bk})\vecv}^{X_jX_{\bk}\chi}(t)$, \eqref{Ibibd} may be used to bound the right-hand side of \eqref{multiindexdiff} by $\ll e^{(\delta_l+\alpha)t}$. This enables us to apply Proposition \ref{BurgerProp1} to $\scrI_{d\pi(U_{\bk})\vecv}^{X_{\bk}\chi}(t)$ with $\psi(t)=-\sum_{j=1}^de^{\alpha_jt}\scrI_{d\pi(U_jU_{\bk})\vecv}^{X_jX_{\bk}\chi}(t)$, $\gamma=\delta_{l}+\alpha$, and $\beta_{\gamma}=\delta_{l}+\frac{\alpha}{2}$ (indeed, note that $\delta_l+\frac{\alpha}{2}\leq (k_0-1)(\frac{\alpha}{2}-\epsilon)+\frac{\alpha}{2}<k_0\frac{\alpha}{2}\leq \eta$), giving
\begin{align}\label{bkI}
\scrI_{d\pi(U_{\bk})\vecv}^{X_{\bk}\chi}(t)=-\sum_{j=1}^d \int_{-\infty}^0 &F(\underline{\lambda},\delta_{l}+\sfrac{\alpha}{2},t,s) e^{\alpha_js}\scrI_{d\pi(U_jU_{\bk})\vecv}^{X_jX_{\bk}\chi}(s)\,ds\\\notag & +\sum_{i=0}^{W} F_{i}(\underline{\lambda},\delta_{l}+\sfrac{\alpha}{2},t)\scrI_{d\pi(Y^{i}U_{\bk})\vecv}^{X_{\bk}\chi}(0). 
\end{align}
We may now apply Proposition \ref{BurgerProp1} \textit{iii)} and \eqref{lambdaSob} to get that for any $m\in\NN$ and $\vecw\in\scrH^{\infty}$, $|F_{i}(\underline{\lambda},\delta_{l}+\frac{\alpha}{2},t)| \ll  e^{\delta_{l+1} t} \inf_{\vecw\in\scrH^{\infty}\setminus\lbrace 0 \rbrace}\frac{\|\vecw\|_{\scrS^{n_0+m}(\scrH)}}{\|\vecw\|_{\scrS^{m}(\scrH)}} $ (since $\delta_{l+1}=\delta_{l}+\frac{\alpha}{2}-\epsilon$). Let $\bj_j$ be the multi-index of order $k_0-l$ corresponding to $X_jX_{\bk}$ and $U_jU_{\bk}$. The induction hypothesis is then used once again: for each $\bj_j$, we have
\begin{align*}
\scrI_{d\pi(U_{\bj_j})\vecv}^{X_{\bj_j}\chi}(t)=\sum_{\bm\in \bI_{l}} \int_{-\infty}^0 \scrA_{\bj_j,\bm}(t,s)\scrI_{d\pi(U_{\bm}U_{\bj_j})\vecv}^{X_{\bm}X_{\bj_j}\chi}(s)\,ds+\sum_{\bl\in\bI_{\leq l-1}}\sum_{i=0}^W \scrB_{\mathbf{j}_j,\mathbf{l},i}(t)\scrI_{d\pi(Y^{i}U_{\bl}U_{\bj_j})\vecv}^{X_{\bl}X_{\bj_j}\chi}(0),
\end{align*}
where the sets of functions $\lbrace\scrA_{\bj_j,\bm}\rbrace$ and $\lbrace\scrB_{\mathbf{j}_j,\mathbf{l},i}\rbrace$ satisfy the bounds of the induction hypothesis. This is entered into \eqref{bkI}, giving
\begin{align*}
\scrI_{d\pi(U_{\bk})\vecv}^{X_{\bk}\chi}(t)& =\sum_{j=1}^d \sum_{\bm\in \bI_{l}}  \int_{-\infty}^0\left( \int_{-\infty}^0  -F(\underline{\lambda},\delta_{l}+\sfrac{\alpha}{2},t,r) e^{\alpha_jr} \scrA_{\bj_j,\bm}(r,s)\,dr\right)\scrI_{d\pi(U_{\bm}U_jU_{\bk})\vecv}^{X_{\bm}X_jX_{\bk}\chi}(s)\,ds\\&\quad+ \sum_{j=1}^d \sum_{\bl\in\bI_{\leq l-1}}\sum_{i=0}^W \left(\int_{-\infty}^0 -F(\underline{\lambda},\delta_{l}+\sfrac{\alpha}{2},t,s) e^{\alpha_js} \scrB_{\mathbf{j}_j,\mathbf{l},i}(s)\,ds\right)\scrI_{d\pi(Y^{i}U_{\bl}U_jU_{\bk})\vecv}^{X_{\bl}X_jX_{\bk}\chi}(0)\\&\qquad+\sum_{i=0}^{W} F_{i}(\underline{\lambda},\delta_{l}+\sfrac{\alpha}{2},t)\scrI_{d\pi(Y^{i}U_{\bk})\vecv}^{X_{\bk}\chi}(0).  
\end{align*}
Collecting terms and using the induction hypothesis, the bounds from Proposition \ref{BurgerProp1}, and \eqref{lambdaSob} complete the induction: \eqref{INDUCTIONHYP}, \textit{i}, and \textit{ii)} are thus valid for any $\bk$ of order $k_0-l$, for all $l$, $1\leq l\leq k_0$. In particular, for $l=k_0$, \eqref{INDUCTIONHYP}, \textit{i)} and \textit{ii)} give bounds of the same form as those in Proposition \ref{BurgerProp2}, though not as sharp. In order to obtain the desired bounds, we carry out the last step (viz., $l=k_0-1$) of the previous induction with the single modification that we use $\beta_{\gamma}=\eta$ in place of $\beta_{\gamma}=\delta_{k_0-1}+\frac{\alpha}{2}$; we thus apply Proposition \ref{BurgerProp1} to $\scrI_{\vecv}^{\chi}(t)$ with $\psi(t)=-\sum_{j=1}^de^{\alpha_j t}\scrI_{d\pi(U_j)\vecv}^{X_j\chi}(t)$, $\gamma=\delta_{k_0-1}+\alpha$ and $\beta_{\gamma}=\eta$. The assumptions of Proposition \ref{BurgerProp1} are still fulfilled; indeed, the only condition that needs a new verification is $\beta_{\gamma}<\eta$, i.e.\ $\eta< \delta_{k_0-1}+\alpha$, and this is an easy consequence of our choice of $\epsilon$. Arguing as before, we obtain
\begin{align*}
\scrI_{\vecv}^{\chi}(t)& =\sum_{j=1}^d \sum_{\bm\in \bI_{k_0-1}}  \int_{-\infty}^0\left( \int_{-\infty}^0  -F(\underline{\lambda},\eta,t,r) e^{\alpha_jr} \scrA_{\bj_j,\bm}(r,s)\,dr\right)\scrI_{d\pi(U_{\bm}U_j)\vecv}^{X_{\bm}X_j\chi}(s)\,ds\\&\quad+ \sum_{j=1}^d \sum_{\bl\in\bI_{\leq k_0-2}}\sum_{i=0}^W \left(\int_{-\infty}^0 -F(\underline{\lambda},\eta,t,s) e^{\alpha_js} \scrB_{\mathbf{j}_j,\mathbf{l},i}(s)\,ds\right)\scrI_{d\pi(Y^{i}U_{\bl}U_j)\vecv}^{X_{\bl}X_j\chi}(0)\\&\qquad+\sum_{i=0}^{W} F_{i}(\underline{\lambda},\eta,t)\scrI_{d\pi(Y^{i})\vecv}^{\chi}(0),
\end{align*}
where $\mathbf{j}_j=(j)$, $j=1,\ldots,d$. Now arguing in the same manner as previously gives the bounds stated in the proposition. 
\end{proof}
\begin{remark}
The functions $\scrA_{\bi,\bj}$, $\scrB_{\bi,\bl,i}$ in the proof can be made completely explicit; they consist of integrals of products of the functions $F$, $F_i$ from Proposition \ref{BurgerProp1}. The uniformity of the bounds in Proposition \ref{BurgerProp2} with respect to the representation $(\pi,\scrH)$  allows us to use the identity \eqref{MAINBURGERIDENT} in all the representations in the decomposition of $(\rho,V)$ into irreducibles. This will be important in the proof of Theorem \ref{maintheorem1}, where we essentially ``integrate'' \eqref{MAINBURGERIDENT} over the direct integral decomposition of $(\rho,V)$. Furthermore, the dependency on $g_{\RR}$ in \textit{(2)} can be quantified, namely it comes from the bound \eqref{lambdaSob}, i.e.\ $1+|\underline{\lambda}|^W\ll_{G,g_{\RR},m} \inf_{\vecw\in\scrH^{\infty}\setminus\lbrace \mathbf{0} \rbrace} \frac{\|\vecw\|_{\scrS^{n_0+m}(\scrH)}}{\|\vecw\|_{\scrS^{m}(\scrH)}}$ (the bound in \textit{(2)} is stated in this manner to ensure that the implied constant does not depend on $(\pi,\scrH)$). This will be of importance when proving Theorem \ref{maintheorem2} (in Section \ref{UNIFORMSEC}), where we will let $g_{\RR}$ vary within the positive Weyl chamber $\overline{A^+}$.
\end{remark}
\section{Sobolev Inequalities}\label{SOBSEC}
While the results of Section \ref{BURGERSEC} allow us to draw various conclusions regarding convergence in $L^2(\GaG)$ of translates of averages of functions, in order to make pointwise statements for non-uniform $\Gamma$ we use an automorphic Sobolev inequality from \cite[Appendix B]{Bern}. Before introducing this inequality, we recall some facts regarding the reduction theory of $\GaG$.
\subsection{Reduction Theory}\label{ReducSec}
Most of the material of this section is drawn from \cite[Chapter 2]{OsborneWarner}. Recall that we have assumed that $\Gamma$ fulfils the assumption of Langlands (cf.\ \cite[Chapter 2, p. 16]{Langlands}, also \cite[pp. 62-63]{OsborneWarner}). Before stating the main result regarding the reduction theory, we introduce the concept of \emph{split components}. Let $P=NAM$ be a parabolic subgroup with corresponding Lie algebra decomposition $\fn\oplus\fa\oplus\fm$. Given a subalgebra $\fa'\subset\fa$, for $\lambda\in \fa'^*$ we define
\begin{equation*}
\fn_{\lambda}=\lbrace X\in \fn\,:\, \ad_H(X)=\lambda(H)X\quad\forall H\in\fa'\rbrace.
\end{equation*}
Letting $\Sigma^+(\fg,\fa')=\lbrace \lambda|_{\fa'}\,:\, \lambda\in\Sigma^+(\fg,\fa)\rbrace$, we then have that
\begin{equation*}
\fn=\bigoplus_{\lambda\in \Sigma^+(\fg,\fa')} \fn_\lambda.
\end{equation*}
Denoting the orthogonal complement of $\fa'$ in $\fa$ with respect to the Killing form by $\fa'^{\perp}$, $\fa'$ and its corresponding analytic subgroup $\scrA$ are both called \emph{split components of $P$} if
\begin{equation*}
\tr(\ad H|_{\fn_{\lambda}})=0\qquad\forall H\in\fa'^{\perp},\,\lambda\in\Sigma^+(\fg,\fa'),
\end{equation*}
cf.\ \cite[pp. 3-5]{Langlands} and \cite[pp. 30-32]{OsborneWarner}. One can show that if $\fa'$ is a split component, then $0\not\in\Sigma^+(\fg,\fa')$, as well as
\begin{equation*}
\rho_{\fa}|_{\fa'}=\rho_{\fa'}:=\sfrac{1}{2}\sum_{\lambda\in \Sigma^+(\fg,\fa')} \dim\fn_{\lambda}\lambda.
\end{equation*}
Moreover, there is a unique set $\Sigma_0^+(\fg,\fa')\subset\Sigma^+(\fg,\fa')$  of simple roots (cf.\ \cite[p. 8]{Langlands} and \cite[p. 34]{OsborneWarner}) for $\fa'$; of importance to us are the facts that $\Sigma_0^+(\fg,\fa')$ is a basis for $\fa'^*$, and every element of $\Sigma^+(\fg,\fa')$ may be written as a linear combination  $\sum_{\lambda\in \Sigma_0^+(\fg,\fa')} m_{\lambda}\lambda$, with all $m_{\lambda}\in \ZZ_{\geq 0}$.
 
For a split component $\fa'$, let $\scrA^{\perp}$ be the analytic subgroup of $G$ corresponding to the subalgebra $\fa'^{\perp}$. We now define a subgroup $S\subset P$ by $S=NM\scrA^{\perp}$. The pair $(P,S)$ is called a \emph{split parabolic pair}, and the triple $(P,S;\scrA)$ a \emph{split parabolic pair with split component $\scrA$}. We now define a few important subsets: for an arbitrary real number $\tau$, let
\begin{align*}
\scrA_{\tau}&=\lbrace \exp(H)\,:\, H\in\fa',\, \lambda(H)\geq \tau\quad\forall\lambda\in\Sigma_0^+(\fg,\fa')\rbrace\\
\scrA^{\tau}&=\lbrace \exp(H)\,:\, H\in\fa',\, |\lambda(H)|\leq \tau\quad\forall\lambda\in\Sigma_0^+(\fg,\fa')\rbrace.
\end{align*}
For a compact subset $\Omega\subset S$, define 
\begin{equation}\label{Omegatau}
\Omega_{\tau}=\bigcup_{a\in\scrA_{\tau}} a^{-1}\Omega a.
\end{equation}
Observe that $\scrA^{\tau}$ is a compact subset of $\scrA$, and $\Omega_{\tau}$ is a relatively compact subset of $S$. We now define a \emph{Siegel set} $\fG_{\Omega,\tau}$ with respect to $(P,S;\scrA)$ by
\begin{equation*}
\fG_{\Omega,\tau} =\Omega\scrA_{\tau}K.
\end{equation*}
A split parabolic pair $(P,S)$ is said to be $\Gamma$-\emph{percuspidal} if 
\textit{i)} $\Gamma\cap P\subset S$ \textit{ii)} $\Gamma\cap N$ is cocompact in $N$ \textit{iii)} $\Gamma\cap S$ is cocompact in $S$. We can now state the result from reduction theory that will be needed: combining \cite[Theorem 2.11]{OsborneWarner} and the assumption on $\Gamma$ (cf.\ \cite[pp. 62-63]{OsborneWarner}), we have:
\begin{prop}\label{REDUCTIONPROP}
There exist $\tau_0\in\RR$, a standard parabolic split pair with split component $(P,S;\scrA)$, and elements $g_1,\,g_2\,\ldots, g_{\kappa} \in G$ such that $(P_i,S_i;\scrA_i)=(g_iPg_i^{-1},g_iSg_i^{-1};g_i\scrA g_i^{-1})$ are $\Gamma$-percuspidal, and
\begin{equation*}
G=\bigcup_{i=1}^{\kappa} \Gamma\fG_i,
\end{equation*} 
where the $\fG_i= \Omega_i\scrA_{i,\tau_0}K$ are Siegel sets with respect to $(P_i,S_i;\scrA_i)$. Furthermore, the set of all $\gamma \in \Gamma$ such that $\gamma\left(\bigcup_{i=1}^{\kappa} \fG_i \right)\cap \left(\bigcup_{i=1}^{\kappa} \fG_i\right)\neq\emptyset$ is finite.
\end{prop}
\begin{remark}
In the case that $\Gamma$ is an arithmetic lattice in an algebraic group, the reduction theory of Borel and Harish-Chandra \cite{BorelHarish} gives more: we may take $P$ as the real points of a minimal $\QQ$-parabolic subgroup of $G_{\QQ}$, $\scrA$ as a maximal $\QQ$-split torus in $A$, and the $g_i\in G_{\QQ}$. If $\rank(G)=1$, then the reduction theory is provided by Garland and Raghunathan \cite{GarlandRaghunathan}. (cf.\ \cite[pp. 14-18]{OsborneWarner}).
\end{remark}
Throughout the remainder of Section \ref{SOBSEC}, we let $P_i$, $S_i$, $\scrA_i$, and $\fG_i$ ($i=1,\ldots,\kappa$) be as in Proposition \ref{REDUCTIONPROP}. We also use the notation $N_i=g_i N g_i^{-1}$, $A_i=g_i A g_i^{-1}$, $M_i=g_i M g_i^{-1}$, $N_i^-=g_i N^- g_i^{-1}$. Finally, the Lie algebra of $A_i$ is denoted $\fa_i$, and the Lie algebra of $\scrA_i$ is denoted $\fa_i'$.
\subsection{Sobolev Inequalities}\label{SOBINEQSEC}
We now follow \cite[Appendix B]{Bern} and state the previously mentioned automorphic Sobolev inequality. For a compact, symmetric neighbourhood $B$ of $e$ in $G$, we define a function $w_B$ on $\GaG$ by letting $w_B(x)$ ($x\in\GaG$) be the operator norm of the mapping from $L^2(\GaG)$ to $L^2(B)$ given by $f(\cdot)\mapsto f(x\cdot)$, i.e.\
\begin{equation*}
w_B(x)=\inf\bigg\lbrace c\in \RR_{>0}\,:\, \bigg(\int_B |f( x g)|^2\,dg\bigg)^{1/2}\leq c\|f\|_{L^2(\GaG)}\quad\forall f\in L^2(\GaG)\bigg\rbrace,
\end{equation*}
or, equivalently,
\begin{equation*}
w_B(x) = \sup_{\substack{f\in L^2(\GaG)\\ f\neq 0}} \frac{\bigg(\int_B |f( x g)|^2\,dg\bigg)^{1/2}}{\|f\|_{L^2(\GaG)}}.
\end{equation*}
From \cite[Proposition B.2]{Bern} we have the following:
\begin{prop}\label{sobbnd}
For $m>\sfrac{\dim(G)}{2}$, we have
\begin{equation*}
|f(x)|\ll_{\Gamma,B} w_B(x)\|f\|_{\scrS^m(\GaG)}\qquad \forall f\in\scrS^m(\GaG),\, x\in\GaG.
\end{equation*} 
\end{prop}
The two properties of $w_{B}$ established in the next lemma will be used at multiple points throughout the remainder of this section:
\begin{lem}$ $\\\vspace*{-13pt}

\begin{enumerate}[i)]
\item
For any compact subset $\scrC$ of $G$ and any $x\in\GaG$, we have 
\begin{equation*}
w_B(xg)\ll_{B,\scrC} w_B(x)\qquad \forall x\in\GaG,\,g\in \scrC.
\end{equation*}
\item For all $g\in G$, $w_B(\Gamma g)^2\leq \#( \Gamma \cap g BB g^{-1})$.
\end{enumerate}
\end{lem}
\begin{proof}
Starting with \textit{i)}, assume that $g\in\scrC$. Let $g_1,\ldots, g_r\in G$ be such that $\scrC B \scrC^{-1} \subset \bigcup_{i=1}^r Bg_i$. We then have that for any $f\in L^2(\GaG)$, $x\in\GaG$, and $g\in \scrC$,
\begin{align*}
\int_B |f( x gh)|^2\,dh=\int_{gBg^{-1}}|f(xhg)|^2\,dh&\leq \sum_{i=1}^r \int_{Bg_i} |f( x hg)|^2\,dh= \sum_{i=1}^r \int_{B} |\big(\rho(g_ig)f\big)( x h)|^2\,dh\\\leq &w^2_B(x)\sum_{i=1}^r \|\rho(g_ig)f\|^2_{L^2(\GaG)}=w^2_B(x)r\|f\|^2_{L^2(\GaG)},
\end{align*}
proving \textit{i)}. For \textit{ii)}, we choose an exact fundamental domain $\scrF$ for $\Gamma$ in $G$, and observe that (using $\mathbf{1}$ to denote characteristic functions)
\begin{align*}
\int_B |f( &\Gamma gh)|^2\,dh=\int_G |f(\Gamma gh)| ^2\mathbf{1}_B(h)\,dh=\int_G |f(\Gamma h)| ^2\mathbf{1}_{gB}(h)\,dh\\=&\sum_{\gamma\in\Gamma}\int_{\gamma\scrF} |f(\Gamma h)| ^2\mathbf{1}_{gB}(h)\,dh=\int_{\scrF} |f(\Gamma h)| ^2\bigg( \sum_{\gamma\in\Gamma}\mathbf{1}_{gB}(\gamma h)\bigg)\,dh\\&\leq \sup_{h\in\scrF}\#\lbrace \gamma\in\Gamma\,:\, \gamma h\in g B\rbrace \|f\|_{L^2(\GaG)}^2=\sup_{h\in G}\#\lbrace \gamma\in\Gamma\,:\, \gamma h\in g B\rbrace \|f\|_{L^2(\GaG)}^2.
\end{align*}
Noting that if $\gamma_0\in\Gamma$, $ \#\lbrace \gamma\in\Gamma\,:\, \gamma h\in g B\rbrace= \#\lbrace \gamma\in\Gamma\,:\, \gamma (\gamma_0h)\in g B\rbrace$, we have
\begin{align*}
=\max_{h\in gB} \#\lbrace \gamma\in\Gamma\,:\, \gamma h\in g B\rbrace  \|f\|_{L^2(\GaG)}^2\leq& \#\lbrace \gamma\in\Gamma\,:\, \gamma gB\cap gB\neq\emptyset\rbrace \|f\|_{L^2(\GaG)}^2\\&=\#(\Gamma\cap gBBg^{-1}) \|f\|_{L^2(\GaG)}^2,
\end{align*}
the last equality holding due to $B$ being symmetric.
\end{proof}
\begin{remark}
More work shows that there is in fact equality between both sides of \textit{ii)}. This is not needed here, however.
\end{remark}
\begin{prop}\label{Ybound}
For $g=sak\in\fG_i$, with $\fG_i$ as in Proposition \ref{REDUCTIONPROP},
\begin{equation}\label{WBBD}
w_{B}(\Gamma sak)\ll e^{\rho_{\fa_i}(a)},
\end{equation}
where the implied constant depends on $\Gamma$, $\fG_i$ and $B$.
\end{prop}

\begin{proof}
Firstly, we note that it suffices to find some set $\scrB$ for which $w_{\scrB}$ satisfies \eqref{WBBD}, since for any other choice of $B$, we may find $g_1,\,g_2\,\ldots,\,g_p\in G$ such that $B\subset \bigcup_{j=1}^p g_j\scrB $. We then have that for all $x\in\GaG$ and $f\in L^2(\GaG)$,
\begin{align*}
\int_{B} |f(xg)|^2\,dg\leq& \sum_{j=1}^p \int_{ g_j\scrB}|f(xg)|^2\,dg= \sum_{j=1}^p \int_{\scrB }|f(xg_j^{-1}g)|^2\,dg\\&\leq \sum_{j=1}^p w_{\scrB}(xg_j^{-1})^2\|f\|^2_{L^2(\GaG)}\ll w_{\scrB}(x)^2\|f\|^2_{L^2(\GaG)},
\end{align*}
so if \eqref{WBBD} holds for $w_{\scrB}$, it will hold for $w_B$ for all other choices of $B$ as well.

For each $i$, we will find a set $\scrB_i$ such that \eqref{WBBD} holds on $\fG_i$. Choosing $\scrB$ contained in the intersection of the sets $\scrB_i$ will then satisfy the requirements of the proposition. For any fixed $B$, we have
\begin{equation*}
w_B(\Gamma s a k)= w_B(\Gamma a (a^{-1}s a)k)\ll w_B(\Gamma a),
\end{equation*}
so we need only consider the restriction of $w_{\scrB}$ to $\scrA_i$. Since $N_iA_iM_i N_i^-$ is open and dense in $G$, we may choose $\scrB_i$ small enough so that
\begin{equation*}
\scrB_i\scrB_i\subset B_{N_i} B_{A_iM_i}B_{N^-_i},
\end{equation*}
where each of the sets in the right-hand side is a small neighbourhood of the identity in the corresponding subgroup. Let us now fix $T\geq 0$ large enough so that
\begin{equation*}
\bigcup_{a'\in\scrA_{i,T}} a'B_{N_i^-}a'^{-1} \subset B_{N_i^-}.
\end{equation*} 
Note that $\scrA_{i,\tau_0}\subset \scrA_i^{T+|\tau_0|}\scrA_{i,T}$ ($\tau_0$ being as in \ref{REDUCTIONPROP}); hence given any $a\in\scrA_{i,\tau_0}$, there exist $a_+\in \scrA_{i,T}$ and $a_-\in\scrA_i^{T+|\tau_0|}$ such that $a=a_+a_-$.
We now have that
\begin{equation*}
w_{\scrB_i}(\Gamma a) \ll w_{\scrB_i}(\Gamma a_+),
\end{equation*}
and
\begin{align*}
w_{\scrB_i}(\Gamma a_+)^2\leq \#(\Gamma \cap a_+ \scrB_i\scrB_i a_+^{-1} )\leq& \#(\Gamma \cap a_+ B_{N_i} B_{A_iM_i}B_{N^-_i} a_+^{-1})\\&=  \#\big(\Gamma \cap (a_+ B_{N_i} a_+^{-1}) B_{A_iM_i} (a_+B_{N^-_i} a_+^{-1}) \big)\\&\quad\leq \#(\Gamma \cap (a_+ B_{N_i} a_+^{-1}) B_{A_iM_i} B_{N^-_i}   ).
\end{align*}
Since we have assumed that $P_i$ is $\Gamma$-percuspidal, $\Gamma_{N_i}=\Gamma\cap N_i$ is cocompact in $N_i$, and there therefore exists a neighbourhood $B_0$ of $e$ in $G$ such that $\Gamma\cap N_iB_0=\Gamma_{N_i}$. We now additionally assume that $\scrB_i$ has been chosen small enough so that $B_{A_iM_i} B_{N^-_i}\subset B_0$. This further assumption allows us to conclude that
\begin{equation*}
w_{\scrB_i}(\Gamma a_+)^2 \leq \#(\Gamma_{N_i} \cap a_+ B_{N_i} a_+^{-1}B_0) = \#(\Gamma_{N_i} \cap a_+ B_{N_i} a_+^{-1}).
\end{equation*}
Using now the fact that $N_i$ is simply connected and nilpotent (and that $\Gamma_{N_i}$ is a lattice in $N_i$), we turn this into a Euclidean counting problem. Let $\lbrace L_j\rbrace$ be a basis of the Lie algebra of $N_i$, $\fn_i$, that is aligned with the restricted root space decomposition of $\fg$ with respect to $\fa_i$. More precisely, for each $j=1,\ldots,d=\dim\fn_i$, there exists some $\lambda_j\in\Sigma^+(\fg,\fa_i)$ such that $[H,L_j]=\lambda_j(H)L_j$ for all $H\in\fa_i$. We now define a map $\phi:\RR^{d}\rightarrow N_i$ by
\begin{equation*}
\phi(\vecx)=\exp\left(\sum_{j=1}^{d} x_j L_j\right),
\end{equation*}
and observe that by \cite[Theorem 1.127]{Knapp2}, $\phi$ is a diffeomorphism. Without loss of generality, we may assume that the basis $\lbrace L_j\rbrace$ has been chosen so that $B_{N_i}\subset \mathfrak{B}=\phi([-1,1]^{d})$. For any $\vecx=(x_1,\ldots,x_{d})\in \RR^{d}$ and $H\in\fa_i$, we have
\begin{align*}
\exp(H) \phi(\vecx) \exp(-H)= \exp\left( e^{\ad H}\sum_{j=1}^{d} x_j L_j\right)=\exp\left(\sum_{j=1}^{d} e^{\lambda_j(H)}x_j L_j\right).
\end{align*}
In particular, if $\phi(\vecx)\in B_{N_i}$, then by assumption $\vecx\in[-1,1]^d$, and so $|x_j e^{\lambda_j(H)}|\leq e^{\lambda_j(H)}$. This gives
\begin{equation*}
\phi^{-1}(a_+B_{N_i}a_+^{-1})\subset \phi^{-1}(a_+\mathfrak{B}a_+^{-1})=\prod_{j=1}^d [-e^{\lambda_j(\log_{\fa_i}a_+)},e^{\lambda_j(\log_{\fa_i}a_+)}].
\end{equation*} 
Since $a_+\in\scrA_{i,T}$, we have $\lambda_j(\log_{\fa_i}a_+)\geq T$ for all $j$, hence
\begin{equation*}
1\ll\mathrm{Vol}\big(\phi^{-1}(a_+B_{N_i}a_+^{-1})\big)\leq \mathrm{Vol}\big(\phi^{-1}(a_+\mathfrak{B}a_+^{-1})\big) =2^d e^{2\rho_{\fa_i}(a_+)}.
\end{equation*}
By \cite[Proposition 5.4.8 (b)]{Corwin}, there exists a lattice $\Lambda_i$ in $\fn_i$ and a finite number of elements $\Upsilon_1,\ldots,\Upsilon_q\in\fn_i$ such that $\log(\Gamma_{N_i}) = \bigcup_{j=1}^q \Upsilon_j+\Lambda_i$. This gives that $\phi^{-1}(\Gamma_{N_i})$ is contained in a finite union of affine lattices in $\RR^{d}$, hence
\begin{equation*}
\#(a_+B_{N_i}a_+^{-1}\cap \Gamma_{N_i})=\#\big( \phi^{-1}(a_+B_{N_i}a_+^{-1})\cap\phi^{-1}(\Gamma_{N_i})\big)\leq \#\big( \phi^{-1}(a_+\mathfrak{B}a_+^{-1})\cap\phi^{-1}(\Gamma_{N_i})\big) .
\end{equation*}
Since the side lengths of the rectangular box $\phi^{-1}(a_+\mathfrak{B}a_+^{-1})$ are bounded from below by $2e^T$ (uniformly over all $a_+\in\scrA_{i,T})$, we have 
\begin{equation*}
\#\big( \phi^{-1}(a_+\mathfrak{B}a_+^{-1})\cap\phi^{-1}(\Gamma_{N_i})\big) \ll \mathrm{Vol} \big( \phi^{-1}(a_+\mathfrak{B}a_+^{-1})\big)\ll  e^{2\rho_{\fa_i}(a_+)}.
\end{equation*}
Now, since $a=a_+a_-$,
\begin{equation*}
e^{2\rho_{\fa_i}(a_+)}\leq \left( \max_{a_0\in \scrA_i^{T+|\tau_0|}}e^{-2\rho_{\fa_i}(a_0)}\right)e^{2\rho_{\fa_i}(a)},
\end{equation*}
giving, for $sak\in\fG_{i}$:
\begin{equation*}
w_{\scrB_i}(\Gamma sak) \ll e^{\rho_{\fa_i}(a)}.
\end{equation*}
\end{proof}
We now fix, once and for all, a compact, symmetric neighbourhood $\scrB$ of the identity in $G$, and define the \textit{invariant height function} $\scrY_{\Gamma}$ on $\GaG$ through $\scrY_{\Gamma}(x):=w_{\scrB}(x)$. Proposition \ref{sobbnd} gives that for $m>\frac{\dim G}{2}$,
\begin{equation}\label{fbdd}
|f(x)|\ll_{\Gamma} \|f\|_{\scrS^m(\GaG)}\scrY_{\Gamma}(x)\qquad\forall f\in\scrS^m(\GaG),\,x\in\GaG.
\end{equation}
\begin{cor}\label{scrYL1cor}
$\scrY_{\Gamma}\in L^1(\GaG)$.
\end{cor}
\begin{proof}
By decomposing the measure on $\fG_i$ in a manner similar to \eqref{measuredecomp} (cf.\ \cite[p. 25]{Langlands}), and then using \eqref{WBBD}, we have
\begin{align*}
\int_{\GaG}&\scrY_{\Gamma}(x) \,d\mu(x)\leq \sum_{i=1}^{\kappa}\int_{\fG_i} \scrY_{\Gamma}(\Gamma g) \,dg = \sum_{i=1}^{\kappa}  \int_{\Omega_i}\int_{\scrA_{i,\tau_0}}\int_K \scrY_{\Gamma}(\Gamma sak) e^{-2\rho_{\fa_i}(a)}\,dk\,da\,d_ls\\&\ll \int_{\scrA_{i,\tau_0}}  e^{\rho_{\fa_i}(a)}e^{-2\rho_{\fa_i}(a)}\,da.
\end{align*}
We now use the fact that the Haar measure $da$ on $\scrA_i$ may be expressed as the pushforward (under the exponential map) of a constant multiple of the Lebesgue measure $dm$ on $\fa_i'$. Letting $H_1,\ldots,H_r$ be the basis of $\fa_i'$ defined by $\alpha_{j}(H_l)=\delta_{jl}$ for all $\alpha_j\in\Sigma_0^+(\fg,\fa_i')=\lbrace\alpha_1,\ldots,\alpha_r\rbrace$, we have
\begin{equation*}
\int_{\scrA_{i,\tau_0}}  e^{-\rho_{\fa_i}(a)}\,da\asymp\int_{\lbrace H\in \fa_i'\,:\, \lambda(H)>\tau_0\;\forall \lambda\in \Sigma^+_0(\fg,\fa_i')\rbrace}  e^{-\rho_{\fa_i}(H)}\,dm(H) \ll \prod_{j=1}^r \int_{\tau_0}^{\infty} e^{-y_j\rho_{\fa_i}(H_j)}\,dy_j<\infty.
\end{equation*}
\end{proof}
We conclude this section by giving bounds on integrals of $\scrY_{\Gamma}$ over  translates of pieces of our horospherical subgroup $U^+$. Such bounds will only be needed for integrals over relatively compact sets $B$ of positive measure in $U^+$. This will allow us to ``thicken'' $B$ to a set $C$ of positive measure in $G$ such that $B\subset C$, and make use of the \emph{wavefront property}; $Bg_{t}$ and $Cg_{t}$ will remain ``close'' to each other as $t\rightarrow-\infty$. In \cite[Proposition 6]{Edwards}, we considered the case $G=\SL(2,\CC)$, and proved a corresponding result for an integral along the \emph{boundary} of such a piece of a horosphere (subject to certain restrictions on the shape of the subset). Since the boundary has measure zero, the method used here does not work, and the proof becomes considerably more complicated. Here, though, the proof is relatively straightforward:
\begin{cor}\label{scrYAVG}
Let $B\subset U^+$ be relatively compact. Then 
\begin{equation*}
\int_B \scrY_{\Gamma}(xug_{t})\,du\ll_{\Gamma,B} \scrY^2_{\Gamma}(x)\qquad \forall x\in\GaG,\, t\leq 0.
\end{equation*}
\end{cor} 
\begin{proof}
Let $P=NAM$ be the parabolic subgroup such that $U^+=N$ and $g_{\RR}\subset A$. We start by choosing a compact symmetric neighbourhood $B_0$ of $e$ in $AMN^-\subset G$. Let  $\widetilde{B}_0$ denote the closure of $\bigcup_{s\leq 0} g_{-s} B_0g_{s}$; as in \eqref{Omegatau}, $\widetilde{B}_0$ is a compact neighbourhood of $e$ in $AMN^-$, and  
\begin{align*}
\scrY_{\Gamma}(xg_{t})\!&=\scrY_{\Gamma}(xbg_{t}(g_{-t} b^{-1}g_{t}))\ll_{\widetilde{B}_0} \scrY_{\Gamma}(xbg_{t})\qquad \forall x\in\GaG,\,t\leq 0,\,b\in \widetilde{B}_0.
\end{align*}
After taking into account the modular function $\Delta(am)=e^{-2\rho_{\fa}(a)}$ of $AM$ on $\widetilde{B}_0$, we have (cf.\ \eqref{measuredecomp})
\begin{equation*}
\int_B \scrY_{\Gamma}(xug_{t})\,du\ll \int_{B\widetilde{B}_0} \scrY_{\Gamma}(xgg_{t})\,dg.
\end{equation*}
We now choose $g_1,\ldots, g_R\in G$ such that $B\widetilde{B}_0\subset \bigcup_{j=1}^R \scrB g_j$. By Corollary \ref{scrYL1cor}, $\sqrt{\scrY_{\Gamma}}\in L^2(\GaG)$, and so from the definition of $\scrY_{\Gamma}$:
\begin{align*}
 \int_{B\widetilde{B}_0} \scrY_{\Gamma}(xgg_{t})\,dg &\leq \sum_{j=1}^R  \int_{\scrB} \left|\big(\rho(g_jg_{t})\sqrt{\scrY_{\Gamma}}\big)(xg)\right|^2\,dg\\\leq & \sum_{j=1}^R \scrY_{\Gamma}^2(x)\left\|\rho(g_jg_{t})\sqrt{\scrY_{\Gamma}} \right\|_{L^2(\GaG)}^2 \ll \scrY_{\Gamma}^2(x).
\end{align*}
\end{proof}
If $G$ is an algebraic group and $\Gamma$ is an arithmetic lattice in $G$, we may use Siegel's Conjecture, proved by Ji \cite{Ji} and Leuzinger \cite{Leuzinger}, to give a bound on a corresponding integral of $\scrY_{\Gamma}^2$. Let $\mathsf{dist}$ denote the left-invariant Riemannian metric on $G$ induced by the inner product $\langle X_1, X_2\rangle =-B(X_1,\theta X_2)$ on $\fg$. We now define a metric $\mathsf{dist}_{\GaG}$ on $\GaG$ by
\begin{equation*}
\mathsf{dist}_{\GaG}(\Gamma g_1, \Gamma g_2)=\inf_{\gamma\in\Gamma} \mathsf{dist}(\gamma g_1,g_2).
\end{equation*}

\begin{cor}\label{scrYAVG2}
Suppose $\mathbf{G}$ is a semisimple algebraic group defined over $\QQ$, $G=\mathbf{G}(\RR)$, with rank greater or equal to $2$, and $\Gamma\subset\mathbf{G}(\QQ)$ is an arithmetic subgroup. Then there exists $p\geq 0$ such that for any relatively compact $B\subset U^+$, 
\begin{equation*}
\int_B \scrY_{\Gamma}^2(xug_{t})\,du\ll_{\Gamma,B,g_{\RR}} \big(1+\mathsf{dist}_{\GaG}(\Gamma e, x)+|t|\big)^p\scrY^2_{\Gamma}(x)\qquad \forall x\in\GaG,\, t\leq 0.
\end{equation*}
\end{cor}
\begin{proof}
Letting $R=\max_{u\in B} \mathsf{dist}(e,u)$ and $\|X\|=\sqrt{\langle X, X\rangle}$ for $X\in\fg$, we have
\begin{equation*}
\mathsf{dist}_{\GaG} (xug_t,\Gamma e)\leq \mathsf{dist}_{\GaG}(x,\Gamma e)+R+|t|\|Y\|\qquad \forall x\in\GaG,\, u\in B,\,t\in\RR. 
\end{equation*}
Making the definitions $\fB_r=\lbrace y\in\GaG\,:\, \mathsf{dist}_{\GaG}(\Gamma e, y)\leq r\rbrace$ and $r(x,t)=\mathsf{dist}_{\GaG}(x,\Gamma e)+R+|t|\|Y\|$, we then have that
\begin{equation*}
\scrY^2_{\Gamma}(xug_t) =\scrY^2_{\Gamma}(xug_t)\mathbf{1}^2_{\fB_{r(x,t)}}(xug_t)\qquad \forall x\in\GaG,\,u\in B,\, t\in \RR.
\end{equation*}
We denote $F=\scrY_{\Gamma}\mathbf{1}_{\fB_{r(x,t)}}$, and note that $F\in L^2(\GaG)$. Now arguing as in the proof of Corollary \ref{scrYAVG}, for $t\leq 0$,
\begin{equation*}
\int_B \scrY_{\Gamma}^2(xug_{t})\,du=\int_B |F(xug_{t})|^2\,du\ll \sum_{j=1}^R \scrY_{\Gamma}^2(x)\left\|\rho(g_jg_{t})F\right\|_{L^2(\GaG)}^2 \ll \scrY_{\Gamma}^2(x)\|F\|^2_{L^2(\GaG)}.
\end{equation*}
All that now remains is to bound $\|F\|^2_{L^2(\GaG)}$. From \cite[Theorem 5.7]{Leuzinger} or \cite[Theorem 7.6]{Ji}, there exists a constant $D>0$ such that
\begin{equation*}
\disG{g}{e}\leq\disGa{\Gamma g}{\Gamma e}+D \qquad \forall g\in \bigcup_{i=1}^{\kappa} \fG_i,
\end{equation*}
and hence
\begin{equation*}
\mathbf{1}_{\fB_{r(x,t)}}(\Gamma g)=1\Rightarrow \mathsf{dist}(g,e)\leq r(x,t)+D \qquad \forall g\in \bigcup_{i=1}^{\kappa} \fG_i.
\end{equation*}
This, together with Proposition \ref{Ybound}, gives
\begin{align*}
\int_{\fG_i} &F^2(\Gamma g)\,dg\ll \int_{\Omega_i}\int_{\scrA_{i,\tau_0}}\int_K e^{2\rho_i(a)}\mathbf{1}_{\fB_{r(x,t)}}(\Gamma sak)e^{-2\rho_i(a)}\,dk\,da\,d_ls.
\end{align*}
Since $\Omega_i$ and $K$ are compact, there exists a constant $c>0$ that enables us to bound the integral by
\begin{align*}
\ll\int_{\lbrace H\in \log\scrA_{i,\tau_0}\,:\,\|H\|\leq c+r(x,t)+D\rbrace}dm(H)\ll (1+r(x,t))^{\dim \scrA_i}.
\end{align*}
\end{proof}
\begin{remark} Using Garland and Raghunathan's reduction theory \cite{GarlandRaghunathan}, it should be possible to prove a similar result for any lattice in an arbitrary $G$ of rank one.
\end{remark}
\section{Proof of Theorem 1} \label{thm1proof}
\subsection*{Proof of Theorem \ref{maintheorem1}} 
We start by making some initial reductions. The constant $m_1$ for which we prove \eqref{mainthmbd} will satisfy $m_1>\frac{\dim G}{2}$, and thus by Proposition \ref{sobbnd}, both sides of \eqref{mainthmbd} depend continuously on $f\in\scrS^{m_1}(V)$ (with respect to $\|\cdot\|_{\scrS^{m_1}(V)}$). Hence, since $V^{\infty}$ is dense in $\scrS^{n_2}(V)$, without loss of generality, from now on we may assume that $f\in V^{\infty}$. Similarly, we may assume that $\chi\in C_c^{\infty}(B)$.

We let $\left(\int_{\mathsf{Z}}^{\oplus} \mathcal{\pi}_{\zeta}\,d\upsilon(\zeta), \int_{\mathsf{Z}}^{\oplus} \mathcal{H}_{\zeta}\,d\upsilon(\zeta) \right)$ denote the direct integral decomposition of $(\rho,V)$ into irreducible unitary representations, and $\int_{\mathsf{Z}} f_{\zeta}\,d\upsilon(\zeta)$ be the corresponding decomposition of $f$. Using the notation $\scrI_f^{\chi}(t)=\scrI_f^{\chi}(g_{t})$ from Section \ref{DIFFEQSSEC}, we then have
\begin{align*}
\scrI_f^{\chi}(t)= \int_{U^+} \chi(u)\rho(ug_{t})f\,du=&\int_{\mathsf{Z}} \int_{U^+} \chi(u)\pi_{\zeta}(ug_{-t})f_{\zeta}\,du\,d\upsilon(\zeta)=\int_{\mathsf{Z}} \scrI_{f_{\zeta}}^{\chi}(t)\,d\upsilon(\zeta).
\end{align*}
Proposition \ref{BurgerProp2} is now applied with $\alpha=\min_j\alpha_j$ to $\upsilon$-a.e. $f_{\zeta}$ (by Lemma \ref{DECOMPRATE}, $\eta$ is a rate of decay for the matrix coefficients of $g_{\RR}$ in $\upsilon$-a.e. $(\pi_{\zeta},\scrH_{\zeta})$), giving
\begin{equation*}
\scrI_{f_{\zeta}}^{\chi}(t)=\sum_{\bk\in\bI_{k_0}}\int_{-\infty}^0 \scrC_{\zeta,\bk}(t,s)\scrI_{d\pi_{\zeta}(U_{\bk})f_{\zeta}}^{X_{\bk}\chi}(s)\,ds+\sum_{\bj\in\bI_{\leq k_0-1}}\sum_{i=0}^{W} \scrD_{\zeta,\bj,i}(t)\scrI_{d\pi_{\zeta}(Y^{i}U_{\bj})f_{\zeta}}^{X_{\bj}\chi}(0).
\end{equation*}
By integrating this identity over $\mathsf{Z}$, we get
\begin{align}\label{IAVDECOMP}
\scrI_{f}^{\chi}(t)=\sum_{\bk\in\bI_{k_0}}&\int_{\mathsf{Z}}\int_{-\infty}^0 \scrC_{\zeta,\bk}(t,s)\scrI_{d\pi_{\zeta}(U_{\bk})f_{\zeta}}^{X_{\bk}\chi}(s)\,ds\,d\upsilon(\zeta)\\\notag&+\sum_{\bj\in\bI_{\leq k_0-1}}\sum_{i=0}^{W}\int_{\mathsf{Z}} \scrD_{\zeta,\bj,i}(t)\scrI_{d\pi_{\zeta}(Y^{i}U_{\bj})f_{\zeta}}^{X_{\bj}\chi}(0)\,d\upsilon(\zeta).
\end{align}
We now form intertwining operators $\lbrace T_{\bk}(\tau,s)\rbrace$, $\lbrace S_{\bj,i}(\tau)\rbrace$ in the manner of Section \ref{Reptheorysec}: for $\phi\in V^{\infty}$, define
\begin{align*}
T_{\bk}(\tau,s)\phi&=\int_{\mathsf{Z}} \scrC_{\zeta,\bk}(\tau,s) \phi_{\zeta}\,d\upsilon(\zeta)\\ S_{\bj,i}(\tau)\phi&=\int_{\mathsf{Z}}  \scrD_{\zeta,\bj,i}(\tau)\phi_{\zeta}\,d\upsilon(\zeta).
\end{align*}
The bounds provided in Proposition \ref{BurgerProp2} suffice to exchange the order of integration in \eqref{IAVDECOMP}. Furthermore, from Proposition \ref{BurgerProp2} \textit{(2)}, each $S_{\bj,i}$ is a linear map $V^{\infty}\rightarrow V^{\infty}$ which is continuous with respect to two suitably chosen Sobolev norms. We may thus rewrite \eqref{IAVDECOMP} as
\begin{equation*}
\scrI_{f}^{\chi}(t)=\sum_{\bk\in\bI_{k_0}}\int_{-\infty}^0 T_{\bk}(t,s)\scrI_{d\rho(U_{\bk})f}^{X_{\bk}\chi}(s)\,ds+\sum_{\bj\in\bI_{\leq k_0 -1}}\sum_{i=0}^{W} S_{\bj,i}(t)\scrI_{d\rho(Y^{i}U_{\bj})f}^{X_{\bj}\chi}(0).
\end{equation*}
The definition of $\scrI$, and the fact that all of the operators $ T_{\bk}(t,s)$ and $S_{\bj,i}(t)$ commute with $\rho$, gives
\begin{align}\label{AVGIDENT}
\int_{U^+}\chi(u)\rho(ug_{t})f\,du=\sum_{\bk\in\bI_{k_0}}&\int_{-\infty}^0 \int_{U^+}(X_{\bk}\chi)(u)\rho(ug_s)T_{\bk}(t,s)d\rho(U_{\bk})f\,du\,ds\\\notag&+\sum_{\bj\in\bI_{\leq k_0 -1}}\sum_{i=0}^{W} \int_{U^+}(X_{\bj}\chi)(u)\rho(u)S_{\bj,i}(t)d\rho(Y^{i}U_{\bj})f\,du.
\end{align}
We now apply the functional ``evaluation at $x$'' to both sides: if $m> \frac{\dim(G)}{2}$, then by \eqref{fbdd} (and recalling that $\supp(\chi)\subset B$),
\begin{align}\label{final}
\left|\int_{U^+}\chi(u)f(xug_{t})\,du\right|&\ll_{\Gamma,m}\!\!\!\sum_{\bk\in\bI_{k_0}}\!\!\!\|X_{\bk}\chi\|_{L^{\infty}(B)}\!\!\int_{-\infty}^0\!\! \!\|T_{\bk}(t,s)d\rho(U_{\bk})f\|_{\scrS^m(\GaG)}\!\!\int_B\! \scrY_{\Gamma}(xug_s)\,du\,ds\\\notag&\quad+\!\!\!\!\sum_{\bj\in\bI_{\leq k_0 -1}}\!\sum_{i=0}^{W} \|X_{\bj}\chi\|_{L^{\infty}(B)}\|S_{\bj,i}(t)d\rho(Y^{i}U_{\bj})f\|_{\scrS^m(\GaG)}\!\!\int_B \!\scrY_{\Gamma}(xu)\,du.
\end{align}
Corollary \ref{scrYAVG} is now used to bound the integrals over $B$ of $\scrY_{\Gamma}$ by $\scrY_{\Gamma}(x)^2$. Decomposing the Sobolev norms $\|T_{\bk}(t,s)d\rho(U_{\bk})f\|_{\scrS^m(\GaG)}$ and
$\|S_{\bj,i}(t)d\rho(Y^{i}U_{\bj})f\|_{\scrS^m(\GaG)}$ as in
\eqref{SOBNORMDECOMP} and applying the bounds from Proposition \ref{BurgerProp2} gives
\begin{align*}
\|T_{\bk}(t,s)d\rho(U_{\bk})f\|_{\scrS^m(\GaG)}&\ll (1+|t|^W) e^{\eta t+ \frac{\alpha}{5}s}\|f\|_{\scrS^{m_1}(\GaG)}\\\notag\|S_{\bj,i}(t)d\rho(Y^{i}U_{\bj})f\|_{\scrS^m(\GaG)}&\ll (1+|t|^W)e^{\eta t}\|f\|_{\scrS^{m_1}(\GaG) }
\end{align*}
for some $m_1\in\NN$. Entering these bounds into \eqref{final} and choosing $m_2$ so that all factors involving $\chi$ in \eqref{final} are bounded by $\|\chi\|_{W^{m_2,\infty}(B)}$ proves the theorem for any $f\in V^{\infty}$ and $\chi\in C_c^{\infty}(B)$. As noted at the beginning of the proof, the continuous dependency on the functions in \eqref{mainthmbd} may now be used to conclude that the same holds for \emph{all} $f\in\scrS^{m_1}(\GaG)$ and $\chi\in W^{m_2,\infty}_0(B)$.

\hspace{425.5pt}\qedsymbol

\begin{remark}\label{ClosedHoros}
The proof of Theorem \ref{maintheorem1} simplifies somewhat when considering the equidistribution of translates of entire closed horospherical orbits. Suppose that $\Gamma\cap U^+$ is a lattice in $U^+$. Let $\varphi\in C_c^{\infty}(U^+)$ be such that $\varphi(u)\geq 0$ for all $u\in U^+$ and $\int_{U^+} \varphi(u)\,du=1$. We now let $\chi= \mathbf{1}_{\scrF}*\varphi\in C_c^{\infty}(U^+)$, where $\scrF$ is a   reasonable fundamental domain for $\Gamma\cap U^+$ in $U^+$. For all $u_0\in U^+$ we have
\begin{equation}\label{chisum}
\sum_{\gamma\in \Gamma\cap U^+} \chi(\gamma u_0)=\sum_{\gamma\in \Gamma\cap U^+} ( \mathbf{1}_{\scrF}*\varphi)(\gamma u_0)= \int_{U^+}\left( \sum_{\gamma\in \Gamma\cap U^+}\mathbf{1}_{\scrF}(\gamma u_0 u^{-1})\right)\varphi(u)\,du=1,
\end{equation}
since $\sum_{\gamma\in \Gamma\cap U^+}\mathbf{1}_{\scrF}(\gamma u_0 u^{-1})=1$ for almost all $u\in U^+$. From this we see that
\begin{equation}\label{closedhoro}
\int_{U^+} \chi(u) \phi(\Gamma u g_t)\,du=\int_{\scrF} \left( \sum_{\gamma\in\Gamma\cap U^+}\chi(\gamma u)\right)\phi(\Gamma u g_t)\,du=\int_{\Gamma\cap U^+\backslash U^+} \phi(\Gamma u g_t)\,du
\end{equation}
for any $\phi\in V^{\infty}$. For $X\in\fu^+$, we have $X\chi=\mathbf{1}_{\scrF} * (X\varphi)$. Since $\varphi$ has compact support, $\int_{U^+}(X\varphi)(u)\,du=\int_{U^+} \left.\frac{d}{dt}\right|_{t=0} \varphi(u\exp(tX))\,du=\left.\frac{d}{dt}\right|_{t=0}\int_{U^+}  \varphi\,du=0$. By the same calculation as in \eqref{chisum}, $\sum_{\gamma\in \Gamma\cap U^+} (X\chi)(\gamma u_0)=0$, and hence
\begin{equation}\label{Xchi0}
\int_{U^+} (X\chi)(u)\phi(\Gamma u g_t)=0
\end{equation}
for any $\phi\in V^{\infty}$ and those $X\in \scrU(\fu^+_{\CC})$ with vanishing scalar component. The bounds discussed during the proof of Theorem \ref{maintheorem1} allow us to evaluate \eqref{AVGIDENT} at $x=\Gamma e$. Using \eqref{closedhoro} and \eqref{Xchi0}, we then have
\begin{equation*}
\int_{\Gamma\cap U^+\backslash U^+} f(\Gamma u g_t)\,du = \sum_{i=0}^{W} \int_{\Gamma\cap U^+\backslash U^+}(S_{\mathbf{0},i}(t)d\rho(Y^i)f)(\Gamma u)\,du.
\end{equation*}
Observe that if $V$ is an irreducible subrepresentation of $(\rho,L^2(\GaG)_0)$, each operator $S_{\mathbf{0},i}(t)$, $i=0,\ldots,W$, is of the form $S_{\mathbf{0},i}(t)=\varphi_i(t)\mathrm{Id}$ for some explicit function $\varphi_i:\RR_{\leq 0}\rightarrow \CC$ depending only on the representation $(\rho, V)$, $\eta$, and $g_{\RR}$. It should thus be possible to use the decomposition of $(\rho,L^2(\GaG)_0)$ to give an explicit formula for $\int_{\Gamma\cap U^+\backslash U^+} f(\Gamma u g_t)\,du$ similar to that of \cite[Theorem 1]{Sarnak}.
\end{remark}

\begin{remark}\label{L2remark}
By using the Cauchy-Schwarz inequality, we have 
\begin{equation*}
\int_{B} |(X_{\bk}\chi)(u)|\scrY_{\Gamma}(xug_s)\,du\leq \|X_{\bk}\chi\|_{L^2(B)}\left( \int_B \scrY_{\Gamma}^2(xug_s)\,du\right)^{1/2}.
\end{equation*} 
In the case that $G$ is an algebraic group of rank greater than or equal to two and $\Gamma$ is an arithmetic lattice in $G$, we may use Corollary \ref{scrYAVG2} to go from \eqref{AVGIDENT} to
\begin{align*}
\left|\int_{U^+}\chi(u)f(xug_{t})\,du\right|\ll&\!\!\!\sum_{\bk\in\bI_{k_0}}\!\!\!\|X_{\bk}\chi\|_{L^{2}(B)}\scrY_{\Gamma}(x)\\&\qquad\times\int_{-\infty}^0\!\! \!\|T_{\bk}(t,s)d\rho(U_{\bk})f\|_{\scrS^m(\GaG)}\big(\!1\!+\!\mathsf{dist}_{\GaG}(\Gamma e, x)\!+\!|s|\big)^{p/2}\,ds\\\notag&+\!\!\!\!\sum_{\bj\in\bI_{\leq k_0 -1}}\!\sum_{i=0}^{W} \|X_{\bj}\chi\|_{L^{2}(B)}\scrY_{\Gamma}(x)\\&\qquad\times\|S_{\bj,i}(t)d\rho(Y^{i}U_{\bj})f\|_{\scrS^m(\GaG)}\big(\!1\!+\!\mathsf{dist}_{\GaG}(\Gamma e, x)\big)^{p/2},
\end{align*}
instead of \eqref{final}. Now continuing with the proof of Theorem \ref{maintheorem1} as before, we may replace \eqref{mainthmbd} with
\begin{equation*}
\left| \int_{U^+} \chi(u)f(xug_{t}) \,du\right|\leq C \|\chi\|_{W^{m_2,2}(B)}\|f\|_{\scrS^{m_1}(\GaG)}\scrY_{\Gamma}(x)(1+\mathsf{dist}_{\GaG}(\Gamma e, x)\big)^{j_1}(1+|t|)^{j_2} e^{\eta t}
\end{equation*}
for some $j_1,\,j_2\geq 0$.
\end{remark}

\section{Uniform Bounds}\label{UNIFORMSEC}
The goal of this section is to prove Theorem \ref{maintheorem2}. Recall that we have fixed a parabolic subgroup $P=NAM$ (which, without loss of generality, we may assume to be standard) and a norm $\|\cdot\|$ on $\fa$ that controls the rate of decay of matrix coefficients of $\overline{A^+}$, i.e. there exists $q\geq0$ such that for all $K$-finite vectors $\vecu,\,\vecv\in V$ and $H\in \overline{\fa^+}$,
\begin{equation*}
|\langle \rho(\exp(H))\vecu,\vecv \rangle|\ll \|\vecu\|\|\vecv\|\left( \dim(\rho(K)\vecu)\dim(\rho(K)\vecv)\right)^{1/2} (1+\|H\|^q)e^{-\|H\|}.
\end{equation*} 
By Lemma \ref{DECOMPRATE}, the same type of bound holds for almost every irreducible representation $(\pi,\scrH)$ occurring in the direct integral decomposition of $V$. Furthermore, by Lemma \ref{SOBNORMS}, we may pass to a corresponding bound for the matrix coefficients of smooth vectors: for almost every $(\pi,\scrH)$, there exist $m\in\NN$, $C$, $q\geq0$ such that 
\begin{equation*}
| \langle \pi(\exp(H))\vecu,\vecv \rangle|\leq C \|\vecu\|_{\scrS^m{}(\scrH)}\|\vecv\|_{\scrS^{m}(\scrH)} (1+\|H\|^q)e^{-\|H\|}\qquad \forall \vecu,\,\vecv\in\scrH^{\infty},\,H\in \overline{\fa^+}. 
\end{equation*}
\subsection{The sets $\fa_{\epsilon}$}
In order to prove Theorem \ref{maintheorem2}, we first show that for every $\epsilon>0$, we can find an appropriately large set in $\overline{\fa^+}$ such that Theorem \ref{maintheorem1} holds \emph{uniformly} for all $g_{\RR}=\exp(\RR H)$, with $H$ in the given set. In order to do this, we must  make explicit the dependency on the subgroup $g_{\RR}$ in Section \ref{BURGERFORMULASEC}, in particular Proposition \ref{BurgerProp2}. We now form subsets
\begin{align*}
\overline{\fa^+_{\epsilon}}&=\left\lbrace H\in \overline{\fa^+}\,:\,\alpha(H)>0 \Rightarrow \alpha(H)\geq \epsilon\|H\|\qquad\forall \alpha\in\Sigma_0^+(\fg,\fa)\right\rbrace \\
\overline{\fa^+_{\epsilon,1}}&=\lbrace H\in\overline{\fa^+_{\epsilon}}\,:\,\|H\|=1\rbrace. 
\end{align*}
Observe that if $H\in\overline{\fa_{\epsilon}^+}$, then $\delta H\in\overline{\fa_{\epsilon}^+}$ for all $\delta\geq 0$. Also, if we denote $\Sigma_0^+(\fg,\fa)=\lbrace \alpha_1,\,\alpha_2,\ldots, \alpha_r\rbrace$, then $\overline{\fa^+_{\epsilon}}$ and $\overline{\fa^+_{\epsilon,1}}$ have at most $2^r$ connected components; namely
\begin{equation*}
\overline{\fa^+_{\epsilon}}=\bigcup_{\scrF\subset\Sigma_0^+(\fg,\fa)}\overline{\fa^+_{\epsilon,\scrF}},
\qquad\overline{\fa^+_{\epsilon,1}}=\bigcup_{\scrF\subset\Sigma_0^+(\fg,\fa)}\overline{\fa^+_{\epsilon,1,\scrF}},
\end{equation*}
where
\begin{equation*}
\overline{\fa^+_{\epsilon,\scrF}}=\left\lbrace H\in\fa\,:\, \alpha(H)\geq\epsilon\,\,\forall\alpha\in \Sigma_0^+(\fg,\fa)\setminus\scrF\,\,\mathrm{and}\, \,\alpha(H)=0\,\,\forall \alpha\in\scrF \right\rbrace,
\end{equation*}
and
\begin{equation*}
\overline{\fa^+_{\epsilon,1,\scrF}}=\left\lbrace H\in\overline{\fa^+_{\epsilon,\scrF}}\,:\, \|H\|=1\right\rbrace.
\end{equation*}
 It is easily seen that for any $H\in\overline{\fa^+}$, there exists $\epsilon_H>0$ such that $H\in \overline{\fa^+_{\epsilon}}$ for all $\epsilon\leq \epsilon_H$. It is the sets $\overline{\fa_{\epsilon,1}^+}$ on which Theorem \ref{maintheorem1} will be shown to hold uniformly. After proving that every element of $\overline{\fa^+}$ may be decomposed in a good way as the sum of an element of $\overline{\fa^+_{\epsilon}}$ and another element of $\overline{\fa^+}$, we will proceed to prove generalizations of the lemmas and propositions needed in the proof of Theorem \ref{maintheorem1}.
\begin{lem}\label{EPSILONH}
There exists a constant $c>0$ depending only on $G$, $\fa$, and $\|\cdot\|$ such that for any $\epsilon>0$ and $H\in\overline{\fa^+}$, there exists $H_{\epsilon}\in\overline{\fa_{\epsilon}^+}$ such that $J_{\epsilon}:=H-H_{\epsilon}\in\overline{\fa^+}$ and $\|H_{\epsilon}\|\geq (1-c\epsilon)\|H\|$.
\end{lem}
\begin{proof}
If $H\in \overline{\fa^+_{\epsilon}}$, then let $H_{\epsilon}=H$. We now assume otherwise, and let $\lbrace H^j\rbrace$ denote the basis of $\fa$ defined by $\alpha_i(H^j)=\delta_{ij}$ (recall that $\Sigma_0^+(\fg,\fa)=\lbrace \alpha_1,\ldots,\alpha_r\rbrace$); we may thus write (in a unique way) $H=\sum_j x_j H^j$ with all $x_j\geq 0$. Assume that $\epsilon$ is chosen small enough so that $H^j\in  \overline{\fa^+_{\epsilon}} $ for all $j=1,\ldots, r$. Let $\|\cdot\|_{\infty}$ be the $\ell^{\infty}$ norm on $\fa$ with respect to this basis, i.e.
\begin{equation*}
\|Y\|_{\infty}:=\max_j \lbrace|\alpha_j(Y)|\rbrace,\qquad \forall Y\in\fa.
\end{equation*}
By equivalence of norms, there exist $D\geq D'>0$ such that $D'\|Y\|_{\infty}\leq\|Y\|\leq D\|Y\|_{\infty}$ for all $Y\in\fa$. Given $\epsilon>0$, we set
$\delta=D\epsilon\|H\|_{\infty}>0$, $\scrF=\lbrace j\in\lbrace 1,\ldots,r\rbrace\,:\,\alpha_j(H)<\delta\rbrace$, and define
\begin{align*}
H_{\epsilon}&= \sum_{j\in \scrF^{\mathsf{c}}} x_j H^j \\J_{\epsilon}&= \sum_{j\in \scrF} x_j H^j. 
\end{align*}
Note that $J_{\epsilon}\in\overline{\fa^+}$. For each $j\in\scrF$, $\alpha_j(H_{\epsilon})=0$, and for each $j\in\scrF^{\mathsf{c}}$ we have
\begin{equation*}
\alpha_j(H_{\epsilon})=\alpha_{j}(H)\geq \delta=D\epsilon\|H\|_{\infty}\geq D\epsilon\|H_{\epsilon}\|_{\infty} \geq \epsilon \|H_{\epsilon}\|,
\end{equation*}
so $H_{\epsilon}\in \overline{\fa_{\epsilon}^+}$. Also,
\begin{equation*}
\|J_{\epsilon}\|\leq D\|J_{\epsilon}\|_{\infty}=D\max\lbrace x_j\,:\,j\in\scrF\rbrace<D\delta=D^2\epsilon\|H\|_{\infty}\leq\frac{D^2\epsilon}{D'}\|H\|,
\end{equation*}
so
\begin{equation*}
\|H_{\epsilon}\|\geq \|H\|-\|J_{\epsilon}\|\geq \left(1- \frac{D^2\epsilon}{D'}\right)\|H\|,
\end{equation*}
proving that $H_{\epsilon}$ and $J_{\epsilon}$ have the desired properties (with $c=\frac{D^2}{D'}$).  
\end{proof}

\begin{remark}\label{EPSremark}
Throughout the remainder of Section \ref{UNIFORMSEC}, we assume that $\epsilon$ is chosen to be small enough so that $\epsilon<\min\lbrace 1, (2c)^{-1}\rbrace$, with $c$ being as in Lemma \ref{EPSILONH}; this ensures that $\frac{\|J_{\epsilon}\|}{\|H_{\epsilon}\|} \leq \frac{c\epsilon \|H\|}{(1-c\epsilon)\|H\|}<1$.
\end{remark}

\subsection{Continuity of the Harish-Chandra Isomorphism}
 We now fix a restricted root basis $\lbrace X_{j}\rbrace_{j=1,\ldots,d}$ of $\fn$, with corresponding elements $\beta_j\in\fa^*$, $j=1,\ldots,d$ (so $\ad_HX_j=\beta_j(H)X_j$ for all $H\in\fa$). Note that there may be repetitions in the collection $\beta_1,\ldots,\beta_d$, and each $\beta_j$ may be written (in a unique way) as a linear combination of elements of $\Sigma^+_0(\fg,\fa)=\lbrace \alpha_1,\ldots,\alpha_r\rbrace$ with all coefficients being non-negative integers. Recalling the notation of the proof of Lemma \ref{LIEIDENT}, let $\fh_{\CC}=\fa_{\CC}\oplus\ft_{\CC}$ be a Cartan subalgebra of $\fg_{\CC}$ containing $\fa_0$, with Weyl group $\scrW=\scrW(\fg_{\CC},\fh_{\CC})$. Letting $W=|\scrW|$, by the Poincar\'e-Birkhoff-Witt theorem we may view $\scrU^{W}(\fg_{\CC})$ as a finite dimensional \emph{topological} vector space (over $\CC$). For notational purposes, given $H\in\fa$, we let $W_H$ denote the integer $|\scrW/\mathrm{Stab}_{\scrW}(H)|$.  We will now prove a version of Lemma \ref{LIEIDENT} which takes into account the topological structure on $\scrU^{W}(\fg_{\CC})$:
\begin{lem}\label{LIEIDENT2}
There exist families of maps $Z_i:\,\fa\rightarrow \Zg\cap \scrU^{W}(\fg_{\CC})$, $i=0,\,1,\ldots,W$, and $U_j:\,\fa\rightarrow \scrU^{W}(\fg_{\CC})$, $j=1,\ldots,d$ which satisfy 
\begin{equation}\label{HIDENT}
\sum_{i=0}^W Z_i(H)H^{i}=\sum_{j=1}^d X_jU_j(H)\qquad \forall H\in\fa,
\end{equation}
and additionally have the following properties:
\begin{enumerate}[i)]
\item $Z_{W_H}(H)=1$ and $Z_i(H)=0\quad\forall H\in \fa,\, i>W_H$.
\item $U_j(H)=0$ $\forall H\in \overline{\fa^+}\cap \ker\beta_j$. 
\item All the maps $U_1,\ldots, U_d$, $Z_0,\ldots,Z_W$ 
are continuous when restricted to any set $\overline{\fa^+_{\epsilon,\scrF}}$.
\end{enumerate}
\end{lem}
\begin{proof}
The construction in the proof Lemma \ref{LIEIDENT} gives, for any given $H\in \fa$, elements $Z_{H,j}$ in $\Zg$ such that
\begin{equation}\label{ZHPOLY}
H^{W_H}+Z_{H,W-1}H^{W_H-1}+\ldots+ Z_{H,1}H +Z_{H,0}=\fn \,\Ug.
\end{equation}
We thus define the functions $Z_i$ on $\fa$ in accordance with these constructions: 
\begin{equation*}
Z_i(H)=\begin{cases}Z_{H,i}\quad&\mathrm{if}\:i<W_H-1\\1\quad&\mathrm{if}\:i=W_H\\0\quad& \mathrm{if}\: i>W_H-1.\end{cases}
\end{equation*}  
The fact that the Harish-Chandra isomorphism respects the canonical filtrations on $\Zg$ and $\Uh^{\scrW}$, as is clear from the proof of \cite[Theorem 5.44]{Knapp2}, gives that $Z_i(H)\in\scrU^{W}(\fg_{\CC})$ for all $H\in\fa$, $i=0,\ldots, W$. Using \eqref{ZHPOLY}, we have that there exist elements $U_{j,H}\in \scrU^W(\fg_{\CC})$ such that 
\begin{equation*}
\sum_{i=0}^W Z_i(H)H^{i}=\sum_{j=1}^d X_jU_{j,H},
\end{equation*}
and the functions $Z_i$ satisfy \textit{i)}. We thus need to use the elements $U_{j,H}$ to define functions $U_j$ that satisfy \eqref{HIDENT}, \textit{ii)} and \textit{iii)}. To any $H\in\overline{\fa^+}$ we may associate a set $\scrF\subset\Sigma_0^+(\fg,\fa)$ consisting of all $\alpha\in \Sigma_0^+(\fg,\fa)$ such that $\alpha(H)=0$. Every $\beta\in\Sigma^+(\fg,\fa) $ for which $\beta(H)=0$ may be thus be written as a linear combination consisting solely of elements of $\scrF$ (since each $\beta\in\Sigma^+(\fg,\fa)$ may be written in a unique way as a linear combination of elements of $\Sigma^+_0(\fg,\fa)$ with all coefficients being non-negative integers). We may now construct a standard parabolic subgroup $P_{\scrF}=N_{\scrF}A_{\scrF}M_{\scrF}\supset P$, with Lie algebra $\fn_{\scrF}\oplus\fa_{\scrF}\oplus\fm_{\scrF}$ such that $H\in \fa_{\scrF}^+$ and a basis for $\fn_{\scrF}$ is given by the elements $X_j$ for which $\beta_j(H)\neq 0$. Since $\fa_{\scrF}\subset\fa\subset \fh_{\CC}$ and $\fn_{\scrF}\subset \fn$, applying the construction in the proof of Lemma \ref{LIEIDENT} to $H$ with $P_{\scrF}$ in place of $P$ gives rise to the same elements $Z_{H,i}$, and so
\begin{equation*}
\sum_{i=0}^W Z_i(H)H^{i}=\sum_{j=1}^d X_jU_{j,H,\scrF}\in \fn_{\scrF}\Ug,
\end{equation*}
where $U_{j,H,\scrF}=0$ for all $j$ such that $X_j\not\in\fn_{\scrF}$, i.e. those $X_j$ such that $\beta_j(H)=0$. We thus construct functions $U_j$ in the following manner: for each subset $\widetilde{\scrF}\subset \Sigma_0^+(\fg,\fa_0)$, we choose a Poincar\'e-Birkhoff-Witt basis of $\fg$ with the first elements of this basis being the $X_j$s that form a basis of $\fn_{\widetilde{\scrF}}$. Then, for each $H\in\fa$, we choose the \emph{unique} maximal subset $\widetilde{\scrF}\subset \Sigma_0^+(\fg,\fa_0)$ such that $H\in\fa_{\widetilde{\scrF}}$. We then have that 
\begin{equation*}
\sum_{i=0}^W Z_i(H)H^{i}=\sum_{j=1}^d X_jU_{j,H,\widetilde{\scrF}}\in \fn_{\widetilde{\scrF}}\Ug,
\end{equation*}
where the elements $U_{j,H,\widetilde{\scrF}}$ are uniquely determined by the choice of basis, and $U_{j,H,\widetilde{\scrF}}=0$ if $X_j\not\in \fn_{\widetilde{\scrF}}$. The functions $U_j$ are then defined by $U_j(H)=U_{j,H,\widetilde{\scrF}}$. This choice of functions $U_j$ will satisfy \textit{ii)}, and (together with the functions $Z_i$) \eqref{HIDENT}.

Finally, in order to prove \textit{iii)}, we first note that if $H_1$ and $H_2$ both belong to some set $\fa^+_{\scrF}$, then $\Stab_{\scrW}(H_1)=\Stab_{\scrW}(H_2)$, indeed, by Chevalley's lemma \cite[Proposition 2.72]{Knapp2}, this set is determined by the choice of $\scrF\subset\Sigma_0^{+}(\fg,\fa)$. This implies that the coefficients of the polynomial \eqref{HPOLY} depend continuously on $H$ as $H$ is allowed to vary in $\fa_{\scrF}^+$. More precisely, if
\begin{equation}\label{HPOLY2}
p_H(x)=\prod_{w\in\scrW/\mathrm{Stab}_{\scrW}(H)} \big(x-(wH+\delta_{\fh_{\CC}}(H)1)\big)=x^{W_H}+\sum_{i=0}^{W_H-1}J_{H,i}x^{i},\end{equation}
then the maps $H\mapsto J_{H,i}$ are continuous maps from $\fa_{\scrF}^+$ to $\Uh^{\scrW}$. Since the Harish-Chandra isomorphism $\gamma|_{\Zg\cap\scrU^{W}(\fg_{\CC})}$ is an invertible linear map between two subspaces of $\scrU^{W}(\fg_{\CC})$, $\gamma|_{\Zg\cap\scrU^{W}(\fg_{\CC})}^{-1}$ is a continuous linear map from $\scrU^{W}(\fh_{\CC})^{\scrW}$ to $\Zg \cap\scrU^{W}(\fg_{\CC})$. Applying this to the right-hand side of \eqref{HPOLY2} gives that for each $i$, $H\mapsto \gamma^{-1}(J_{H,i})=Z_i(H)$ is continuous on $\fa^+_{\scrF}$. Since the sum $\sum_{i=0}^W Z_i(H)H^{i}$ depends continuously on $H\in\fa^+_{\scrF}$, $\sum_{j=1}^d X_j U_j(H)$ must as well. Now using the fact that the coordinates in our Poincar\'e-Birkhoff-Witt basis are continuous functions on $\scrU^W(\fg_{\CC})$, we conclude that the maps $U_j|_{\fa^+_{\scrF}}$ are also continuous.
\end{proof}
\subsection{Differential Equations}\label{HDIFFSEC} Following Section \ref{DIFFEQSSEC}, for an irreducible unitary representation $(\pi,\scrH)$ for which $\|\cdot\|$ controls the rate of decay of matrix coefficients of $\overline{A^+}$, we define
\begin{equation*}
\scrI_{\vecv}^{\chi}(H,t)=\int_N \chi(n)\pi\big(n\exp(tH)\big)\vecv\,dn,
\end{equation*}
where $\vecv\in\scrH^{\infty}$, $\chi\in C_c^{\infty}(N)$, $t\leq 0$ and $H\in\overline{\fa^+}$. Furthermore (as previously), by Schur's lemma, there exist functions $a_i:\fa\rightarrow\CC$ such that $d\pi(Z_i(H))\vecw=a_i(H)\vecw$ for all $\vecw\in \scrH^{\infty}$ , $i=0,\ldots,W$ ($Z_i(H)$ being as in Lemma \ref{LIEIDENT2}). From Lemma \ref{LIEIDENT2} we have
\begin{equation*}
\sum_{i=0}^W a_i(H)d\pi(H^{i})\vecw=\sum_{j=1}^d d\pi(X_jU_j(H))\vecw\qquad\forall \vecw\in\scrH^{\infty}.
\end{equation*}
Note that $a_{W_H}(H)=0$ and $a_i(H)=0$ for $i>W_H$. Similarly to Section \ref{BURGERFORMULASEC}, for each $H\in\fa$ we let $\underline{\lambda}(H)$ denote the multi-set (of order $W_H$) of roots to the polynomial
\begin{equation*}
z^{W_H}+a_{W_H-1}(H)z^{W_H-1}+\ldots+a_1(H)z+a_0(H).
\end{equation*}
Analogously to \eqref{MAINODE}, we have 
\begin{equation}\label{IHODE}
\left( \prod_{\lambda\in\underline{\lambda}(H)}(\sfrac{d}{dt}-\lambda)\right)\scrI_{\vecv}^{\chi}(H,t)=\sum_{j=1}^d -e^{\beta_j(H)t}\scrI_{d\pi(U_j(H))\vecv}^{X_j\chi}(H,t).
\end{equation}
By Lemma \ref{LIEIDENT2}, the right-hand side of this equation satisfies a bound $\ll e^{\alpha_0(H)t}$ for $t\leq 0$, where
\begin{equation*}
\alpha_0(H)=\min\lbrace \alpha(H)\,:\,\alpha\in\Sigma_0^+(\fg,\fa)\,\,\mathrm{such\,\,that}\,\, \alpha(H)>0\rbrace.
\end{equation*} 
This bound will enable us to use Lemma \ref{gint lem} and Proposition \ref{BurgerProp1} uniformly on the sets $\overline{\fa^+_{\epsilon,1}}$. Letting $|\underline{\lambda}(H)|_{\infty}=\max_{\lambda\in\underline{\lambda}(H)}|\lambda|$, we have the following lemma, which gives a ``uniform'' version of \eqref{lambdaSob}:
\begin{lem}\label{UNIFORMLAMBDASOB}
There exists a function $\Psi:\fa\rightarrow \RR_{>0}$ such that for all $H\in \fa$, $\vecw\in\scrH^{\infty}$, $m\in\NN$,
\begin{equation*}
|\underline{\lambda}(H)|_{\infty}^{W_H}\|\vecw\|_{\scrS^m(\scrH)}\ll_m \Psi(H)\|\vecw\|_{\scrS^{2W+m}(\scrH)}. 
\end{equation*}
Moreover, $\Psi$ is continuous on each connected component of $\overline{\fa^+_{\epsilon}}$.
\end{lem}
\begin{proof}
Arguing as in the proof of Corollary \ref{fbd}, by Fujiwara's bound \cite{Fujiwara}, we have
\begin{align*}
|\underline{\lambda}(H)|_{\infty}^{W_H}\|\vecw\|_{\scrS^m(\scrH)}\ll& \max_{i=0,\ldots, W_H-1} |a_i(H)|^{W_H/(W_H-i)} \|\vecw\|_{\scrS^m(\scrH)}\\&\ll \max_{i=0,\ldots, W_H-1} \|\vecw\|_{\scrS^m(\scrH)}+ \left\|d\pi\left(Z_{i}(H)^{\lceil W_H/(W_H-i)\rceil} \right)\vecw\right\|_{\scrS^m(\scrH)}.
\end{align*}
Once again using the fact that the Harish-Chandra isomorphism preserves the canonical filtrations, we have $Z_i(H)\in \scrU^{W_H-i}(\fg_{\CC})$ (cf. \eqref{HPOLY2}), hence $Z_{i}(H)^{\lceil W_H/(W_H-i)\rceil}\in \scrU^{2W_H-i}(\fg_{\CC})\subset \scrU^{2W}(\fg_{\CC})$. The continuity (on each connected component of $\overline{\fa^+_{\epsilon}}$) of the maps $H\mapsto Z_i(H)$ completes the proof. 
\end{proof}
\subsection{Burger's Formula} The aim of this section is to make Proposition \ref{BurgerProp2} uniform on sets $\overline{\fa^+_{\epsilon,1}}$. We adapt the multi-index notation introduced prior to Proposition \ref{BurgerProp2} to the present situation: for any multi-index $\bj=(j_1,j_2,\ldots,j_l)$, we define $X_{\bj}=X_{j_1}X_{j_2}\ldots X_{j_l}$ and $U_{\bj}(H)=U_{j_1}(H)U_{j_2}(H)\ldots U_{j_l}(H)$, where $X_j$ and $U_j(H)$ are as in Lemma \ref{LIEIDENT2}. We may now state a version of Proposition \ref{BurgerProp2} that is uniform over the sets $\overline{\fa^+_{\epsilon,1}}$:
\begin{prop}\label{BurgerProp3}
Given an irreducible unitary representation $(\pi,\scrH)$, a norm $\|\cdot\|$ on $\fa^+$ that controls the matrix coefficients of $\overline{A^+}$ in $(\pi,\scrH)$, and $\epsilon>0$, there exist functions $\scrC_{\bk}:\overline{\fa^+_{\epsilon,1}}\times\RR_{\leq0}\times \RR_{\leq 0}\rightarrow\CC$ ($\,\bk\in\bI_{\lfloor \frac{2}{\epsilon}\rfloor}$) and $\scrD_{\bj,i}:\overline{\fa^+_{\epsilon,1}}\times\RR_{\leq0}\rightarrow\CC$ ($\,\bj\in \bI_{\leq \lfloor \frac{2}{\epsilon}\rfloor-1}$ and $i\in\lbrace 0,\ldots, W\rbrace $) such that for all $\vecv\in\scrH^{\infty}$, $\chi\in C_c^{\infty}(N)$, $H\in\overline{\fa^+_{\epsilon,1}}$, and $t\leq0$:
\begin{equation*}
\scrI_{\vecv}^{\chi}(H,t)\!=\!\!\sum_{\bk\in\bI_{\lfloor \frac{2}{\epsilon}\rfloor}}\!\!\int_{-\infty}^0\!\! \scrC_{\bk}(H,t,s)\scrI_{d\pi(U_{\bk}(H))\vecv}^{X_{\bk}\chi}(H,s)\,ds+\!\!\!\!\sum_{\bj\in\bI_{\leq \lfloor \frac{2}{\epsilon}\rfloor-1}}\sum_{i=0}^{W} \scrD_{\bj,i}(H,t)\scrI_{d\pi(H^{i}U_{\bj}(H))\vecv}^{X_{\bj}\chi}(H,0).
\end{equation*}
Furthermore, the following hold:\vspace{2pt}
\begin{enumerate}
\item $|\scrC_{\bk}(H,t,s)|\ll(1+|t|^{W})e^{ t+\frac{\epsilon}{5}s}$, where the implied constant depends only on $G$, $\|\cdot\|$, and $\epsilon$.
\item For any $m\in\NN$,
\begin{equation*}
|\scrD_{\bj,i}(H,t)|\ll(1+|t|^{W}) e^{ t}\inf_{\vecw\in\scrH^{\infty}\setminus\lbrace \mathbf{0} \rbrace} \frac{\|\vecw\|_{\scrS^{2W+m}(\scrH)}}{\|\vecw\|_{\scrS^{m}(\scrH)}},
\end{equation*}
where the implied constant depends only on $G$, $m$, $\|\cdot\|$, and $\epsilon$.
\end{enumerate}
\end{prop}
\begin{proof}
For all $H\in\overline{\fa^+_{\epsilon,1}}$, we have $\alpha_0(H)\geq\epsilon$. By letting $\alpha=\epsilon$ and $\eta=1$, the construction in the proof of Proposition \ref{BurgerProp2} works for all $H\in\overline{\fa^+_{\epsilon,1}}$. Note that we have assumed that $\epsilon<1$ (cf. Remark \ref{EPSremark}), so $\alpha<\eta$; it is therefore this case that is used from Proposition \ref{BurgerProp2}. Also, even though $N$ is \emph{not} the expanding horospherical subgroup associated to those $H$ that lie on one or more of the walls of $\fa^+$, the fact that $U_j(H)=0$ for all $j$ such that $\beta_j(H)=0$ permits the use of the same proof for such $H$, since we have the bound
\begin{equation*}
\left\|-\sum_{j=1}^d e^{\beta_j(H)t}\scrI_{d\pi(U_{\bk}(H))\vecv}^{X_{\bk}\chi}(H,t)\right\| \ll \|\chi\|_{W^{|\bk|,1}(N)}\|\vecv\|_{\scrS^{W|\bk|}(\scrH)} e^{\alpha t}\qquad \forall H\in\overline{\fa^+_{\epsilon,1}},\,t\leq0.
\end{equation*}
Using Lemma \ref{UNIFORMLAMBDASOB} instead of \eqref{lambdaSob} throughout the proof replaces the bound in Proposition \ref{BurgerProp2} \textit{(2)} with the uniform bound stated here (i.e. property \textit{(2)}); in particular, the dependency of the implied constant on $g_{\RR}$ in Proposition \ref{BurgerProp2} \textit{(2)} is removed. 
\end{proof}
\subsection{The Invariant Height Function}
The final result needed before proceeding to the proof of Theorem \ref{maintheorem2} is the following generalization of Corollary \ref{scrYAVG}:
\begin{lem}\label{scrYavg2}
Let $B\subset N$ be a relatively compact subset of positive measure. Then 
\begin{equation*}
\int_B \scrY_{\Gamma}\big(xn\exp(-H)\big)\,dn\ll_{\Gamma,B} \scrY^2_{\Gamma}(x)\qquad \forall x\in\GaG,\,\,H\in\overline{\fa^+}.
\end{equation*}
\end{lem} 
\begin{proof}
The proof is essentially the same as that of  Corollary \ref{scrYAVG}; given a compact neighbourhood $B_0$ of $e$ in $AM N^-$, we let $\widetilde{B_0}$ be the closure of $\bigcup_{H\in\overline{\fa^+}} \exp(H) B_0\exp(-H)$. The set $\widetilde{B_0}$ is compact, allowing us to use the same argument as previously to bound the integral.
\end{proof}
\subsection{Proof of Theorem \ref{maintheorem2}} 
Given $H\in\overline{\fa^+}$, we use Lemma \ref{EPSILONH} to find $H_{\epsilon}'$ and $J_{\epsilon}'$ such that $H=H_{\epsilon}'+J_{\epsilon}'$, with $H_{\epsilon}'\in\overline{\fa^+_{\epsilon}}$, $\|H_{\epsilon}'\|\geq (1-c\epsilon)\|H\|$, and $J_{\epsilon}'\in\overline{\fa^+}$. Observe that if $\frac{\alpha_0(H)}{\|H\|}\geq\epsilon$, then $H_{\epsilon}'=H$. We now let $H_{\epsilon}=\frac{H_{\epsilon}'}{\|H_{\epsilon}'\|}\in\overline{\fa^+_{\epsilon,1}}$ and $J_{\epsilon}=\frac{J_{\epsilon}'}{\|H_{\epsilon}'\|}\in \overline{\fa^+}$. As in the proof of Theorem \ref{maintheorem1}, we first make the reduction to the case when $f\in V^{\infty}$ and $\chi\in C_c^{\infty}(B)$. We then have that for $t\leq 0$,
\begin{align*}
\scrI_f^{\chi}(H_{\epsilon}+J_{\epsilon},t)=&\scrI_f^{\chi}(H_{\epsilon},t)+\int_0^1 \frac{d}{d{\tau}} \scrI_f^{\chi}(H_{\epsilon}+\tau J_{\epsilon},t)\,d\tau\\&=\scrI_f^{\chi}(H_{\epsilon},t)-\int_t^0  \int_N \chi(n)\frac{d}{d\tau}\rho\big(n\exp(tH_{\epsilon}+\tau J_{\epsilon})\big)f\,dn\,d\tau\\&=\scrI_f^{\chi}(H_{\epsilon},t)-\int_t^0  \int_N \chi(n)\rho\big(n\exp(tH_{\epsilon}+\tau J_{\epsilon})\big)d\rho(J_{\epsilon})f\,dn\,d\tau.
\end{align*}
The $\tau$-integral is rewritten in the following manner:
\begin{align*}
\int_t^0  \int_N& \chi(n)\rho\big(n\exp(tH_{\epsilon}+\tau J_{\epsilon})\big)d\rho(J_{\epsilon})f\,dn\,d\tau\\&=\int_t^0 \rho(\exp(\tau J_{\epsilon}))  \int_N \chi(n)\rho\big(\exp(-\tau J_{\epsilon})n\exp(\tau J_{\epsilon}) \big)\rho(\exp(tH_{\epsilon}))d\rho(J_{\epsilon})f\,dn\,d\tau\\&=\int_t^0 e^{2\tau\rho_{\fa}(J_{\epsilon})} \rho(\exp(\tau J_{\epsilon}))  \int_N \chi\big(\exp(\tau J_{\epsilon})n\exp(-\tau J_{\epsilon})\big)\rho\big(n\exp(tH_{\epsilon})\big)d\rho(J_{\epsilon})f\,dn\,d\tau\\& =\int_t^0 e^{2\tau\rho_{\fa}(J_{\epsilon})} \rho(\exp(\tau J_{\epsilon}))\scrI_{d\rho(J_{\epsilon})f}^{\chi_{\tau}}(H_{\epsilon},t)\,d\tau,
\end{align*} 
where, for any $\psi\in C_c^{\infty}(N)$,
\begin{equation*}
\psi_{\tau}(n):=\psi\big(\exp(\tau J_{\epsilon})n\exp(-\tau J_{\epsilon})\big).
\end{equation*}
This is now entered into the original equation, giving
\begin{equation}\label{MAINHEPSIDENT}
\scrI_f^{\chi}(H_{\epsilon}+J_{\epsilon},t)=\scrI_f^{\chi}(H_{\epsilon},t)-\int_t^0 e^{2\tau\rho_{\fa}(J_{\epsilon})} \rho(\exp(\tau J_{\epsilon}))\scrI_{d\rho(J_{\epsilon})f}^{\chi_{\tau}}(H_{\epsilon},t)\,d\tau.
\end{equation}
We will now use Proposition \ref{BurgerProp3} in the same way that Proposition \ref{BurgerProp2} was used in the proof of Theorem \ref{maintheorem1}: let $\lbrace T_{\bk}(H_{\epsilon},t,s)\rbrace$, $\lbrace S_{\bj,i}(H_{\epsilon},t)\rbrace$ be the intertwining operators defined on $V^{\infty}$ by (for $\phi\in V^{\infty}$)
\begin{align*}
T_{\bk}(H_{\epsilon},t,s)\phi&=\int_{\mathsf{Z}} \scrC_{\zeta,\bk}(H_{\epsilon},t,s) \phi_{\zeta}\,d\upsilon(\zeta)\\ S_{\bj,i}(H_{\epsilon},t)\phi&=\int_{\mathsf{Z}}  \scrD_{\zeta,\bj,i}(H_{\epsilon},t)_{\zeta}\phi_{\zeta}\,d\upsilon(\zeta).
\end{align*}
Since $H_{\epsilon}$ will be fixed for the remainder of the proof, for notational convenience we temporarily suppress the argument in the functions $U_{\bk}$; that is $U_{\bk}=U_{\bk}(H_{\epsilon})$. Analogously to the proof of Theorem \ref{maintheorem1}, by Proposition \ref{BurgerProp3}, these intertwining operators may be used to write \eqref{MAINHEPSIDENT} as
\begin{align}\label{mainUNIIDENT}
\scrI_f^{\chi}(H_{\epsilon}+J_{\epsilon},t)\!=\!\!&\sum_{\bk\in\bI_{m_2}}\!\!\int_{-\infty}^0\! T_{\bk}(H_{\epsilon},t,s)\scrI_{d\rho(U_{\bk})f}^{X_{\bk}\chi}(H_{\epsilon},s)\,ds+\!\!\!\sum_{\bj\in\bI_{\leq m_2 -1}}\sum_{i=0}^{W} S_{\bj,i}(H_{\epsilon},t)\scrI_{d\rho(H_{\epsilon}^{i}U_{\bj})f}^{X_{\bj}\chi}(H_{\epsilon},0)\\\notag&- \sum_{\bk\in\bI_{m_2}}\int_t^0 \int_{-\infty}^0 e^{2\tau\rho_{\fa}(J_{\epsilon})} \rho(\exp(\tau J_{\epsilon}))T_{\bk}(H_{\epsilon},t,s)\scrI_{d\rho(U_{\bk}J_{\epsilon})f}^{X_{\bk}\chi_{\tau }}(H_{\epsilon},s)\,ds\,d\tau
\\\notag&\quad- \sum_{\bj\in\bI_{\leq m_2 -1}}\sum_{i=0}^{W}\int_t^0 e^{2\tau\rho_{\fa}(J_{\epsilon})} \rho(\exp(\tau J_{\epsilon})) S_{\bj,i}(H_{\epsilon},t)\scrI_{d\rho(H_{\epsilon}^{i}U_{\bj}J_{\epsilon})f}^{X_{\bj}\chi_{\tau}}(H_{\epsilon},0)\,d\tau,
\end{align}
where $m_2=\lfloor \frac{2}{\epsilon}\rfloor$. For any element $X_j$ of our chosen root basis of $\fn$, we have
\begin{align*}
(X_{j}\chi_{\tau})(n)=&\left.\frac{d}{dv}\right|_{v=0} \chi_{\tau}(n\exp(v X_{\mu}))=\left.\frac{d}{dv}\right|_{v=0} \chi\big(\exp(\tau J_{\epsilon})n\exp(v X_j)\exp(-\tau J_{\epsilon})\big)\\&=\left.\frac{d}{dv}\right|_{v=0} \chi\big(\exp(\tau J_{\epsilon})n\exp(-\tau J_{\epsilon})\exp(v\Ad_{\exp(\tau J_{\epsilon})}X_j)\big)\\&\quad= e^{\tau\beta_j(J_{\epsilon})}(X_j\chi)\big(\exp(\tau J_{\epsilon})n\exp(-\tau J_{\epsilon})\big)=e^{\tau\beta_j(J_{\epsilon})}(X_j\chi)_{\tau}(n).
\end{align*}
Hence, for any multi-index $\bk=(k_1,\ldots,k_l)$,
\begin{align*}
(X_{\bk}\chi_{\tau})(n)=e^{\tau(\beta_{k_1}+\ldots+\beta_{k_l})(J_{\epsilon})}  (X_{\bk}\chi)_{\tau}(n).
\end{align*}
Since $J_{\epsilon}\in\overline{\fa^+}$, $\beta_j(J_{\epsilon})\geq 0$ for all $j=1,\ldots,d$, giving $e^{\tau\beta_j(J_{\epsilon})}\leq 1$ ($\tau\leq 0$), and so 
\begin{equation*}
|(X_j\chi_{\tau})(n)|= \left|e^{\tau\beta_j(J_{\epsilon})}(X_j\chi)_{\tau}(n)\right|=\left|e^{\tau\beta_j(J_{\epsilon})}(X_j\chi)\big(\exp(\tau J_{\epsilon})n\exp(-\tau J_{\epsilon})\big)\right|\leq \|X_j\chi\|_{L^{\infty}(N)}.
\end{equation*}
Hence, for any multi-index $\bk$ and $\tau\leq 0$,
\begin{equation}\label{ADJbnd}
 \|X_{\bk}\chi_{\tau}\|_{L^{\infty}(N)}= \|e^{\tau(\beta_{k_1}+\ldots+\beta_{k_l})(J_{\epsilon})}  (X_{\bk}\chi)_{\tau}\|_{L^{\infty}(N)}\leq\|(X_{\bk}\chi)_{\tau}\|_{L^{\infty}(N)}=\|X_{\bk}\chi\|_{L^{\infty}(N)}.
\end{equation}
We now consider the terms occurring in the last two lines of \eqref{mainUNIIDENT}. To start with, for $\bk=(k_1,\dots k_{m_2})\in\bI_{m_2}$, we have
\begin{align*}
&\int_t^0 \int_{-\infty}^0 e^{2\tau\rho_{\fa}(J_{\epsilon})} \rho(\exp(\tau J_{\epsilon}))T_{\bk}(H_{\epsilon},t,s)\scrI_{d\rho(U_{\bk}J_{\epsilon})f}^{X_{\bk}\chi_{\tau }}(H_{\epsilon},s)\,ds\,d\tau\\&=\int_t^0 \int_{-\infty}^0\int_N e^{2\tau\rho_{\fa}(J_{\epsilon})} e^{\tau(\beta_{k_1}+\ldots+\beta_{k_{m_2}})(J_{\epsilon})}\big(X_{\bk}\chi\big)_{\tau}(n)\\&\qquad\qquad\qquad\qquad\qquad\qquad \times\rho\big(\exp(\tau J_{\epsilon})n\exp(s H_{\epsilon})\big)T_{\bk}(H_{\epsilon},t,s)d\rho(U_{\bk}J_{\epsilon})f\,dn\,ds\,d\tau
\\&=\int_t^0 \int_{-\infty}^0\int_N e^{\tau(\beta_{k_1}+\ldots+\beta_{k_{m_2}})(J_{\epsilon})}(X_{\bk}\chi)(n)\rho\big(n\exp(\tau J_{\epsilon}+s H_{\epsilon})\big)T_{\bk}(H_{\epsilon},t,s)d\rho(U_{\bk}J_{\epsilon})f\,dn\,ds\,d\tau.
\end{align*}
Taking into account the fact that $\supp(\chi)\subset B$, the bound given in Proposition \ref{BurgerProp3} \textit{(1)}, Proposition \ref{sobbnd}, and Lemma \ref{scrYavg2} (note that $-\tau J_{\epsilon}-s H_{\epsilon}\in\overline{\fa^+}$ for all $\tau,s\leq 0$), we have the following bound for $m>\frac{\dim G}{2}$ and all $x\in\GaG$:
\begin{align*}
&\int_N \left| e^{\tau(\beta_{k_1}+\ldots+\beta_{k_{m_2}})(J_{\epsilon})}(X_{\bk}\chi)(n)T_{\bk}(H_{\epsilon},t,s)d\rho(U_{\bk}(H_{\epsilon})J_{\epsilon})f\big(x n\exp(\tau J_{\epsilon}+s H_{\epsilon})\big)\right|\,dn\\&\qquad\ll e^{\tau(\beta_{k_1}+\ldots+\beta_{k_{m_2}})(J_{\epsilon})} \| X_{\bk}\chi\|_{L^{\infty}(B)}\|d\rho(U_{\bk}(H_{\epsilon})J_{\epsilon})f\|_{\scrS^m(\GaG)}\scrY_{\Gamma}(x)^2(1+|t|^{W})e^{ t+\frac{\epsilon s}{5}}.
\end{align*}
Since $\|J_{\epsilon}\|< 1$ (cf. Remark \ref{EPSremark}), $\|d\rho(J_{\epsilon})\phi\|_{L^2(\GaG)}\ll \|\phi\|_{\scrS^1(\GaG)}$ for all $\phi\in V^{\infty}$. This, together with \eqref{ADJbnd} and Lemma \ref{LIEIDENT2} (in particular, the continuity of $Y\mapsto U_{j}(Y)$ on each connected component of $\overline{\fa^+_{\epsilon,1}}$) gives
\begin{align*}
\int_t^0\!\!\! & \int_{-\infty}^0\!\int_N \! \left| e^{\tau(\beta_{k_1}+\ldots+\beta_{k_{m_2}})(J_{\epsilon})}\!(X_{\bk}\chi)(n)T_{\bk}(H_{\epsilon},t,s)d\rho(U_{\bk}(H_{\epsilon})J_{\epsilon})f\big(x n\exp(\tau J_{\epsilon}+s H_{\epsilon})\big)\!\right|\!dn\,ds\,d\tau\\&\qquad\ll \| \chi\|_{W^{m_2,\infty}(B)}\|f\|_{\scrS^{W m_2+1+m}(\GaG)}\scrY_{\Gamma}(x)^2(1+|t|^{W+1})e^{ t},
\end{align*}
Similar arguments are used for the other terms of \eqref{mainUNIIDENT} that have an integral term $\int_{t}^0\,d\tau$, i.e. for $\bj=(j_1,\ldots,j_l)\in\bI_{\leq m_2-1}$ and $i\in \lbrace 0, \ldots, W\rbrace$,
\begin{align*}
&\left|\int_t^0 e^{2\tau\rho_{\fa}(J_{\epsilon})} \left(\rho(\exp(\tau J_{\epsilon})) S_{\bj,i}(H_{\epsilon},t)\scrI_{d\rho(H_{\epsilon}^{i}U_{\bj}J_{\epsilon})f}^{X_{\bj}\chi_{\tau}}(H_{\epsilon},0)\right)(x)\,d\tau\right|\\&\quad
\leq\int_t^0 \int_N e^{\tau(\beta_{j_1}+\ldots+\beta_{j_l})(J_{\epsilon})}\left|(X_{\bj})\chi)(n)S_{\bj,i}(H_{\epsilon},t)d\rho(H_{\epsilon}^{i}U_{\bj}(H_{\epsilon})J_{\epsilon})f\big(x n\exp(\tau J_{\epsilon})\big)\right|\,dn \,d\tau\\&\qquad\qquad\ll \| \chi\|_{W^{m_2,\infty}(B)}\|f\|_{\scrS^{W(m_2+3)+1+m}(\GaG)}\scrY_{\Gamma}(x)^2(1+|t|^{W+1})e^{ t},
\end{align*}
where Proposition \ref{BurgerProp3} \textit{(2)} was now used. The remaining terms in \eqref{mainUNIIDENT} are dealt with in same way as in the proof of Theorem \ref{maintheorem1}. Putting these bounds together, we have
\begin{align*}
&\left| \int_B \chi(n)f\big(xn\exp(t(H_{\epsilon}+J_{\epsilon})\big)\,dn\right|\ll \| \chi\|_{W^{m_2,\infty}(B)}\|f\|_{\scrS^{m_1}(\GaG)}\scrY_{\Gamma}(x)^2(1+|t|^{W+1})e^{t}.
\end{align*}
Letting $t=-\|H_{\epsilon}'\|$ (recall that $H=\|H_{\epsilon}'\|(H_{\epsilon}+J_{\epsilon})$ and $ (1-c\epsilon)\|H\|\leq \|H_{\epsilon}'\|\leq (1+c\epsilon)\|H\|$) now completes the proof.

\hspace{425.5pt}\qedsymbol

\end{document}